\documentclass[a4paper,11pt, reqno]{amsart}
\usepackage{a4wide}
\usepackage{amssymb}
\usepackage{enumerate}
\usepackage[all,cmtip]{xy}
\title{Bijections for Entringer families}
\author{Yoann Gelineau}
\address[Yoann Gelineau]{Universit\'{e} de Lyon; Universit\'{e} Lyon 1; Institut Camille Jordan; UMR 5208 du CNRS; 43, boulevard du 11 novembre 1918, F-69622 Villeurbanne Cedex, France}
\email{gelineau@math.univ-lyon1.fr}

\author{Heesung Shin}
\address[Heesung Shin]{Universit\'{e} de Lyon; Universit\'{e} Lyon 1; Institut Camille Jordan; UMR 5208 du CNRS; 43, boulevard du 11 novembre 1918, F-69622 Villeurbanne Cedex, France}
\email{hshin@math.univ-lyon1.fr}

\author{Jiang Zeng}
\address[Jiang Zeng]{Universit\'{e} de Lyon; Universit\'{e} Lyon 1; Institut Camille Jordan; UMR 5208 du CNRS; 43, boulevard du 11 novembre 1918, F-69622 Villeurbanne Cedex, France}
\email{gelineau@math.univ-lyon1.fr}

\date{\today}

\usepackage{color}
\usepackage{ifpdf}
\ifpdf 
  \usepackage[pdftex]{graphicx}
  \DeclareGraphicsExtensions{.pdf,.png,.mps}
  \usepackage{pgf}
  \usepackage{tikz}
\else 
  \usepackage{graphicx}
  \DeclareGraphicsExtensions{.eps,.bmp}
  \usepackage{pgf}
  \usepackage{tikz}
  \usepackage{pstricks}
\fi
\usepackage{epic}
\usepackage{wrapfig}

\newtheorem{thm}{Theorem}[section]
\newtheorem{cor}[thm]{Corollary}

\newtheorem{prop}[thm]{Proposition}
\theoremstyle{definition}
\newtheorem{defn}[thm]{Definition}
\newtheorem{rmk}[thm]{Remark}
\newtheorem{ex}[thm]{Example}

\DeclareMathOperator\inv{inv}
\DeclareMathOperator\threeonetwo{31-2}

\newcommand\set[1]{\left\{#1\right\}}
\newcommand\abs[1]{\left|#1\right|}
\def\ESnk{\mathcal{ES}_{n,k}}
\def\ESn{\mathcal{ES}_{n}}
\def\A{\mathcal{DU}}
\def\An{\mathcal{DU}_n}
\def\Dn{\mathcal{ES}_n}
\def\Ank{\mathcal{DU}_{n,k}}
\def\ank{E_{n,k}}
\def\T{\mathcal{T}}
\def\Tn{\mathcal{BT}_n}
\def\Tnk{\mathcal{BT}_{n,k}}
\def\312{\threeonetwo}
\def\rev{\text{rev}}

\renewcommand{\qedsymbol}{$\blacksquare$}

\begin{document}
\maketitle

\begin{abstract}
Andr\'e proved that the number of down-up permutations on $\{1, 2, ..., n\}$  is equal to the Euler number $E_n$.
A refinement of Andr\'e's result was given by Entringer, who proved that counting  down-up permutations  according to the first element gives rise to Seidel's
triangle  $(E_{n,k})$ for  computing the Euler numbers.
In a series of papers, using generating function method and induction, Poupard gave several further  combinatorial interpretations for $E_{n,k}$  both in down-up permutations and increasing trees.
Kuznetsov, Pak, and Postnikov have given more combinatorial interpretations of $E_{n,k}$ in the model of trees.
The aim of this paper is to provide bijections between the different models for $E_{n,k}$ as well as some new interpretations.
 In particular, we give  the first  explicit one-to-one correspondence between Entringer's down-up permutation model and Poupard's increasing tree model.
\end{abstract}

\tableofcontents

\section{Introduction}

The \emph{Euler numbers} $E_{n}$ are defined by the generating function\begin{align*}
 \sum\limits_{n \geq 0} E_n \frac{x^n}{n!} & = \tan(x)+\sec(x) \\
 & = 1+x+\frac{x^2}{2!} + 2 \frac{x^3}{3!} + 5 \frac{x^4}{4!} + 16 \frac{x^5}{5!} + 61 \frac{x^6}{6!} + 272 \frac{x^7}{7!} + 1385 \frac{x^8}{8!} + \cdots. \end{align*}
Let $\An$ be the set of \emph{down-up  permutations of $[n]:=\set{1,2,\dots,n}$}, that is, the permutations $\pi=\pi_1\pi_2\dots\pi_n$ on $[n]$ satisfying
$\pi_1 > \pi_2 < \pi_3 > \pi_4 < \cdots$. For example, the down-up permutations of $[4]$ are:
\[ 
2\,1\,4\,3, \qquad 3\,2\,4\,1, \qquad 3\,1\,4\,2, \qquad 4\,2\,3\,1, \qquad 4\,1\,3\,2.
 \]
Andr\'{e}~\cite{And79} proved that  the cardinality of the set $\An$ equals the Euler number $E_n$. Counting  the down-up permutations according to the first term leads to 
the  \emph{Entringer numbers}~\cite{Ent66}. More precisely, let  $\Ank$ be the set of permutations $\pi \in \An$ such that $\pi_1=k$ and $E_{n,k}$
  the  cardinality of $\Ank$. The first values of $E_{n,k}$ are given in Table~\ref{tab1}.

\begin{table}[!h]
\begin{tabular}{c|cccccccc}
$n \setminus k$ & $1$ & $2$ & $3$ & $4$ & $5$ & $6$ & $7$ & \\
\hline
$1$ & $1$ &     &  & & & &\\
$2$ &   0  & $1$ &     & & & &\\
$3$ &    0 & $1$ & $1$ &     &     & &\\
$4$ &    0 & $1$ & $2$ & $2$ &     & &\\
$5$ &    0 & $2$ & $4$ & $5$ & $5$ &  & \\
$6$ &   0  & $5$ & $10$ & $14$ & $16$ & $16$ & \\
$7$ &   0  & $16$ & $32$ & $46$ & $56$ & $61$  & $61$ \\
\multicolumn{2}{c}{\quad } \\
\end{tabular}
\quad
\caption{\label{tab1} The first values of Entringer numbers $\ank$}
\end{table}

 \begin{thm}[Entringer] \label{Ent1} The numbers
$(E_{n,k})$ ( $n \geq k\geq 1$) are defined  by
\begin{equation} \label{eq1}
E_{1,1}=1,\quad E_{n,1}=0 \;(n\geq 2),\qquad E_{n,k} = E_{n,k-1} + E_{n-1,n+1-k}.
\end{equation}
\end{thm}

Iterating the above recurrence, we get  $E_{n+1,n+1}=E_{n,n}+E_{n,n-1}+\cdots +E_{n,1}$,
which is equal to $E_n$
by Andr\'e's result. Hence the Euler numbers $E_n=E_{n+1,n+1}$   are the diagonal entries in Table~\ref{tab1}.
As an historical remark, Entringer's recurrence \eqref{eq1} is just a combinatorial interpretation of the Seidel's scheme \cite{Sei77} to compute Euler numbers, i.e., 
 \[\footnotesize
\begin{tabular}{ccccccccc}
& &          &&  $E_{1,1}$ \\
&& & $E_{2,1}$ & $\rightarrow$ & $E_{2,2}$ \\
&& $E_{3,3}$ & $\leftarrow$ & $E_{3,2}$ & $\leftarrow$ & $E_{3,1}$ \\
&$E_{1,1}$ & $\rightarrow$ & $E_{4,2}$ & $\rightarrow$ & $E_{4,3}$ & $\rightarrow$ & $E_{4,4}$& \\
$E_{5,5}$&$\leftarrow$ &$E_{5,4}$ & $\leftarrow$ & $E_{5,3}$ & $\leftarrow$ & $E_{5,2}$ & $\leftarrow$ & $E_{5,1}$ \\
& & & $\cdots$ \end{tabular}
\Longleftrightarrow \begin{tabular}{ccccccccc}
&&&          &  1 \\
&& & 0 & $\rightarrow$ & 1 \\
&& 1& $\leftarrow$ & 1 & $\leftarrow$ &0  \\
&0 & $\rightarrow$ & 1& $\rightarrow$ & 2& $\rightarrow$ & 2 \\
5&$\leftarrow$ &5& $\leftarrow$ & 4 & $\leftarrow$ & 2 & $\leftarrow$ & 0 \\
& & & $\cdots$ \end{tabular}
 \]
The above  scheme  was  later rediscovered several times in the literature (see \cite{Kem33, MSY96}).
 A recent survey on down-up permutations and  Euler numbers is given by
 Stanley~\cite{Sta09}.

A sequence of sets $(X_{n,k})_{1 \leq k \leq n}$ is called an \emph{Entringer family} if the cardinality of $X_{n,k}$ is equal to $E_{n,k}$ for $1 \leq k \leq n$.

Let $X=\{x_1,\ldots,x_n\}_{<}$ be an ordered set such that  $x_1 < \cdots < x_n$.
An \emph{increasing tree} on $X$ is a spanning tree of the complete graph on $X$, rooted at $x_1$ and oriented from the smallest
vertex $x_1$, such that the vertices increase along the edges. Let $\Tn$ be the set of {\em binary  increasing trees} $T$ on $[n]$, i.e., the increasing trees such that at most two edges go out from every vertex (see Figure~\ref{fig012}).

\begin{figure}[!h]
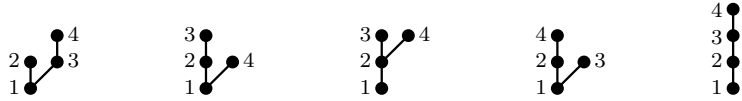

\centering
\begin{pgfpicture}{2.20mm}{3.70mm}{15.30mm}{17.50mm}
\pgfsetxvec{\pgfpoint{0.70mm}{0mm}}
\pgfsetyvec{\pgfpoint{0mm}{0.70mm}}
\color[rgb]{0,0,0}\pgfsetlinewidth{0.30mm}\pgfsetdash{}{0mm}
\pgfcircle[fill]{\pgfxy(10.00,10.00)}{0.70mm}
\pgfcircle[stroke]{\pgfxy(10.00,10.00)}{0.70mm}
\pgfcircle[fill]{\pgfxy(10.00,15.00)}{0.70mm}
\pgfcircle[stroke]{\pgfxy(10.00,15.00)}{0.70mm}
\pgfcircle[fill]{\pgfxy(15.00,15.00)}{0.70mm}
\pgfcircle[stroke]{\pgfxy(15.00,15.00)}{0.70mm}
\pgfcircle[fill]{\pgfxy(15.00,20.00)}{0.70mm}
\pgfcircle[stroke]{\pgfxy(15.00,20.00)}{0.70mm}
\pgfmoveto{\pgfxy(10.00,15.00)}\pgflineto{\pgfxy(10.00,10.00)}\pgfstroke
\pgfmoveto{\pgfxy(10.00,10.00)}\pgflineto{\pgfxy(15.00,15.00)}\pgfstroke
\pgfmoveto{\pgfxy(15.00,20.00)}\pgflineto{\pgfxy(15.00,15.00)}\pgfstroke
\pgfputat{\pgfxy(8.00,14.00)}{\pgfbox[bottom,left]{\fontsize{7.97}{9.56}\selectfont \makebox[0pt][r]{2}}}
\pgfputat{\pgfxy(8.00,9.00)}{\pgfbox[bottom,left]{\fontsize{7.97}{9.56}\selectfont \makebox[0pt][r]{1}}}
\pgfputat{\pgfxy(17.00,14.00)}{\pgfbox[bottom,left]{\fontsize{7.97}{9.56}\selectfont 3}}
\pgfputat{\pgfxy(17.00,19.00)}{\pgfbox[bottom,left]{\fontsize{7.97}{9.56}\selectfont 4}}
\end{pgfpicture}\hspace{1cm}%
\centering
\begin{pgfpicture}{2.20mm}{3.70mm}{15.30mm}{17.50mm}
\pgfsetxvec{\pgfpoint{0.70mm}{0mm}}
\pgfsetyvec{\pgfpoint{0mm}{0.70mm}}
\color[rgb]{0,0,0}\pgfsetlinewidth{0.30mm}\pgfsetdash{}{0mm}
\pgfcircle[fill]{\pgfxy(10.00,10.00)}{0.70mm}
\pgfcircle[stroke]{\pgfxy(10.00,10.00)}{0.70mm}
\pgfcircle[fill]{\pgfxy(10.00,15.00)}{0.70mm}
\pgfcircle[stroke]{\pgfxy(10.00,15.00)}{0.70mm}
\pgfcircle[fill]{\pgfxy(10.00,20.00)}{0.70mm}
\pgfcircle[stroke]{\pgfxy(10.00,20.00)}{0.70mm}
\pgfcircle[fill]{\pgfxy(15.00,15.00)}{0.70mm}
\pgfcircle[stroke]{\pgfxy(15.00,15.00)}{0.70mm}
\pgfmoveto{\pgfxy(10.00,15.00)}\pgflineto{\pgfxy(10.00,10.00)}\pgfstroke
\pgfmoveto{\pgfxy(10.00,10.00)}\pgflineto{\pgfxy(15.00,15.00)}\pgfstroke
\pgfmoveto{\pgfxy(10.00,20.00)}\pgflineto{\pgfxy(10.00,15.00)}\pgfstroke
\pgfputat{\pgfxy(8.00,14.00)}{\pgfbox[bottom,left]{\fontsize{7.97}{9.56}\selectfont \makebox[0pt][r]{2}}}
\pgfputat{\pgfxy(8.00,9.00)}{\pgfbox[bottom,left]{\fontsize{7.97}{9.56}\selectfont \makebox[0pt][r]{1}}}
\pgfputat{\pgfxy(8.00,19.00)}{\pgfbox[bottom,left]{\fontsize{7.97}{9.56}\selectfont \makebox[0pt][r]{3}}}
\pgfputat{\pgfxy(17.00,14.00)}{\pgfbox[bottom,left]{\fontsize{7.97}{9.56}\selectfont 4}}
\end{pgfpicture}\hspace{1cm}%
\centering
\begin{pgfpicture}{2.20mm}{3.70mm}{15.30mm}{17.50mm}
\pgfsetxvec{\pgfpoint{0.70mm}{0mm}}
\pgfsetyvec{\pgfpoint{0mm}{0.70mm}}
\color[rgb]{0,0,0}\pgfsetlinewidth{0.30mm}\pgfsetdash{}{0mm}
\pgfcircle[fill]{\pgfxy(10.00,10.00)}{0.70mm}
\pgfcircle[stroke]{\pgfxy(10.00,10.00)}{0.70mm}
\pgfcircle[fill]{\pgfxy(10.00,15.00)}{0.70mm}
\pgfcircle[stroke]{\pgfxy(10.00,15.00)}{0.70mm}
\pgfcircle[fill]{\pgfxy(10.00,20.00)}{0.70mm}
\pgfcircle[stroke]{\pgfxy(10.00,20.00)}{0.70mm}
\pgfcircle[fill]{\pgfxy(15.00,20.00)}{0.70mm}
\pgfcircle[stroke]{\pgfxy(15.00,20.00)}{0.70mm}
\pgfmoveto{\pgfxy(10.00,15.00)}\pgflineto{\pgfxy(10.00,10.00)}\pgfstroke
\pgfmoveto{\pgfxy(10.00,15.00)}\pgflineto{\pgfxy(15.00,20.00)}\pgfstroke
\pgfmoveto{\pgfxy(10.00,20.00)}\pgflineto{\pgfxy(10.00,15.00)}\pgfstroke
\pgfputat{\pgfxy(8.00,14.00)}{\pgfbox[bottom,left]{\fontsize{7.97}{9.56}\selectfont \makebox[0pt][r]{2}}}
\pgfputat{\pgfxy(8.00,9.00)}{\pgfbox[bottom,left]{\fontsize{7.97}{9.56}\selectfont \makebox[0pt][r]{1}}}
\pgfputat{\pgfxy(8.00,19.00)}{\pgfbox[bottom,left]{\fontsize{7.97}{9.56}\selectfont \makebox[0pt][r]{3}}}
\pgfputat{\pgfxy(17.00,19.00)}{\pgfbox[bottom,left]{\fontsize{7.97}{9.56}\selectfont 4}}
\end{pgfpicture}\hspace{1cm}%
\centering
\begin{pgfpicture}{16.20mm}{3.70mm}{29.30mm}{17.50mm}
\pgfsetxvec{\pgfpoint{0.70mm}{0mm}}
\pgfsetyvec{\pgfpoint{0mm}{0.70mm}}
\color[rgb]{0,0,0}\pgfsetlinewidth{0.30mm}\pgfsetdash{}{0mm}
\pgfcircle[fill]{\pgfxy(30.00,10.00)}{0.70mm}
\pgfcircle[stroke]{\pgfxy(30.00,10.00)}{0.70mm}
\pgfcircle[fill]{\pgfxy(30.00,15.00)}{0.70mm}
\pgfcircle[stroke]{\pgfxy(30.00,15.00)}{0.70mm}
\pgfcircle[fill]{\pgfxy(30.00,20.00)}{0.70mm}
\pgfcircle[stroke]{\pgfxy(30.00,20.00)}{0.70mm}
\pgfcircle[fill]{\pgfxy(35.00,15.00)}{0.70mm}
\pgfcircle[stroke]{\pgfxy(35.00,15.00)}{0.70mm}
\pgfmoveto{\pgfxy(30.00,15.00)}\pgflineto{\pgfxy(30.00,10.00)}\pgfstroke
\pgfmoveto{\pgfxy(30.00,10.00)}\pgflineto{\pgfxy(35.00,15.00)}\pgfstroke
\pgfmoveto{\pgfxy(30.00,15.00)}\pgflineto{\pgfxy(30.00,20.00)}\pgfstroke
\pgfputat{\pgfxy(28.00,14.00)}{\pgfbox[bottom,left]{\fontsize{7.97}{9.56}\selectfont \makebox[0pt][r]{2}}}
\pgfputat{\pgfxy(28.00,9.00)}{\pgfbox[bottom,left]{\fontsize{7.97}{9.56}\selectfont \makebox[0pt][r]{1}}}
\pgfputat{\pgfxy(28.00,19.00)}{\pgfbox[bottom,left]{\fontsize{7.97}{9.56}\selectfont \makebox[0pt][r]{4}}}
\pgfputat{\pgfxy(37.00,14.00)}{\pgfbox[bottom,left]{\fontsize{7.97}{9.56}\selectfont 3}}
\end{pgfpicture}\hspace{1cm}%
\centering
\begin{pgfpicture}{16.20mm}{3.70mm}{23.70mm}{20.65mm}
\pgfsetxvec{\pgfpoint{0.70mm}{0mm}}
\pgfsetyvec{\pgfpoint{0mm}{0.70mm}}
\color[rgb]{0,0,0}\pgfsetlinewidth{0.30mm}\pgfsetdash{}{0mm}
\pgfcircle[fill]{\pgfxy(30.00,10.00)}{0.70mm}
\pgfcircle[stroke]{\pgfxy(30.00,10.00)}{0.70mm}
\pgfcircle[fill]{\pgfxy(30.00,15.00)}{0.70mm}
\pgfcircle[stroke]{\pgfxy(30.00,15.00)}{0.70mm}
\pgfcircle[fill]{\pgfxy(30.00,20.00)}{0.70mm}
\pgfcircle[stroke]{\pgfxy(30.00,20.00)}{0.70mm}
\pgfcircle[fill]{\pgfxy(30.00,25.00)}{0.70mm}
\pgfcircle[stroke]{\pgfxy(30.00,25.00)}{0.70mm}
\pgfmoveto{\pgfxy(30.00,15.00)}\pgflineto{\pgfxy(30.00,10.00)}\pgfstroke
\pgfputat{\pgfxy(28.00,14.00)}{\pgfbox[bottom,left]{\fontsize{7.97}{9.56}\selectfont \makebox[0pt][r]{2}}}
\pgfputat{\pgfxy(28.00,9.00)}{\pgfbox[bottom,left]{\fontsize{7.97}{9.56}\selectfont \makebox[0pt][r]{1}}}
\pgfputat{\pgfxy(28.00,18.50)}{\pgfbox[bottom,left]{\fontsize{7.97}{9.56}\selectfont \makebox[0pt][r]{3}}}
\pgfputat{\pgfxy(28.00,23.50)}{\pgfbox[bottom,left]{\fontsize{7.97}{9.56}\selectfont \makebox[0pt][r]{4}}}
\pgfmoveto{\pgfxy(30.00,20.00)}\pgflineto{\pgfxy(30.00,15.00)}\pgfstroke
\pgfmoveto{\pgfxy(30.00,25.00)}\pgflineto{\pgfxy(30.00,20.00)}\pgfstroke
\end{pgfpicture}%
\caption{\label{fig012} The binary increasing  trees on $[4]$}
\end{figure}

Foata and Sch\"{u}tzenberger proved in  \cite[\S5]{FS73} that
the Euler number $E_n$ is the cardinality of $\Tn$.
A one-to-one correspondance between $\An$ and $\Tn$ was then
constructed by Donaghey \cite{Don75} (see also \cite{Cal05}).
However the tree counterpart of Entringer's result was found only in 1982 by
Poupard~\cite{Pou82}. If $T$ is a binary  increasing tree and if $(i,j)$ is an edge in $T$, $i < j$, we call $i$ the \emph{parent} of $j$, and $j$ a \emph{child} of $i$. If $i$ has no child, we say that $i$ is a \emph{leaf} of $T$. A \emph{path} in $T$ is a sequence of vertices $(a_i)$ such that $a_i$ is a child of $a_{i-1}$ in T, and the {\em minimal path} of $T$ is the path $(a_i)_{1 \leq i \leq \ell}$ such that $a_1=1$, $a_i$ ($i=2,\ldots,\ell$) is the smallest child of $a_{i-1}$ and $a_{\ell}$ is a leaf, denoted by $p(T)$. Let's denote by $\Tnk$ the set of trees $T \in \Tn$ such that $p(T)=k$.

\begin{thm}[Poupard] \label{thm-pou1}The sequence $(\Tnk)_{1 \leq k \leq n}$ is an Entringer family. \end{thm}

Note that contrary to the case of  down-up permutations, it is not easy to interpret recurrence \eqref{eq1}
 in the model of binary  increasing increasing trees.
 Indeed, Donaghey's bijection doesn't induce a bijection between $\Ank$ and $\Tnk$ and Poupard's proof in \cite{Pou82} was analytic in nature.
 Finding a direct explanation in the model of trees was then raised as an open problem in \cite{KPP94}.
The first aim of this paper is to build a bijection between $\Ank$ and $\Tnk$ and answer the above open problem.
In other words, we have the following theorem.

\begin{thm}\label{mainth}  For all $n \geq 1$, there is an explicit  bijection $\Psi: \An\to \Tn$ satisfying
\[ \forall \pi \in \An, \quad \textsc{First}\,(\pi)=\textsc{Leaf}\,(\Psi(\pi)), \] 
where $\textsc{First}\,(\pi)$ is the first element of the permutation $\pi$ 
 and $\textsc{Leaf}\,(\Psi(\pi))$ is the leaf of the minimal path of the tree $\Psi(\pi)$.
\end{thm}

Poupard \cite{Pou82,Pou97} gave also other interpretations for Entringer numbers $E_{n,k}$ (see Section~\ref{sec-poupard}) in binary increasing trees and down-up permutations with induction proofs.
Our second aim is to provide simple bijections between the other interpretations of Poupard in down-up permutations and the original interpretation in $\mathcal{DU}_{n,k}$.  Note that some other interpretations of Entringer numbers $E_{n,k}$ in the model of increasing trees were given in  \cite{KPP94}.
Recently, two new interpretations of Euler numbers were given by Martin and Wagner~\cite{MW09} in the model of G-words and R-words. We shall give the corresponding interpretations of
 the Entringer number $E_{n,k}$  in the later models.

The rest of this paper is organized as follows.
In Section~\ref{s-dnk}, we introduce an intermediate model $\ESnk$ and present a bijection $\psi$ between $\Ank$ and $\ESnk$.
In Section~\ref{s-bij}, we describe a bijection $\varphi$ between $\ESnk$ and $\Tnk$ so that $\Psi=\varphi\circ \psi$  provides the  bijection for Theorem~\ref{mainth}.
As an application, in Subsection~\ref{s-rmq}, we give a direct interpretation of \eqref{eq1} in the model of increasing trees. In Section~\ref{sec-poupard}, we recall the other interpretations of $E_{n,k}$ found by Poupard and establish simple bijections between these models. In Section~\ref{s-2new}, we give some new interpretations for $E_{n,k}$, first refining the results of Martin and Wagner \cite{MW09} in their model of G-words and R-words, and secondly introducing the new model of U-words.

\section{The left-to-right coding $\psi$ of down-up permutations} \label{s-dnk}
Consider down-up permutations   on  any finite subset  $I=\{a_1,a_2, \ldots,a_m\}_{<}$ of $\mathbb{N}$.
Two elements $a$ and $b$ in $I$ are said to be \emph{adjacent}  if there is no $c\in I$ between $a$ and $b$.
Let   $\pi$ be a down-up permutation on $I$, i.e., $\pi_1 > \pi_2 < \pi_3 > \pi_4 < \ldots$.
Suppose  $\pi_1=a_i$ and $\pi_2=a_j$ with $a_{i}>a_{j}$.   If $\pi_{1}$ and $\pi_{2}$ are adjacent, then, deleting $\pi_{1}\pi_{2}$,  we obtain 
again a down-up permutation on $I\setminus \{\pi_{1}, \pi_{2}\}$, otherwise, we can apply 
successively  the adjacent transpositions
$(a_{i}, a_{i-1})$, $(a_{i-1}, a_{i-2}),\ldots, (a_{j+2}, a_{j+1})$ to $\pi$ (from left-to-right):
\begin{align*}
\pi^{(1)}=(a_{i}, a_{i-1})\circ \pi,\quad
\pi^{(2)}=(a_{i-1}, a_{i-2})\circ \pi^{(1)},\quad \ldots,\quad
\pi^{(i-j-1)}=(a_{j+2}, a_{j+1})\circ \pi^{(i-j-2)},
\end{align*}
so that  all the  permutations $\pi^{(1)},\ldots, \pi^{(k-j-1)}$ are down-up permutations and 
 the first two elements in  $\pi^{(i-j-1)}$ are adjacent. Deleting the first two elements,   we get  again a
  down-up  permutation, say  $\pi^{(i-j)}$, on $I\setminus \{a_{j+1},a_{j}\}$. 
If we register  $(a,b)$ for the composition from left with the adjacent involution $(a,b)$, and  $(a,b)^{*}$ for the deletion of the first two letters $a$ and $b$, then 
the operations in the above process can be encoded by  the  word
  $$
 (a_{i}, a_{i-1})\, (a_{i-1}, a_{i-2})\, \ldots\, (a_{j+2}, a_{j+1})\, (a_{j+1}, a_{j})^{*}.
  $$
Since  the resulting permutation $\pi^{(i-j)}$ is still down-up, we can  
 iterate  this process until we obtain the empty permutation. Clearly the last deletion is $(n)^{*}$ if $n$ is odd.
 We shall call \emph{left-to-right code} the resulting  sequence of the successive operations in this process and denote it by
 $\psi(\pi)=(\Delta_{\ell})_{\ell}$, where  each entry $\Delta_{\ell}$ is either a transposition $(j,i)$, 
a  deletion $(j,i)^{*}, \;1 \leq i<j \leq n$,  or the deletion $(n)^{*}$.
Formally, we can write the algorithm as follows: \\
\begin{enumerate}
\item Start with $(\pi, \Delta=\emptyset)$ and support set $I=\{a_1,a_2, \ldots,a_m\}_{<}$
\item While $Card(A) \geq 2$, do:
\begin{enumerate}
\item While there is  $a\in I$ such that $\pi_1>a>\pi_2$, do: \begin{itemize}
\item[]$\Delta \leftarrow (\Delta,(\pi_{1},a'))$, where $a'=\max\{a\in I | \pi_1>a>\pi_2\}$,
\item[] $\pi \leftarrow (\pi_1, a') \circ \pi$.
\end{itemize}
\item If there is no $a\in I$ such that $\pi_1>a>\pi_2$, do:
\begin{itemize}
\item[] $\Delta \leftarrow (\Delta,(\pi_{1},\pi_2)^{*} )$,
\item[] $\pi \leftarrow \pi_3 \pi_4 \ldots \pi_n$ (eventually $\pi=\emptyset$),
\item[] $I\leftarrow I \setminus \{ \pi_{1},\pi_2\}$.
\end{itemize}
\end{enumerate}
\item If $Card(I)=1$ with $I=\{ a_m \}$, do:
\begin{itemize}
\item[]$D_\pi \leftarrow (\Delta,(a_m)^{*} )$,
\item[]$\pi \leftarrow \emptyset$,
\item[]$I \leftarrow \emptyset$.
\end{itemize}
\end{enumerate}

\begin{ex} \label{ex2.2} If $\pi = 7\,4\,8\,5\,9\,1\,6\,2\,3 \in \mathcal{DU}_{9,7}$, then the algorithm goes as follows:
\[ \begin{tabular}{c|r|l} Step & \multicolumn{1}{c|}{$\pi^{(\ell)}$} & $\Delta_\ell$  \\
\hline
0 & $7\,4\,8\,5\,9\,1\,6\,2\,3$ & $\emptyset$\\
1 & $6\,4\,8\,5\,9\,1\,7\,2\,3$ & $(7,6)$ \\
2 & $5\,4\,8\,6\,9\,1\,7\,2\,3$ & $(6,5)$ \\
3 & $  8\,6\,9\,1\,7\,2\,3$ & $(5,4)^{*} $ \\
\hline
4 & $  7\,6\,9\,1\,8\,2\,3$ & $(8,7)$ \\
5 & $    9\,1\,8\,2\,3$ & $(7,6)^{*} $ \\
\hline
6 & $    8\,1\,9\,2\,3$ & $(9,8)$ \\
7 & $    3\,1\,9\,2\,8$ & $(8,3)$ \\
8 & $    2\,1\,9\,3\,8$ & $(3,2)$ \\
9 & $      9\,3\,8$ & $(2,1)^{*} $ \\
\hline
10 & $      8\,3\,9$ & $(9,8)$ \\
11 & $        9$ & $(8,3)^{*} $ \\
\hline
12 & $         \emptyset$ & $(9)^{*}$
 \end{tabular}. \]
Thus, the left-to-right code of $\pi$ is
$$
 \psi(\pi) = (7,6)\,(6,5)\,(5,4)^{*} \,(8,7)\,(7,6)^{*}\,(9,8)\,(8,3)\,(3,2)\,(2,1)^{*}\,(9,8)\,(8,3)^{*} \,(9)^{*}.
$$
\end{ex}

A  \emph{domino} on $[n]$  is an ordered pairs $(j,i)$ ($1 \leq i < j \leq n$) and a \emph{starred domino} on $[n]$ is a
starred ordered pairs $(j,i)^{*}$ ($1 \leq i < j \leq n$) or $(n)^{*}=(n,n)^{*}$.  Let $\mathbb{A}_{n}$ be the alphabet consisting of dominos (starred or non) on $[n]$.

\begin{defn} \label{def:enc} A word $\Delta=\Delta_{1}\ldots \Delta_r$ on $\mathbb{A}_{n}$ is an \emph{encoding sequence of $[n]$} if the following conditions are verified:
\begin{itemize}
\item[(i)] the  entries of  starred dominos are disjoint and their union equals $[n]$,
\item[(ii)] if $\Delta_\ell=(j,i)^{*}$, then the next domino (if there is one) starts with an entry $>i$, and no entry of a later domino
 lies between $i$ and $j$,
\item[(iii)] if $\Delta_\ell=(j,i)$, then both $i$ and $j$ appear in a later domino, with $i$ the first entry of the next domino, and each integer between $i$ and $j$ appears in an earlier starred domino.
\end{itemize}
 \end{defn}
\begin{rmk}
It is clear from the definition that $(n, i)^{*}$ ($1\leq i \leq n$) can only take the last position in an encoding sequence and 
an encoding sequence must start with $(k,k-1)$ or $(k,k-1)^{*}$  for $2\leq k\leq n$.
\end{rmk}

We denote by $\Dn$ the set of encoding sequences of $[n]$, and by $\ESnk$ the subset of $\ESn$ consisting of encoding sequences starting with 
$(k,k-1)$ or $(k,k-1)^{*}$, $2\leq k\leq n$.
For example, the set  $\mathcal{ES}_4$  is the union of the three subsets:
\begin{align*}
\mathcal{ES}_{4,2}&=\left\{ (2,1)^{*}\, (4,3)^{*}\right\},\\
\mathcal{ES}_{4,3}&=\left\{(3,2)^{*}\, (4,1)^{*},\;  (3,2)\, (2,1)^{*}\, (4,3)^{*} \right\},\\
\mathcal{ES}_{4,4}&= \left\{(4,3)\, (3,2)^{*}\, (4,1)^{*}, \; (4,3)\,(3,2)\, (2,1)^{*}\,(4,3)^{*} \right\}.
\end{align*}
\begin{thm}
For all $n \geq 1$ and $k \in [n]$, the  mapping $\psi: \Ank\to \ESnk$ is a bijection.
Therefore, the sequence $(\ESnk)_{1 \leq k \leq n}$ is an Entringer family.
\end{thm}
\begin{proof}
Let $\pi=\pi_{1}\ldots \pi_{n}$ be an element in $ \Ank$. Then $\pi_{1}=k$, so the first letter of $\psi(\pi)$ is $(k,i)$ or $(k,i)^{*}$ ($1\leq i<k$) by  definition of $\psi$.
It remains  to show that the word $\psi(\pi)$ verifies the conditions (i)-(iii) of Definition~\ref{def:enc}.
Since the process reduces the permutation $\pi$ to  empty permutation,  the condition  (i) is verified.
\begin{itemize}
\item If  $\Delta_\ell=(j,i)^{*} $, as $i$ and $j$ are adjacent in the support set of $\pi^{(\ell)}$, the integers between $i$ and $j$ have been removed in previous starred 
dominos, also the first entry of the next domino is $>i$ because $\pi^{(\ell)}$ is down-up.
\item If $\Delta_\ell=(j,i)$,  as $i$ and $j$ are adjacent in the support set of $\pi^{(\ell)}$, the integers between $i$ and $j$ have been removed in previous starred 
dominos, also  the next domino  must be $(i,m)$ or $(i,m)^{*}$ with $i>m$ because $i$ is the first entry of $\pi^{(\ell)}$.
\end{itemize}
It results  that $\psi(\pi) \in \ESnk$. 

Conversely, starting from an encoding sequence $\Delta=\Delta_1\ldots\Delta_\ell \in \mathcal{ES}_{n,k}$, we construct by induction 
$\pi^{(j)}$ such that $\textsc{First}(\pi^{(j)})$  equals the first entry of $\Delta_j$ for $j=\ell, \ell-1, \ldots, 1$.

First, if $\Delta_{\ell}=(n)^{*}$ then define 
$\pi^{(\ell)}=n$, if $\Delta_{\ell}=(n,i)^{*}$ with $i<n$, then 
define $\pi^{(\ell)}=n\,i$.

Assume  that $\pi^{(j+1)}$ is constructed with $\text{First}(\pi^{(j+1)})=k_{j+1}$. By definition of  $\Delta$, there are two cases:
\begin{itemize}
\item[(i)]
 if $\Delta_j=(k_j,k_{j+1})$, where $k_j$ and  
 $k_{j+1}$ are adjacent in the support set of $\pi^{(j+1)}$, then
 define $\pi^{(j)}:=(k_j,k_{j+1}) \circ \pi^{(j+1)}$. This permutation is still down-up and the first element of $\pi^{(j)}$ is $k_j$;
 \item[(ii)]
 if $\Delta_j=(a_j,b_j)^*$, where  $a_j > b_j < k_{j+1}$, and $a_j$, $b_j$  are not in the support set of $\pi^{(j+1)}$, then 
define $\pi^{(j)}$ as the word $a_j b_j \pi^{(j+1)}$. Since $a_j > b_j < k_{j+1}$, the permutation $\pi^{(j)}$ is down-up with $a_{j}$ as the  first element.
\end{itemize}
Let $\psi^{-1}(\Delta) :=\pi^{(1)}$, which is an element in $ \mathcal{DU}_{n,k}$.
\end{proof}

\begin{rmk} Denote
the largest integer less than $x$ by $\displaystyle \left\lfloor x \right\rfloor$ and  the number of ordered pairs $(i,j) \in \{1,\ldots,n\}$ such that $i+1<j$ and $\pi_i > \pi_{i+1} < \pi_j < \pi_i$ by
 $(\312)\pi$. Then, one can show that the length of the sequence $\psi(\pi)$ is equal to
\[ (\312)\pi + \left\lfloor \frac{n+1}{2} \right\rfloor. \]
Indeed, $(\312)\pi$ corresponds to the number of occurences of terms $(j,i)$, $j> i$, in $\psi(\pi)$, and there are $\displaystyle \left\lfloor\frac{n+1}{2}\right\rfloor$ occurences of terms $(j,i)^{*} $, $j>i$, in $\psi(\pi)$. Note that various formulae for counting $\312$-patterns in down-up
permutations are given in \cite{Che08, JV09,SZ10}.
\end{rmk}

\begin{prop} \label{s-rmq2}
Let $n \geq 2$ and $k \geq 2$. The number of elements starting with $(k,k-1)$ equals $E_{n,k-1}$, and the number of elements starting with $(k,k-1)^{*}$ equals $E_{n-1, n+1-k}$.
\end{prop}
\begin{proof} Let $\Delta \in \ESnk$.
If $\Delta_1=(k,k-1)$, the remaining sequence $(\Delta_2,\Delta_3,\ldots)$ is still an encoding sequence of $[n]$, starting with $\Delta_2 \in \{ (k-1,i), (k-1,i), 1 \leq i \leq k-2 \}$. Thus, there are $E_{n,k-1}$ encoding sequences starting by $(k,k-1)$.
If $\Delta_1=(k,k-1)^*$,  the remaining sequence $(\Delta_2,\Delta_3,\ldots)$ doesn't contain the elements $k$ and $k+1$ and starts with an element in $\{ (i,j), (i,j)^*, 1 \leq j \leq i-1 \}$ with $i \geq k+1$. In other words, this is an encoding sequence of $n-2$ elements, starting with an integer $i$ that must be greater than the $k-2$ first elements. Thus, there are $E_{n-2,k-1}+E_{n-2,k}+\cdots + E_{n-2,n-2} = E_{n-1,n+1-k}$ encoding sequences starting by $(k,k-1)^*$.
\end{proof}

Since any sequence in $\ESnk$ begins  with either $(k,k-1)$ or $(k,k-1)^*$ ($2\leq k\leq n$), 
 Entringer's formula~\eqref{eq1} results from  the above proposition.

\section{The left-to-right coding of binary  trees} \label{s-bij}
\subsection{The bijection $\varphi : \ESnk \rightarrow \Tnk$}
Starting from an encoding sequence $\Delta=\Delta_{1}\ldots \Delta_{\ell} \in \ESnk$, we construct a
 tree $T=\varphi(\Delta) \in \Tnk$ by reading the sequence $\Delta$ in reverse order, i.e., from right to left.
  More precisely,   for $m=\ell, \ell-1, \ldots, 1$, we shall construct a tree $T_{m}$
  corresponding to 
   the word $\Delta_{m}\ldots \Delta_{\ell-1}\Delta_{\ell}$
   such that  \begin{align}\label{propm}
\textrm{$\Delta_m = (j_m,i_m)$ or $(j_m,i_m)^{*}$}\Longrightarrow  \textsc{Leaf}(T_m)=j_m,
 \end{align}
and define  $T=T_{1}:=\varphi(\Delta)$.  The algorithm goes as follows:

 If $\Delta_{\ell}=(n)^{*} $,  construct the tree $T_{\ell}$ with  only one vertex $n$; if $\Delta_{\ell}=(n,i)^{*} $,
construct  the increasing tree $T_{\ell}$ with only  one edge $i\to n$. Clearly \eqref{propm}  is verified.

Assume  that we have constructed such a  tree $T_{m+1}$ corresponding to the word 
 $\Delta_{m+1}\,\ldots\,\Delta_{\ell}$.

\begin{itemize}
\item[(i)] If  $\Delta_{m}=(j_m,i_m)^{*}$, we add vertices $i_m$ and $j_m$ to the tree $T_{m+1}$ to obtain $T_{m}$. Suppose  that the minimal path of $T_{m+1}$ is $(a_1,\ldots,a_{p_m})$.
\begin{itemize}
\item[$\bullet$] If $i_m < a_1$,  add  the  edges $(i_m,a_1)$ and  $(i_m,j_m)$ to the tree $T_{m+1}$.  Then, the tree $T_{m}$
 is an increasing tree rooted at $i_m$ with  $(i_m,j_m)$ as the minimal path.
    $$\qquad 
    
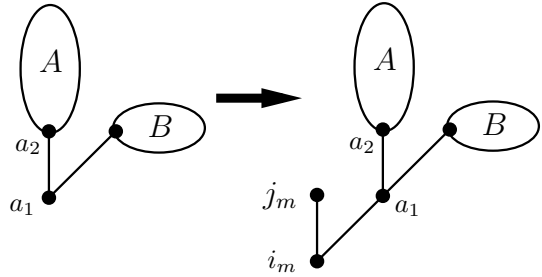
\centering
\begin{pgfpicture}{42.05mm}{27.22mm}{119.94mm}{67.08mm}
\pgfsetxvec{\pgfpoint{0.80mm}{0mm}}
\pgfsetyvec{\pgfpoint{0mm}{0.80mm}}
\color[rgb]{0,0,0}\pgfsetlinewidth{0.30mm}\pgfsetdash{}{0mm}
\pgfsetlinewidth{1.20mm}\pgfmoveto{\pgfxy(94.38,65.61)}\pgflineto{\pgfxy(105.38,65.61)}\pgfstroke
\pgfmoveto{\pgfxy(105.38,65.61)}\pgflineto{\pgfxy(102.57,66.31)}\pgflineto{\pgfxy(102.57,64.91)}\pgflineto{\pgfxy(105.38,65.61)}\pgfclosepath\pgffill
\pgfmoveto{\pgfxy(105.38,65.61)}\pgflineto{\pgfxy(102.57,66.31)}\pgflineto{\pgfxy(102.57,64.91)}\pgflineto{\pgfxy(105.38,65.61)}\pgfclosepath\pgfstroke
\pgfcircle[fill]{\pgfxy(110.96,49.57)}{0.80mm}
\pgfsetlinewidth{0.30mm}\pgfcircle[stroke]{\pgfxy(110.96,49.57)}{0.80mm}
\pgfputat{\pgfxy(107.78,48.21)}{\pgfbox[bottom,left]{\fontsize{12.12}{14.54}\selectfont \makebox[0pt][r]{$j_m$}}}
\pgfputat{\pgfxy(107.66,37.47)}{\pgfbox[bottom,left]{\fontsize{10.02}{12.02}\selectfont \makebox[0pt][r]{$i_m$}}}
\pgfcircle[fill]{\pgfxy(110.96,38.57)}{0.80mm}
\pgfcircle[stroke]{\pgfxy(110.96,38.57)}{0.80mm}
\pgfmoveto{\pgfxy(121.96,49.57)}\pgflineto{\pgfxy(110.96,38.57)}\pgfstroke
\pgfmoveto{\pgfxy(110.96,49.57)}\pgflineto{\pgfxy(110.96,38.57)}\pgfstroke
\pgfcircle[fill]{\pgfxy(66.88,49.11)}{0.80mm}
\pgfcircle[stroke]{\pgfxy(66.88,49.11)}{0.80mm}
\pgfmoveto{\pgfxy(66.88,49.11)}\pgflineto{\pgfxy(77.88,60.11)}\pgfstroke
\pgfmoveto{\pgfxy(66.88,49.11)}\pgflineto{\pgfxy(66.88,60.11)}\pgfstroke
\pgfcircle[fill]{\pgfxy(66.88,60.11)}{0.80mm}
\pgfcircle[stroke]{\pgfxy(66.88,60.11)}{0.80mm}
\pgfputat{\pgfxy(65.64,56.92)}{\pgfbox[bottom,left]{\fontsize{10.02}{12.02}\selectfont \makebox[0pt][r]{$a_{2}$}}}
\pgfputat{\pgfxy(64.68,46.91)}{\pgfbox[bottom,left]{\fontsize{10.02}{12.02}\selectfont \makebox[0pt][r]{$a_1$}}}
\pgfcircle[fill]{\pgfxy(77.88,60.11)}{0.80mm}
\pgfcircle[stroke]{\pgfxy(77.88,60.11)}{0.80mm}
\pgfellipse[stroke]{\pgfxy(67.09,70.45)}{\pgfxy(4.87,0.00)}{\pgfxy(0.00,10.60)}
\pgfputat{\pgfxy(65.29,69.97)}{\pgfbox[bottom,left]{\fontsize{12.52}{15.02}\selectfont $A$}}
\pgfellipse[stroke]{\pgfxy(84.98,60.67)}{\pgfxy(7.54,0.00)}{\pgfxy(0.00,4.08)}
\pgfputat{\pgfxy(83.13,59.21)}{\pgfbox[bottom,left]{\fontsize{12.52}{15.02}\selectfont $B$}}
\pgfcircle[fill]{\pgfxy(121.78,49.41)}{0.80mm}
\pgfcircle[stroke]{\pgfxy(121.78,49.41)}{0.80mm}
\pgfmoveto{\pgfxy(121.78,49.41)}\pgflineto{\pgfxy(132.78,60.41)}\pgfstroke
\pgfmoveto{\pgfxy(121.78,49.41)}\pgflineto{\pgfxy(121.78,60.41)}\pgfstroke
\pgfcircle[fill]{\pgfxy(121.78,60.41)}{0.80mm}
\pgfcircle[stroke]{\pgfxy(121.78,60.41)}{0.80mm}
\pgfputat{\pgfxy(120.55,57.23)}{\pgfbox[bottom,left]{\fontsize{10.02}{12.02}\selectfont \makebox[0pt][r]{$a_{2}$}}}
\pgfputat{\pgfxy(123.74,46.75)}{\pgfbox[bottom,left]{\fontsize{10.02}{12.02}\selectfont $a_1$}}
\pgfcircle[fill]{\pgfxy(132.78,60.41)}{0.80mm}
\pgfcircle[stroke]{\pgfxy(132.78,60.41)}{0.80mm}
\pgfellipse[stroke]{\pgfxy(121.99,70.75)}{\pgfxy(4.87,0.00)}{\pgfxy(0.00,10.60)}
\pgfputat{\pgfxy(120.20,70.28)}{\pgfbox[bottom,left]{\fontsize{12.52}{15.02}\selectfont $A$}}
\pgfellipse[stroke]{\pgfxy(139.89,60.98)}{\pgfxy(7.54,0.00)}{\pgfxy(0.00,4.08)}
\pgfputat{\pgfxy(138.04,59.51)}{\pgfbox[bottom,left]{\fontsize{12.52}{15.02}\selectfont $B$}}
\end{pgfpicture}%
$$
\item[$\bullet$] If $i_m>a_1$,  by induction hypothesis and property (ii) of encoding sequences,  we see  that $a_1 < m$. Hence, there exists $k \in \{1,\ldots,p_m-1\}$ such that $a_k < i_m < a_{k+1}$. Then, erase the edge $(a_k,a_{k+1})$, create the edges $(a_k,i_m)$, $(i_m,a_{k+1})$ and $(i_m,j_m)$.
Clearly, the tree $T_{m}$ is  an increasing tree  with  $(i_{m}, j_m)$ as the last edge  of the minimal path.
$$
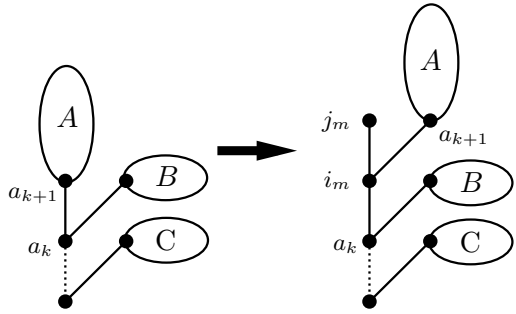
\centering
\begin{pgfpicture}{38.00mm}{29.20mm}{122.00mm}{73.40mm}
\pgfsetxvec{\pgfpoint{0.80mm}{0mm}}
\pgfsetyvec{\pgfpoint{0mm}{0.80mm}}
\color[rgb]{0,0,0}\pgfsetlinewidth{0.30mm}\pgfsetdash{}{0mm}
\pgfsetlinewidth{1.20mm}\pgfmoveto{\pgfxy(95.00,65.00)}\pgflineto{\pgfxy(105.00,65.00)}\pgfstroke
\pgfmoveto{\pgfxy(105.00,65.00)}\pgflineto{\pgfxy(102.20,65.70)}\pgflineto{\pgfxy(102.20,64.30)}\pgflineto{\pgfxy(105.00,65.00)}\pgfclosepath\pgffill
\pgfmoveto{\pgfxy(105.00,65.00)}\pgflineto{\pgfxy(102.20,65.70)}\pgflineto{\pgfxy(102.20,64.30)}\pgflineto{\pgfxy(105.00,65.00)}\pgfclosepath\pgfstroke
\pgfsetdash{{0.30mm}{0.50mm}}{0mm}\pgfsetlinewidth{0.30mm}\pgfmoveto{\pgfxy(120.00,50.00)}\pgflineto{\pgfxy(120.00,40.00)}\pgfstroke
\pgfcircle[fill]{\pgfxy(120.00,50.00)}{0.80mm}
\pgfsetdash{}{0mm}\pgfcircle[stroke]{\pgfxy(120.00,50.00)}{0.80mm}
\pgfcircle[fill]{\pgfxy(120.00,40.00)}{0.80mm}
\pgfcircle[stroke]{\pgfxy(120.00,40.00)}{0.80mm}
\pgfmoveto{\pgfxy(120.00,50.00)}\pgflineto{\pgfxy(130.00,60.00)}\pgfstroke
\pgfcircle[fill]{\pgfxy(120.00,70.00)}{0.80mm}
\pgfcircle[stroke]{\pgfxy(120.00,70.00)}{0.80mm}
\pgfmoveto{\pgfxy(120.00,40.00)}\pgflineto{\pgfxy(130.00,50.00)}\pgfstroke
\pgfcircle[fill]{\pgfxy(130.00,50.00)}{0.80mm}
\pgfcircle[stroke]{\pgfxy(130.00,50.00)}{0.80mm}
\pgfputat{\pgfxy(117.00,69.00)}{\pgfbox[bottom,left]{\fontsize{9.10}{10.93}\selectfont \makebox[0pt][r]{$j_m$}}}
\pgfputat{\pgfxy(117.00,59.00)}{\pgfbox[bottom,left]{\fontsize{9.10}{10.93}\selectfont \makebox[0pt][r]{$i_m$}}}
\pgfcircle[fill]{\pgfxy(120.00,60.00)}{0.80mm}
\pgfcircle[stroke]{\pgfxy(120.00,60.00)}{0.80mm}
\pgfmoveto{\pgfxy(120.00,50.00)}\pgflineto{\pgfxy(120.00,60.00)}\pgfstroke
\pgfmoveto{\pgfxy(130.00,70.00)}\pgflineto{\pgfxy(120.00,60.00)}\pgfstroke
\pgfmoveto{\pgfxy(120.00,70.00)}\pgflineto{\pgfxy(120.00,60.00)}\pgfstroke
\pgfcircle[fill]{\pgfxy(130.00,70.00)}{0.80mm}
\pgfcircle[stroke]{\pgfxy(130.00,70.00)}{0.80mm}
\pgfcircle[fill]{\pgfxy(130.00,60.00)}{0.80mm}
\pgfcircle[stroke]{\pgfxy(130.00,60.00)}{0.80mm}
\pgfsetdash{{0.30mm}{0.50mm}}{0mm}\pgfmoveto{\pgfxy(70.00,50.00)}\pgflineto{\pgfxy(70.00,40.00)}\pgfstroke
\pgfcircle[fill]{\pgfxy(70.00,50.00)}{0.80mm}
\pgfsetdash{}{0mm}\pgfcircle[stroke]{\pgfxy(70.00,50.00)}{0.80mm}
\pgfcircle[fill]{\pgfxy(70.00,40.00)}{0.80mm}
\pgfcircle[stroke]{\pgfxy(70.00,40.00)}{0.80mm}
\pgfmoveto{\pgfxy(70.00,50.00)}\pgflineto{\pgfxy(80.00,60.00)}\pgfstroke
\pgfmoveto{\pgfxy(70.00,40.00)}\pgflineto{\pgfxy(80.00,50.00)}\pgfstroke
\pgfcircle[fill]{\pgfxy(80.00,50.00)}{0.80mm}
\pgfcircle[stroke]{\pgfxy(80.00,50.00)}{0.80mm}
\pgfmoveto{\pgfxy(70.00,50.00)}\pgflineto{\pgfxy(70.00,60.00)}\pgfstroke
\pgfcircle[fill]{\pgfxy(70.00,60.00)}{0.80mm}
\pgfcircle[stroke]{\pgfxy(70.00,60.00)}{0.80mm}
\pgfputat{\pgfxy(68.88,57.11)}{\pgfbox[bottom,left]{\fontsize{9.10}{10.93}\selectfont \makebox[0pt][r]{$a_{k+1}$}}}
\pgfputat{\pgfxy(68.00,48.00)}{\pgfbox[bottom,left]{\fontsize{9.10}{10.93}\selectfont \makebox[0pt][r]{$a_k$}}}
\pgfcircle[fill]{\pgfxy(80.00,60.00)}{0.80mm}
\pgfcircle[stroke]{\pgfxy(80.00,60.00)}{0.80mm}
\pgfputat{\pgfxy(118.20,48.56)}{\pgfbox[bottom,left]{\fontsize{9.10}{10.93}\selectfont \makebox[0pt][r]{$a_k$}}}
\pgfputat{\pgfxy(131.38,66.61)}{\pgfbox[bottom,left]{\fontsize{9.10}{10.93}\selectfont $a_{k+1}$}}
\pgfellipse[stroke]{\pgfxy(70.19,69.40)}{\pgfxy(4.43,0.00)}{\pgfxy(0.00,9.64)}
\pgfellipse[stroke]{\pgfxy(130.18,79.61)}{\pgfxy(4.43,0.00)}{\pgfxy(0.00,9.64)}
\pgfputat{\pgfxy(68.56,68.97)}{\pgfbox[bottom,left]{\fontsize{11.38}{13.66}\selectfont $A$}}
\pgfputat{\pgfxy(128.26,78.34)}{\pgfbox[bottom,left]{\fontsize{11.38}{13.66}\selectfont $A$}}
\pgfellipse[stroke]{\pgfxy(86.46,60.51)}{\pgfxy(6.85,0.00)}{\pgfxy(0.00,3.71)}
\pgfellipse[stroke]{\pgfxy(86.60,50.16)}{\pgfxy(6.85,0.00)}{\pgfxy(0.00,3.71)}
\pgfellipse[stroke]{\pgfxy(136.65,49.88)}{\pgfxy(6.85,0.00)}{\pgfxy(0.00,3.71)}
\pgfellipse[stroke]{\pgfxy(136.65,59.67)}{\pgfxy(6.85,0.00)}{\pgfxy(0.00,3.71)}
\pgfputat{\pgfxy(84.78,59.18)}{\pgfbox[bottom,left]{\fontsize{11.38}{13.66}\selectfont $B$}}
\pgfputat{\pgfxy(134.84,58.20)}{\pgfbox[bottom,left]{\fontsize{11.38}{13.66}\selectfont $B$}}
\pgfputat{\pgfxy(84.78,48.56)}{\pgfbox[bottom,left]{\fontsize{11.38}{13.66}\selectfont C}}
\pgfputat{\pgfxy(134.84,48.28)}{\pgfbox[bottom,left]{\fontsize{11.38}{13.66}\selectfont C}}
\end{pgfpicture}%
$$
\end{itemize}
\item[(ii)] If  $\Delta_{m}=(j_m,i_m)$, where $i_{m}$ and $j_{m}$ are not siblings in $T_{m+1}$, by induction hypothesis and property (iii) of encoding sequences, 
we derive  that $i_m$ is at the end of the minimal path. Then, we transform the tree $T_{m+1}$ as follows:
 just exchange the places of $i_m$ and $j_m$ in $T_{m+1}$. The tree remains increasing because Then $j_m$ is at the end of the minimal path in $T_{m}$.
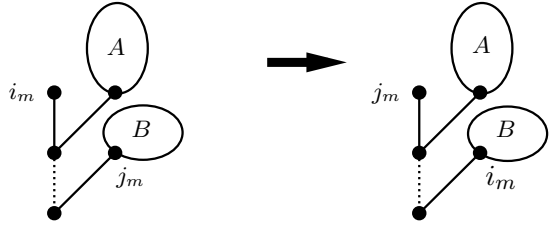
$$
\centering
\begin{pgfpicture}{38.62mm}{29.20mm}{114.81mm}{61.87mm}
\pgfsetxvec{\pgfpoint{0.80mm}{0mm}}
\pgfsetyvec{\pgfpoint{0mm}{0.80mm}}
\color[rgb]{0,0,0}\pgfsetlinewidth{0.30mm}\pgfsetdash{}{0mm}
\pgfsetlinewidth{1.20mm}\pgfmoveto{\pgfxy(95.00,65.00)}\pgflineto{\pgfxy(105.00,65.00)}\pgfstroke
\pgfmoveto{\pgfxy(105.00,65.00)}\pgflineto{\pgfxy(102.20,65.70)}\pgflineto{\pgfxy(102.20,64.30)}\pgflineto{\pgfxy(105.00,65.00)}\pgfclosepath\pgffill
\pgfmoveto{\pgfxy(105.00,65.00)}\pgflineto{\pgfxy(102.20,65.70)}\pgflineto{\pgfxy(102.20,64.30)}\pgflineto{\pgfxy(105.00,65.00)}\pgfclosepath\pgfstroke
\pgfputat{\pgfxy(57.00,59.00)}{\pgfbox[bottom,left]{\fontsize{9.10}{10.93}\selectfont \makebox[0pt][r]{$i_m$}}}
\pgfsetdash{{0.30mm}{0.50mm}}{0mm}\pgfsetlinewidth{0.30mm}\pgfmoveto{\pgfxy(60.00,50.00)}\pgflineto{\pgfxy(60.00,40.00)}\pgfstroke
\pgfcircle[fill]{\pgfxy(70.00,60.00)}{0.80mm}
\pgfsetdash{}{0mm}\pgfcircle[stroke]{\pgfxy(70.00,60.00)}{0.80mm}
\pgfcircle[fill]{\pgfxy(60.00,50.00)}{0.80mm}
\pgfcircle[stroke]{\pgfxy(60.00,50.00)}{0.80mm}
\pgfcircle[fill]{\pgfxy(60.00,40.00)}{0.80mm}
\pgfcircle[stroke]{\pgfxy(60.00,40.00)}{0.80mm}
\pgfmoveto{\pgfxy(60.00,40.00)}\pgflineto{\pgfxy(70.00,50.00)}\pgfstroke
\pgfcircle[fill]{\pgfxy(70.00,50.00)}{0.80mm}
\pgfcircle[stroke]{\pgfxy(70.00,50.00)}{0.80mm}
\pgfputat{\pgfxy(70.20,44.94)}{\pgfbox[bottom,left]{\fontsize{9.10}{10.93}\selectfont $j_m$}}
\pgfcircle[fill]{\pgfxy(60.00,60.00)}{0.80mm}
\pgfcircle[stroke]{\pgfxy(60.00,60.00)}{0.80mm}
\pgfmoveto{\pgfxy(60.00,50.00)}\pgflineto{\pgfxy(60.00,60.00)}\pgfstroke
\pgfmoveto{\pgfxy(70.00,60.00)}\pgflineto{\pgfxy(60.00,50.00)}\pgfstroke
\pgfellipse[stroke]{\pgfxy(70.37,67.33)}{\pgfxy(5.00,0.00)}{\pgfxy(0.00,7.51)}
\pgfputat{\pgfxy(70.19,66.22)}{\pgfbox[bottom,left]{\fontsize{9.10}{10.93}\selectfont \makebox[0pt]{$A$}}}
\pgfputat{\pgfxy(74.34,52.55)}{\pgfbox[bottom,left]{\fontsize{9.10}{10.93}\selectfont \makebox[0pt]{$B$}}}
\pgfputat{\pgfxy(117.00,59.00)}{\pgfbox[bottom,left]{\fontsize{9.10}{10.93}\selectfont \makebox[0pt][r]{$j_m$}}}
\pgfsetdash{{0.30mm}{0.50mm}}{0mm}\pgfmoveto{\pgfxy(120.00,50.00)}\pgflineto{\pgfxy(120.00,40.00)}\pgfstroke
\pgfcircle[fill]{\pgfxy(130.00,60.00)}{0.80mm}
\pgfsetdash{}{0mm}\pgfcircle[stroke]{\pgfxy(130.00,60.00)}{0.80mm}
\pgfcircle[fill]{\pgfxy(120.00,50.00)}{0.80mm}
\pgfcircle[stroke]{\pgfxy(120.00,50.00)}{0.80mm}
\pgfcircle[fill]{\pgfxy(120.00,40.00)}{0.80mm}
\pgfcircle[stroke]{\pgfxy(120.00,40.00)}{0.80mm}
\pgfmoveto{\pgfxy(120.00,40.00)}\pgflineto{\pgfxy(130.00,50.00)}\pgfstroke
\pgfcircle[fill]{\pgfxy(130.00,50.00)}{0.80mm}
\pgfcircle[stroke]{\pgfxy(130.00,50.00)}{0.80mm}
\pgfcircle[fill]{\pgfxy(120.00,60.00)}{0.80mm}
\pgfcircle[stroke]{\pgfxy(120.00,60.00)}{0.80mm}
\pgfmoveto{\pgfxy(120.00,50.00)}\pgflineto{\pgfxy(120.00,60.00)}\pgfstroke
\pgfmoveto{\pgfxy(130.00,60.00)}\pgflineto{\pgfxy(120.00,50.00)}\pgfstroke
\pgfellipse[stroke]{\pgfxy(130.31,67.33)}{\pgfxy(5.00,0.00)}{\pgfxy(0.00,7.51)}
\pgfellipse[stroke]{\pgfxy(74.59,53.36)}{\pgfxy(6.50,0.00)}{\pgfxy(0.00,4.51)}
\pgfputat{\pgfxy(130.25,66.53)}{\pgfbox[bottom,left]{\fontsize{9.10}{10.93}\selectfont \makebox[0pt]{$A$}}}
\pgfputat{\pgfxy(134.21,52.55)}{\pgfbox[bottom,left]{\fontsize{9.10}{10.93}\selectfont \makebox[0pt]{$B$}}}
\pgfellipse[stroke]{\pgfxy(134.52,53.17)}{\pgfxy(6.50,0.00)}{\pgfxy(0.00,4.51)}
\pgfputat{\pgfxy(130.89,44.75)}{\pgfbox[bottom,left]{\fontsize{11.38}{13.66}\selectfont $i_m$}}
\end{pgfpicture}%
    $$
\item[(iii)] If  $\Delta_{m}=(j_m,i_m)$, where $i_m$ and $j_m$ are siblings in $T_{m+1}$, as in the previous case, $i_{m}$ is at the end of the minimal path. Then,
transform $T_{m+1}$ with the following procedure. If $m_1$ denotes the parent of $i_m$ and $j_m$ in $T$, erase the edge $(m_1,j_m)$, create an edge $(i_m,j_m)$, then if $A$ and $B$ are the two subtrees starting from $j_m$ with $\min(A) < \min(B)$ (eventually $B$ is empty), cut the subtree $A$ from $j_m$ and add it as a direct subtree of $m_1$, cut the subtree $B$ from $j_m$ and add it as a direct subtree of $i_m$. The procedure can be illustrated with the following picture:
    $$
    
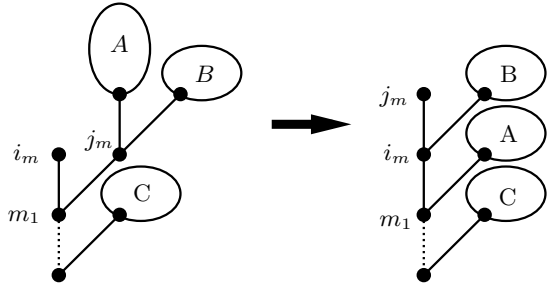
\centering
\begin{pgfpicture}{37.91mm}{29.20mm}{122.00mm}{70.00mm}
\pgfsetxvec{\pgfpoint{0.80mm}{0mm}}
\pgfsetyvec{\pgfpoint{0mm}{0.80mm}}
\color[rgb]{0,0,0}\pgfsetlinewidth{0.30mm}\pgfsetdash{}{0mm}
\pgfsetdash{{0.30mm}{0.50mm}}{0mm}\pgfmoveto{\pgfxy(120.00,50.00)}\pgflineto{\pgfxy(120.00,40.00)}\pgfstroke
\pgfcircle[fill]{\pgfxy(120.00,50.00)}{0.80mm}
\pgfsetdash{}{0mm}\pgfcircle[stroke]{\pgfxy(120.00,50.00)}{0.80mm}
\pgfcircle[fill]{\pgfxy(120.00,40.00)}{0.80mm}
\pgfcircle[stroke]{\pgfxy(120.00,40.00)}{0.80mm}
\pgfmoveto{\pgfxy(120.00,50.00)}\pgflineto{\pgfxy(130.00,60.00)}\pgfstroke
\pgfcircle[fill]{\pgfxy(120.00,70.00)}{0.80mm}
\pgfcircle[stroke]{\pgfxy(120.00,70.00)}{0.80mm}
\pgfmoveto{\pgfxy(120.00,40.00)}\pgflineto{\pgfxy(130.00,50.00)}\pgfstroke
\pgfcircle[fill]{\pgfxy(130.00,50.00)}{0.80mm}
\pgfcircle[stroke]{\pgfxy(130.00,50.00)}{0.80mm}
\pgfputat{\pgfxy(117.90,59.08)}{\pgfbox[bottom,left]{\fontsize{9.10}{10.93}\selectfont \makebox[0pt][r]{$i_m$}}}
\pgfputat{\pgfxy(117.56,68.69)}{\pgfbox[bottom,left]{\fontsize{9.10}{10.93}\selectfont \makebox[0pt][r]{$j_m$}}}
\pgfcircle[fill]{\pgfxy(120.00,60.00)}{0.80mm}
\pgfcircle[stroke]{\pgfxy(120.00,60.00)}{0.80mm}
\pgfmoveto{\pgfxy(120.00,50.00)}\pgflineto{\pgfxy(120.00,60.00)}\pgfstroke
\pgfmoveto{\pgfxy(130.00,70.00)}\pgflineto{\pgfxy(120.00,60.00)}\pgfstroke
\pgfmoveto{\pgfxy(120.00,70.00)}\pgflineto{\pgfxy(120.00,60.00)}\pgfstroke
\pgfellipse[stroke]{\pgfxy(133.50,73.49)}{\pgfxy(6.50,0.00)}{\pgfxy(0.00,4.51)}
\pgfputat{\pgfxy(134.00,72.00)}{\pgfbox[bottom,left]{\fontsize{9.10}{10.93}\selectfont \makebox[0pt]{B}}}
\pgfcircle[fill]{\pgfxy(130.00,70.00)}{0.80mm}
\pgfcircle[stroke]{\pgfxy(130.00,70.00)}{0.80mm}
\pgfputat{\pgfxy(118.11,48.11)}{\pgfbox[bottom,left]{\fontsize{9.10}{10.93}\selectfont \makebox[0pt][r]{$m_1$}}}
\pgfcircle[fill]{\pgfxy(130.00,60.00)}{0.80mm}
\pgfcircle[stroke]{\pgfxy(130.00,60.00)}{0.80mm}
\pgfputat{\pgfxy(134.00,62.00)}{\pgfbox[bottom,left]{\fontsize{9.10}{10.93}\selectfont \makebox[0pt]{A}}}
\pgfellipse[stroke]{\pgfxy(133.50,63.49)}{\pgfxy(6.50,0.00)}{\pgfxy(0.00,4.51)}
\pgfellipse[stroke]{\pgfxy(133.50,53.49)}{\pgfxy(6.50,0.00)}{\pgfxy(0.00,4.51)}
\pgfsetlinewidth{1.20mm}\pgfmoveto{\pgfxy(95.00,65.00)}\pgflineto{\pgfxy(105.00,65.00)}\pgfstroke
\pgfmoveto{\pgfxy(105.00,65.00)}\pgflineto{\pgfxy(102.20,65.70)}\pgflineto{\pgfxy(102.20,64.30)}\pgflineto{\pgfxy(105.00,65.00)}\pgfclosepath\pgffill
\pgfmoveto{\pgfxy(105.00,65.00)}\pgflineto{\pgfxy(102.20,65.70)}\pgflineto{\pgfxy(102.20,64.30)}\pgflineto{\pgfxy(105.00,65.00)}\pgfclosepath\pgfstroke
\pgfputat{\pgfxy(57.00,59.00)}{\pgfbox[bottom,left]{\fontsize{9.10}{10.93}\selectfont \makebox[0pt][r]{$i_m$}}}
\pgfsetdash{{0.30mm}{0.50mm}}{0mm}\pgfsetlinewidth{0.30mm}\pgfmoveto{\pgfxy(60.00,50.00)}\pgflineto{\pgfxy(60.00,40.00)}\pgfstroke
\pgfcircle[fill]{\pgfxy(70.00,60.00)}{0.80mm}
\pgfsetdash{}{0mm}\pgfcircle[stroke]{\pgfxy(70.00,60.00)}{0.80mm}
\pgfcircle[fill]{\pgfxy(60.00,50.00)}{0.80mm}
\pgfcircle[stroke]{\pgfxy(60.00,50.00)}{0.80mm}
\pgfcircle[fill]{\pgfxy(60.00,40.00)}{0.80mm}
\pgfcircle[stroke]{\pgfxy(60.00,40.00)}{0.80mm}
\pgfmoveto{\pgfxy(70.00,60.00)}\pgflineto{\pgfxy(80.00,70.00)}\pgfstroke
\pgfmoveto{\pgfxy(60.00,40.00)}\pgflineto{\pgfxy(70.00,50.00)}\pgfstroke
\pgfcircle[fill]{\pgfxy(70.00,50.00)}{0.80mm}
\pgfcircle[stroke]{\pgfxy(70.00,50.00)}{0.80mm}
\pgfputat{\pgfxy(69.09,61.42)}{\pgfbox[bottom,left]{\fontsize{9.10}{10.93}\selectfont \makebox[0pt][r]{$j_m$}}}
\pgfcircle[fill]{\pgfxy(60.00,60.00)}{0.80mm}
\pgfcircle[stroke]{\pgfxy(60.00,60.00)}{0.80mm}
\pgfmoveto{\pgfxy(60.00,50.00)}\pgflineto{\pgfxy(60.00,60.00)}\pgfstroke
\pgfputat{\pgfxy(57.00,49.00)}{\pgfbox[bottom,left]{\fontsize{9.10}{10.93}\selectfont \makebox[0pt][r]{$m_1$}}}
\pgfmoveto{\pgfxy(70.00,60.00)}\pgflineto{\pgfxy(60.00,50.00)}\pgfstroke
\pgfmoveto{\pgfxy(70.00,70.00)}\pgflineto{\pgfxy(70.00,60.00)}\pgfstroke
\pgfellipse[stroke]{\pgfxy(70.00,77.49)}{\pgfxy(5.00,0.00)}{\pgfxy(0.00,7.51)}
\pgfputat{\pgfxy(70.00,77.00)}{\pgfbox[bottom,left]{\fontsize{9.10}{10.93}\selectfont \makebox[0pt]{$A$}}}
\pgfcircle[fill]{\pgfxy(70.00,70.00)}{0.80mm}
\pgfcircle[stroke]{\pgfxy(70.00,70.00)}{0.80mm}
\pgfcircle[fill]{\pgfxy(80.00,70.00)}{0.80mm}
\pgfcircle[stroke]{\pgfxy(80.00,70.00)}{0.80mm}
\pgfputat{\pgfxy(84.00,72.00)}{\pgfbox[bottom,left]{\fontsize{9.10}{10.93}\selectfont \makebox[0pt]{$B$}}}
\pgfellipse[stroke]{\pgfxy(83.50,73.49)}{\pgfxy(6.50,0.00)}{\pgfxy(0.00,4.51)}
\pgfellipse[stroke]{\pgfxy(73.50,53.49)}{\pgfxy(6.50,0.00)}{\pgfxy(0.00,4.51)}
\pgfputat{\pgfxy(72.23,52.33)}{\pgfbox[bottom,left]{\fontsize{9.10}{10.93}\selectfont C}}
\pgfputat{\pgfxy(132.34,51.67)}{\pgfbox[bottom,left]{\fontsize{9.10}{10.93}\selectfont C}}
\end{pgfpicture}%
    $$

\end{itemize}

Let $\varphi(\Delta):=T_{1}$, which is an element in $\mathcal{BT}_{n,k}$.


\begin{thm} \label{thmvarphi}
For all $n \geq 1$ and $k \in [n]$, the mapping $\varphi: \ESnk\longrightarrow \Tnk$  is a bijection.
\end{thm}

\begin{proof}
It is sufficient to construct the inverse mapping of $\varphi$ to show that this is a bijection.
Given $T$ an increasing tree on the ordered set $\{a_1,\ldots,a_n \}$ with $a_1<\cdots<a_n$, such that $p(T)=a_k$ (that can be interpreted by an element of $\Tnk$), we construct an encoding sequence $\Delta=\varphi^{-1}(T)$ of $[n]$ recursively as follows:
 \begin{itemize}
\item[(a)] If $a_{k-1}$ is the parent of $a_{k}$ in $T$,
then let $m$ ($m >{a_k}$) be the other child of $a_{k-1}$ ($m=\infty$ if $a_k$ is the only child of $a_{k-1}$) and $s$ ($s>k$) be a sibling of
$a_{k-1}$ ($s=\infty$ if $a_{k-1}$ has no sibling), and $j$ the parent of $a_{k-1}$ in $T$.
\begin{itemize}
\item[(a1)] 
 If $m<\infty$ and $m<s$, then define $\varphi^{-1}(T)=\left((a_k,a_{k-1})^{*} ,\varphi^{-1}(T')\right)$, where $T'$ is the tree obtained from $T$ by deleting the vertices $a_{k-1}$, $a_{k}$ and their adjacent edges in $T$, and adding a new edge between $m$ and $j$.
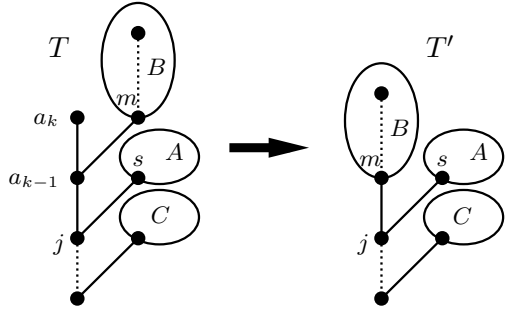
    $$
    \centering
\begin{pgfpicture}{38.00mm}{29.20mm}{122.00mm}{74.00mm}
\pgfsetxvec{\pgfpoint{0.80mm}{0mm}}
\pgfsetyvec{\pgfpoint{0mm}{0.80mm}}
\color[rgb]{0,0,0}\pgfsetlinewidth{0.30mm}\pgfsetdash{}{0mm}
\pgfputat{\pgfxy(67.00,80.00)}{\pgfbox[bottom,left]{\fontsize{11.38}{13.66}\selectfont \makebox[0pt]{$T$}}}
\pgfsetdash{{0.30mm}{0.50mm}}{0mm}\pgfmoveto{\pgfxy(70.00,50.00)}\pgflineto{\pgfxy(70.00,40.00)}\pgfstroke
\pgfcircle[fill]{\pgfxy(70.00,50.00)}{0.80mm}
\pgfsetdash{}{0mm}\pgfcircle[stroke]{\pgfxy(70.00,50.00)}{0.80mm}
\pgfcircle[fill]{\pgfxy(70.00,40.00)}{0.80mm}
\pgfcircle[stroke]{\pgfxy(70.00,40.00)}{0.80mm}
\pgfmoveto{\pgfxy(70.00,50.00)}\pgflineto{\pgfxy(80.00,60.00)}\pgfstroke
\pgfcircle[fill]{\pgfxy(70.00,70.00)}{0.80mm}
\pgfcircle[stroke]{\pgfxy(70.00,70.00)}{0.80mm}
\pgfmoveto{\pgfxy(70.00,40.00)}\pgflineto{\pgfxy(80.00,50.00)}\pgfstroke
\pgfcircle[fill]{\pgfxy(80.00,50.00)}{0.80mm}
\pgfcircle[stroke]{\pgfxy(80.00,50.00)}{0.80mm}
\pgfputat{\pgfxy(67.00,69.00)}{\pgfbox[bottom,left]{\fontsize{9.10}{10.93}\selectfont \makebox[0pt][r]{$a_k$}}}
\pgfputat{\pgfxy(67.00,59.00)}{\pgfbox[bottom,left]{\fontsize{9.10}{10.93}\selectfont \makebox[0pt][r]{$a_{k-1}$}}}
\pgfcircle[fill]{\pgfxy(70.00,60.00)}{0.80mm}
\pgfcircle[stroke]{\pgfxy(70.00,60.00)}{0.80mm}
\pgfmoveto{\pgfxy(70.00,50.00)}\pgflineto{\pgfxy(70.00,60.00)}\pgfstroke
\pgfmoveto{\pgfxy(80.00,70.00)}\pgflineto{\pgfxy(70.00,60.00)}\pgfstroke
\pgfmoveto{\pgfxy(70.00,70.00)}\pgflineto{\pgfxy(70.00,60.00)}\pgfstroke
\pgfellipse[stroke]{\pgfxy(80.00,79.49)}{\pgfxy(6.00,0.00)}{\pgfxy(0.00,9.51)}
\pgfputat{\pgfxy(83.00,77.00)}{\pgfbox[bottom,left]{\fontsize{9.10}{10.93}\selectfont \makebox[0pt]{$B$}}}
\pgfcircle[fill]{\pgfxy(80.00,70.00)}{0.80mm}
\pgfcircle[stroke]{\pgfxy(80.00,70.00)}{0.80mm}
\pgfputat{\pgfxy(80.00,72.00)}{\pgfbox[bottom,left]{\fontsize{9.10}{10.93}\selectfont \makebox[0pt][r]{$m$}}}
\pgfputat{\pgfxy(68.00,48.00)}{\pgfbox[bottom,left]{\fontsize{9.10}{10.93}\selectfont \makebox[0pt][r]{$j$}}}
\pgfcircle[fill]{\pgfxy(80.00,60.00)}{0.80mm}
\pgfcircle[stroke]{\pgfxy(80.00,60.00)}{0.80mm}
\pgfputat{\pgfxy(86.00,63.00)}{\pgfbox[bottom,left]{\fontsize{9.10}{10.93}\selectfont \makebox[0pt]{$A$}}}
\pgfputat{\pgfxy(80.00,62.00)}{\pgfbox[bottom,left]{\fontsize{9.10}{10.93}\selectfont \makebox[0pt]{$s$}}}
\pgfellipse[stroke]{\pgfxy(83.50,63.49)}{\pgfxy(6.50,0.00)}{\pgfxy(0.00,4.51)}
\pgfellipse[stroke]{\pgfxy(83.50,53.49)}{\pgfxy(6.50,0.00)}{\pgfxy(0.00,4.51)}
\pgfsetlinewidth{1.20mm}\pgfmoveto{\pgfxy(95.00,65.00)}\pgflineto{\pgfxy(105.00,65.00)}\pgfstroke
\pgfmoveto{\pgfxy(105.00,65.00)}\pgflineto{\pgfxy(102.20,65.70)}\pgflineto{\pgfxy(102.20,64.30)}\pgflineto{\pgfxy(105.00,65.00)}\pgfclosepath\pgffill
\pgfmoveto{\pgfxy(105.00,65.00)}\pgflineto{\pgfxy(102.20,65.70)}\pgflineto{\pgfxy(102.20,64.30)}\pgflineto{\pgfxy(105.00,65.00)}\pgfclosepath\pgfstroke
\pgfputat{\pgfxy(129.71,80.00)}{\pgfbox[bottom,left]{\fontsize{11.38}{13.66}\selectfont \makebox[0pt]{$T'$}}}
\pgfcircle[fill]{\pgfxy(80.00,84.00)}{0.80mm}
\pgfsetlinewidth{0.30mm}\pgfcircle[stroke]{\pgfxy(80.00,84.00)}{0.80mm}
\pgfsetdash{{0.30mm}{0.50mm}}{0mm}\pgfmoveto{\pgfxy(80.00,84.00)}\pgflineto{\pgfxy(80.00,70.00)}\pgfstroke
\pgfmoveto{\pgfxy(120.00,50.00)}\pgflineto{\pgfxy(120.00,40.00)}\pgfstroke
\pgfcircle[fill]{\pgfxy(120.00,50.00)}{0.80mm}
\pgfsetdash{}{0mm}\pgfcircle[stroke]{\pgfxy(120.00,50.00)}{0.80mm}
\pgfcircle[fill]{\pgfxy(120.00,40.00)}{0.80mm}
\pgfcircle[stroke]{\pgfxy(120.00,40.00)}{0.80mm}
\pgfmoveto{\pgfxy(120.00,50.00)}\pgflineto{\pgfxy(130.00,60.00)}\pgfstroke
\pgfmoveto{\pgfxy(120.00,40.00)}\pgflineto{\pgfxy(130.00,50.00)}\pgfstroke
\pgfcircle[fill]{\pgfxy(130.00,50.00)}{0.80mm}
\pgfcircle[stroke]{\pgfxy(130.00,50.00)}{0.80mm}
\pgfmoveto{\pgfxy(120.00,50.00)}\pgflineto{\pgfxy(120.00,60.00)}\pgfstroke
\pgfellipse[stroke]{\pgfxy(120.00,69.49)}{\pgfxy(6.00,0.00)}{\pgfxy(0.00,9.51)}
\pgfputat{\pgfxy(123.00,67.00)}{\pgfbox[bottom,left]{\fontsize{9.10}{10.93}\selectfont \makebox[0pt]{$B$}}}
\pgfcircle[fill]{\pgfxy(120.00,60.00)}{0.80mm}
\pgfcircle[stroke]{\pgfxy(120.00,60.00)}{0.80mm}
\pgfputat{\pgfxy(120.00,62.00)}{\pgfbox[bottom,left]{\fontsize{9.10}{10.93}\selectfont \makebox[0pt][r]{$m$}}}
\pgfputat{\pgfxy(118.00,48.00)}{\pgfbox[bottom,left]{\fontsize{9.10}{10.93}\selectfont \makebox[0pt][r]{$j$}}}
\pgfcircle[fill]{\pgfxy(130.00,60.00)}{0.80mm}
\pgfcircle[stroke]{\pgfxy(130.00,60.00)}{0.80mm}
\pgfputat{\pgfxy(136.00,63.00)}{\pgfbox[bottom,left]{\fontsize{9.10}{10.93}\selectfont \makebox[0pt]{$A$}}}
\pgfputat{\pgfxy(130.00,62.00)}{\pgfbox[bottom,left]{\fontsize{9.10}{10.93}\selectfont \makebox[0pt]{$s$}}}
\pgfellipse[stroke]{\pgfxy(133.50,63.49)}{\pgfxy(6.50,0.00)}{\pgfxy(0.00,4.51)}
\pgfellipse[stroke]{\pgfxy(133.50,53.49)}{\pgfxy(6.50,0.00)}{\pgfxy(0.00,4.51)}
\pgfcircle[fill]{\pgfxy(120.00,74.00)}{0.80mm}
\pgfcircle[stroke]{\pgfxy(120.00,74.00)}{0.80mm}
\pgfsetdash{{0.30mm}{0.50mm}}{0mm}\pgfmoveto{\pgfxy(120.00,74.00)}\pgflineto{\pgfxy(120.00,60.00)}\pgfstroke
\pgfputat{\pgfxy(83.69,52.47)}{\pgfbox[bottom,left]{\fontsize{9.10}{10.93}\selectfont \makebox[0pt]{$C$}}}
\pgfputat{\pgfxy(133.30,52.33)}{\pgfbox[bottom,left]{\fontsize{9.10}{10.93}\selectfont \makebox[0pt]{$C$}}}
\end{pgfpicture}%
$$
\item[(a2)] In the other cases ($m=\infty$ or $m>s$), then define $\varphi^{-1}(T)=\left((a_{k},a_{k-1}),\varphi^{-1}(T')\right)$, where $T'$ is the tree obtained from $T$ by erasing the edges $(a_{k-1},a_k)$, $(a_{k-1},m)$ and $(j,s)$ in $T$, and adding the edges $(j,a_k)$, $(a_k,s)$, $(a_k,m)$. The procedure can be illustrated with the following picture:
    $$
    
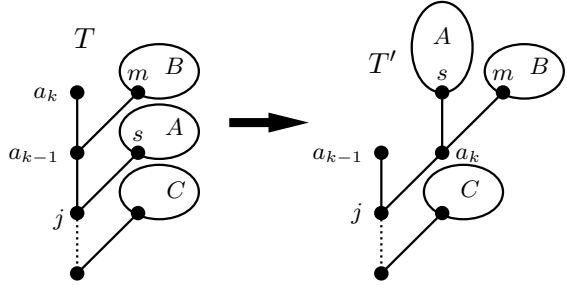
\centering
\begin{pgfpicture}{44.31mm}{29.20mm}{122.00mm}{70.00mm}
\pgfsetxvec{\pgfpoint{0.80mm}{0mm}}
\pgfsetyvec{\pgfpoint{0mm}{0.80mm}}
\color[rgb]{0,0,0}\pgfsetlinewidth{0.30mm}\pgfsetdash{}{0mm}
\pgfputat{\pgfxy(71.26,77.35)}{\pgfbox[bottom,left]{\fontsize{11.38}{13.66}\selectfont \makebox[0pt]{$T$}}}
\pgfsetdash{{0.30mm}{0.50mm}}{0mm}\pgfmoveto{\pgfxy(70.00,50.00)}\pgflineto{\pgfxy(70.00,40.00)}\pgfstroke
\pgfcircle[fill]{\pgfxy(70.00,50.00)}{0.80mm}
\pgfsetdash{}{0mm}\pgfcircle[stroke]{\pgfxy(70.00,50.00)}{0.80mm}
\pgfcircle[fill]{\pgfxy(70.00,40.00)}{0.80mm}
\pgfcircle[stroke]{\pgfxy(70.00,40.00)}{0.80mm}
\pgfmoveto{\pgfxy(70.00,50.00)}\pgflineto{\pgfxy(80.00,60.00)}\pgfstroke
\pgfcircle[fill]{\pgfxy(70.00,70.00)}{0.80mm}
\pgfcircle[stroke]{\pgfxy(70.00,70.00)}{0.80mm}
\pgfmoveto{\pgfxy(70.00,40.00)}\pgflineto{\pgfxy(80.00,50.00)}\pgfstroke
\pgfcircle[fill]{\pgfxy(80.00,50.00)}{0.80mm}
\pgfcircle[stroke]{\pgfxy(80.00,50.00)}{0.80mm}
\pgfputat{\pgfxy(67.00,69.00)}{\pgfbox[bottom,left]{\fontsize{9.10}{10.93}\selectfont \makebox[0pt][r]{$a_k$}}}
\pgfputat{\pgfxy(67.00,59.00)}{\pgfbox[bottom,left]{\fontsize{9.10}{10.93}\selectfont \makebox[0pt][r]{$a_{k-1}$}}}
\pgfcircle[fill]{\pgfxy(70.00,60.00)}{0.80mm}
\pgfcircle[stroke]{\pgfxy(70.00,60.00)}{0.80mm}
\pgfmoveto{\pgfxy(70.00,50.00)}\pgflineto{\pgfxy(70.00,60.00)}\pgfstroke
\pgfmoveto{\pgfxy(80.00,70.00)}\pgflineto{\pgfxy(70.00,60.00)}\pgfstroke
\pgfmoveto{\pgfxy(70.00,70.00)}\pgflineto{\pgfxy(70.00,60.00)}\pgfstroke
\pgfellipse[stroke]{\pgfxy(83.50,73.49)}{\pgfxy(6.50,0.00)}{\pgfxy(0.00,4.51)}
\pgfputat{\pgfxy(86.00,73.00)}{\pgfbox[bottom,left]{\fontsize{9.10}{10.93}\selectfont \makebox[0pt]{$B$}}}
\pgfcircle[fill]{\pgfxy(80.00,70.00)}{0.80mm}
\pgfcircle[stroke]{\pgfxy(80.00,70.00)}{0.80mm}
\pgfputat{\pgfxy(80.00,72.00)}{\pgfbox[bottom,left]{\fontsize{9.10}{10.93}\selectfont \makebox[0pt]{$m$}}}
\pgfputat{\pgfxy(68.00,48.00)}{\pgfbox[bottom,left]{\fontsize{9.10}{10.93}\selectfont \makebox[0pt][r]{$j$}}}
\pgfcircle[fill]{\pgfxy(80.00,60.00)}{0.80mm}
\pgfcircle[stroke]{\pgfxy(80.00,60.00)}{0.80mm}
\pgfputat{\pgfxy(86.00,63.00)}{\pgfbox[bottom,left]{\fontsize{9.10}{10.93}\selectfont \makebox[0pt]{$A$}}}
\pgfputat{\pgfxy(80.00,62.00)}{\pgfbox[bottom,left]{\fontsize{9.10}{10.93}\selectfont \makebox[0pt]{$s$}}}
\pgfellipse[stroke]{\pgfxy(83.50,63.49)}{\pgfxy(6.50,0.00)}{\pgfxy(0.00,4.51)}
\pgfellipse[stroke]{\pgfxy(83.50,53.49)}{\pgfxy(6.50,0.00)}{\pgfxy(0.00,4.51)}
\pgfsetlinewidth{1.20mm}\pgfmoveto{\pgfxy(95.00,65.00)}\pgflineto{\pgfxy(105.00,65.00)}\pgfstroke
\pgfmoveto{\pgfxy(105.00,65.00)}\pgflineto{\pgfxy(102.20,65.70)}\pgflineto{\pgfxy(102.20,64.30)}\pgflineto{\pgfxy(105.00,65.00)}\pgfclosepath\pgffill
\pgfmoveto{\pgfxy(105.00,65.00)}\pgflineto{\pgfxy(102.20,65.70)}\pgflineto{\pgfxy(102.20,64.30)}\pgflineto{\pgfxy(105.00,65.00)}\pgfclosepath\pgfstroke
\pgfputat{\pgfxy(120.27,73.25)}{\pgfbox[bottom,left]{\fontsize{11.38}{13.66}\selectfont \makebox[0pt]{$T'$}}}
\pgfputat{\pgfxy(117.00,59.00)}{\pgfbox[bottom,left]{\fontsize{9.10}{10.93}\selectfont \makebox[0pt][r]{$a_{k-1}$}}}
\pgfsetdash{{0.30mm}{0.50mm}}{0mm}\pgfsetlinewidth{0.30mm}\pgfmoveto{\pgfxy(120.00,50.00)}\pgflineto{\pgfxy(120.00,40.00)}\pgfstroke
\pgfcircle[fill]{\pgfxy(130.00,60.00)}{0.80mm}
\pgfsetdash{}{0mm}\pgfcircle[stroke]{\pgfxy(130.00,60.00)}{0.80mm}
\pgfcircle[fill]{\pgfxy(120.00,50.00)}{0.80mm}
\pgfcircle[stroke]{\pgfxy(120.00,50.00)}{0.80mm}
\pgfcircle[fill]{\pgfxy(120.00,40.00)}{0.80mm}
\pgfcircle[stroke]{\pgfxy(120.00,40.00)}{0.80mm}
\pgfmoveto{\pgfxy(130.00,60.00)}\pgflineto{\pgfxy(140.00,70.00)}\pgfstroke
\pgfmoveto{\pgfxy(120.00,40.00)}\pgflineto{\pgfxy(130.00,50.00)}\pgfstroke
\pgfcircle[fill]{\pgfxy(130.00,50.00)}{0.80mm}
\pgfcircle[stroke]{\pgfxy(130.00,50.00)}{0.80mm}
\pgfputat{\pgfxy(132.21,59.00)}{\pgfbox[bottom,left]{\fontsize{9.10}{10.93}\selectfont $a_k$}}
\pgfcircle[fill]{\pgfxy(120.00,60.00)}{0.80mm}
\pgfcircle[stroke]{\pgfxy(120.00,60.00)}{0.80mm}
\pgfmoveto{\pgfxy(120.00,50.00)}\pgflineto{\pgfxy(120.00,60.00)}\pgfstroke
\pgfputat{\pgfxy(117.00,49.00)}{\pgfbox[bottom,left]{\fontsize{9.10}{10.93}\selectfont \makebox[0pt][r]{$j$}}}
\pgfmoveto{\pgfxy(130.00,60.00)}\pgflineto{\pgfxy(120.00,50.00)}\pgfstroke
\pgfmoveto{\pgfxy(130.00,70.00)}\pgflineto{\pgfxy(130.00,60.00)}\pgfstroke
\pgfellipse[stroke]{\pgfxy(130.00,77.49)}{\pgfxy(5.00,0.00)}{\pgfxy(0.00,7.51)}
\pgfputat{\pgfxy(130.00,78.00)}{\pgfbox[bottom,left]{\fontsize{9.10}{10.93}\selectfont \makebox[0pt]{$A$}}}
\pgfcircle[fill]{\pgfxy(130.00,70.00)}{0.80mm}
\pgfcircle[stroke]{\pgfxy(130.00,70.00)}{0.80mm}
\pgfputat{\pgfxy(140.00,72.00)}{\pgfbox[bottom,left]{\fontsize{9.10}{10.93}\selectfont \makebox[0pt]{$m$}}}
\pgfcircle[fill]{\pgfxy(140.00,70.00)}{0.80mm}
\pgfcircle[stroke]{\pgfxy(140.00,70.00)}{0.80mm}
\pgfputat{\pgfxy(146.00,73.00)}{\pgfbox[bottom,left]{\fontsize{9.10}{10.93}\selectfont \makebox[0pt]{$B$}}}
\pgfputat{\pgfxy(130.00,72.00)}{\pgfbox[bottom,left]{\fontsize{9.10}{10.93}\selectfont \makebox[0pt]{$s$}}}
\pgfellipse[stroke]{\pgfxy(143.50,73.49)}{\pgfxy(6.50,0.00)}{\pgfxy(0.00,4.51)}
\pgfellipse[stroke]{\pgfxy(133.50,53.49)}{\pgfxy(6.50,0.00)}{\pgfxy(0.00,4.51)}
\pgfputat{\pgfxy(86.26,52.54)}{\pgfbox[bottom,left]{\fontsize{9.10}{10.93}\selectfont \makebox[0pt]{$C$}}}
\pgfputat{\pgfxy(134.59,52.67)}{\pgfbox[bottom,left]{\fontsize{9.10}{10.93}\selectfont \makebox[0pt]{$C$}}}
\end{pgfpicture}%
    $$
    \end{itemize}
\item[(b)] If $a_{k-1}$ is not the parent of $a_k$ in $T$, then define $\varphi^{-1}(T)=\left((a_k,a_{k-1}),\varphi^{-1}(T')\right)$, where $T'$ is the tree obtained from $T$ by exchanging the labels $a_{k-1}$ and $a_k$ in $T$. 
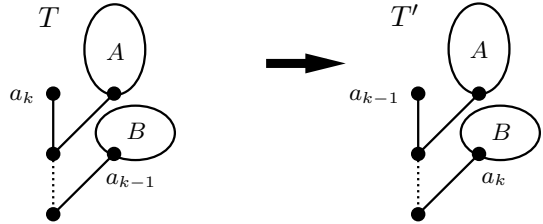
\label{case:B}
    $$
    \centering
\begin{pgfpicture}{38.98mm}{29.20mm}{114.00mm}{62.05mm}
\pgfsetxvec{\pgfpoint{0.80mm}{0mm}}
\pgfsetyvec{\pgfpoint{0mm}{0.80mm}}
\color[rgb]{0,0,0}\pgfsetlinewidth{0.30mm}\pgfsetdash{}{0mm}
\pgfsetlinewidth{1.20mm}\pgfmoveto{\pgfxy(95.00,65.00)}\pgflineto{\pgfxy(105.00,65.00)}\pgfstroke
\pgfmoveto{\pgfxy(105.00,65.00)}\pgflineto{\pgfxy(102.20,65.70)}\pgflineto{\pgfxy(102.20,64.30)}\pgflineto{\pgfxy(105.00,65.00)}\pgfclosepath\pgffill
\pgfmoveto{\pgfxy(105.00,65.00)}\pgflineto{\pgfxy(102.20,65.70)}\pgflineto{\pgfxy(102.20,64.30)}\pgflineto{\pgfxy(105.00,65.00)}\pgfclosepath\pgfstroke
\pgfputat{\pgfxy(117.00,59.00)}{\pgfbox[bottom,left]{\fontsize{9.10}{10.93}\selectfont \makebox[0pt][r]{$a_{k-1}$}}}
\pgfsetdash{{0.30mm}{0.50mm}}{0mm}\pgfsetlinewidth{0.30mm}\pgfmoveto{\pgfxy(120.00,50.00)}\pgflineto{\pgfxy(120.00,40.00)}\pgfstroke
\pgfcircle[fill]{\pgfxy(130.00,60.00)}{0.80mm}
\pgfsetdash{}{0mm}\pgfcircle[stroke]{\pgfxy(130.00,60.00)}{0.80mm}
\pgfcircle[fill]{\pgfxy(120.00,50.00)}{0.80mm}
\pgfcircle[stroke]{\pgfxy(120.00,50.00)}{0.80mm}
\pgfcircle[fill]{\pgfxy(120.00,40.00)}{0.80mm}
\pgfcircle[stroke]{\pgfxy(120.00,40.00)}{0.80mm}
\pgfmoveto{\pgfxy(120.00,40.00)}\pgflineto{\pgfxy(130.00,50.00)}\pgfstroke
\pgfcircle[fill]{\pgfxy(130.00,50.00)}{0.80mm}
\pgfcircle[stroke]{\pgfxy(130.00,50.00)}{0.80mm}
\pgfputat{\pgfxy(130.41,45.25)}{\pgfbox[bottom,left]{\fontsize{9.10}{10.93}\selectfont $a_k$}}
\pgfcircle[fill]{\pgfxy(120.00,60.00)}{0.80mm}
\pgfcircle[stroke]{\pgfxy(120.00,60.00)}{0.80mm}
\pgfmoveto{\pgfxy(120.00,50.00)}\pgflineto{\pgfxy(120.00,60.00)}\pgfstroke
\pgfmoveto{\pgfxy(130.00,60.00)}\pgflineto{\pgfxy(120.00,50.00)}\pgfstroke
\pgfellipse[stroke]{\pgfxy(130.00,67.42)}{\pgfxy(5.00,0.00)}{\pgfxy(0.00,7.51)}
\pgfputat{\pgfxy(130.00,65.88)}{\pgfbox[bottom,left]{\fontsize{9.10}{10.93}\selectfont \makebox[0pt]{$A$}}}
\pgfputat{\pgfxy(133.67,52.27)}{\pgfbox[bottom,left]{\fontsize{9.10}{10.93}\selectfont \makebox[0pt]{$B$}}}
\pgfellipse[stroke]{\pgfxy(133.50,53.49)}{\pgfxy(6.50,0.00)}{\pgfxy(0.00,4.51)}
\pgfputat{\pgfxy(57.00,59.00)}{\pgfbox[bottom,left]{\fontsize{9.10}{10.93}\selectfont \makebox[0pt][r]{$a_k$}}}
\pgfsetdash{{0.30mm}{0.50mm}}{0mm}\pgfmoveto{\pgfxy(60.00,50.00)}\pgflineto{\pgfxy(60.00,40.00)}\pgfstroke
\pgfcircle[fill]{\pgfxy(70.00,60.00)}{0.80mm}
\pgfsetdash{}{0mm}\pgfcircle[stroke]{\pgfxy(70.00,60.00)}{0.80mm}
\pgfcircle[fill]{\pgfxy(60.00,50.00)}{0.80mm}
\pgfcircle[stroke]{\pgfxy(60.00,50.00)}{0.80mm}
\pgfcircle[fill]{\pgfxy(60.00,40.00)}{0.80mm}
\pgfcircle[stroke]{\pgfxy(60.00,40.00)}{0.80mm}
\pgfmoveto{\pgfxy(60.00,40.00)}\pgflineto{\pgfxy(70.00,50.00)}\pgfstroke
\pgfcircle[fill]{\pgfxy(70.00,50.00)}{0.80mm}
\pgfcircle[stroke]{\pgfxy(70.00,50.00)}{0.80mm}
\pgfputat{\pgfxy(72.81,44.98)}{\pgfbox[bottom,left]{\fontsize{9.10}{10.93}\selectfont \makebox[0pt]{$a_{k-1}$}}}
\pgfcircle[fill]{\pgfxy(60.00,60.00)}{0.80mm}
\pgfcircle[stroke]{\pgfxy(60.00,60.00)}{0.80mm}
\pgfmoveto{\pgfxy(60.00,50.00)}\pgflineto{\pgfxy(60.00,60.00)}\pgfstroke
\pgfmoveto{\pgfxy(70.00,60.00)}\pgflineto{\pgfxy(60.00,50.00)}\pgfstroke
\pgfellipse[stroke]{\pgfxy(70.13,67.29)}{\pgfxy(5.00,0.00)}{\pgfxy(0.00,7.51)}
\pgfellipse[stroke]{\pgfxy(73.50,53.49)}{\pgfxy(6.50,0.00)}{\pgfxy(0.00,4.51)}
\pgfputat{\pgfxy(70.26,65.61)}{\pgfbox[bottom,left]{\fontsize{9.10}{10.93}\selectfont \makebox[0pt]{$A$}}}
\pgfputat{\pgfxy(73.67,52.54)}{\pgfbox[bottom,left]{\fontsize{9.10}{10.93}\selectfont \makebox[0pt]{$B$}}}
\pgfputat{\pgfxy(59.35,71.05)}{\pgfbox[bottom,left]{\fontsize{11.38}{13.66}\selectfont \makebox[0pt]{$T$}}}
\pgfputat{\pgfxy(117.87,71.14)}{\pgfbox[bottom,left]{\fontsize{11.38}{13.66}\selectfont \makebox[0pt]{$T'$}}}
\end{pgfpicture}%
$$
   \end{itemize}
   Note that cases (a1), (a2) and (b) in the construction of $\varphi^{-1}$ correspond, respectively, to cases (i), (ii) and (iii) of the construction of $\varphi$.
   
   It remains to prove that the obtained sequence $\Delta$ verifies the points (i)-(iii) of Definition~\ref{def:enc}.
\begin{itemize}
\item[$\bullet$] It is easily seen that each integer of $[n]$ is removed once off $T$. So (i) is verified.
\item[$\bullet$] If an element $(j,i)^{*}$ appears in $\Delta$, that corresponds to the case (a1), when we delete the vertices $i$ and $j$ from the tree $T$. Then the next elements in $\Delta$ don't contain either $i$ or $j$ since they correspond to $\varphi^{-1}(T')$. Moreover, if we are in the case (a1), the minimal path in the tree $T'$ contains at least one element $m$ with $m > j> i$, so the next element in $\Delta$ must be $(m,k)$ with $m>k$. Thus  (ii) is verified.
\item[$\bullet$] If an element $(j,i)$ appears in $\Delta$, in both Case (a2) or Case (b), the tree $T'$ has $i$ as the leaf of the minimal path. Then, the next element in $\Delta$ must be $(i,k)$ with $i>k$. Moreover, $i$ and $j$ must be consecutive elements in the ordered set of labels in $T$. Then the elements $\ell$ such that $i<\ell<j$ don't appear in $T$. Thus (iii) is  verified.
\end{itemize}
\end{proof}

Let $\Psi=\varphi\circ \psi$. Then $\Psi: \Ank\to \Tnk$ is a bijection satisfying
$ \pi_1=p(\Psi(\pi))$ for all $ \pi \in \Ank$. Thus Theorem~\ref{mainth} is proved.

\begin{ex} Continuing the Example~\ref{ex2.2}, we apply $\Psi$ to $\pi$ by using  the known LR-code  of $\pi=7\,4\,8\,5\,9\,1\,6\,2\,3$. The details are
given in Figure~\ref{fig1}.
\begin{figure}[!h]
\begin{tabular}{c||c|c|c|c|c|c}
$\pi^{(m)}$ & $9$ & $839$ & $938$ & $21938$ & $31928$ & $81923$
\\
\hline
$\textsc{First}(\pi^{(m)})$ & 9 & 8 & 9 & 2 & 3 & 8
\\
\hline
$\Delta_m$\rule[-7pt]{0pt}{20pt} & $(9)^{*} $ & $(8,3)^{*} $ & $(9,8)$ & $(2,1)^{*} $ & $(3,2)$ & $(8,3)$
\\
\hline
$
\centering
\begin{pgfpicture}{-2.00mm}{-2.01mm}{11.51mm}{15.98mm}
\pgfsetxvec{\pgfpoint{0.70mm}{0mm}}
\pgfsetyvec{\pgfpoint{0mm}{0.70mm}}
\color[rgb]{0,0,0}\pgfsetlinewidth{0.30mm}\pgfsetdash{}{0mm}
\pgfputat{\pgfxy(7.00,9.00)}{\pgfbox[bottom,left]{\fontsize{11}{9.56}\selectfont \makebox[0pt]{$T_{m}=\Psi(\pi^{(m)})$}}}
\end{pgfpicture}%
$ &
$
\centering
\begin{pgfpicture}{1.50mm}{3.70mm}{9.70mm}{10.50mm}
\pgfsetxvec{\pgfpoint{0.70mm}{0mm}}
\pgfsetyvec{\pgfpoint{0mm}{0.70mm}}
\color[rgb]{0,0,0}\pgfsetlinewidth{0.30mm}\pgfsetdash{}{0mm}
\pgfcircle[fill]{\pgfxy(10.00,10.00)}{0.70mm}
\pgfcircle[stroke]{\pgfxy(10.00,10.00)}{0.70mm}
\pgfputat{\pgfxy(7.00,9.00)}{\pgfbox[bottom,left]{\fontsize{7.97}{9.56}\selectfont \makebox[0pt][r]{$9$}}}
\end{pgfpicture}%
$ &
$
\centering
\begin{pgfpicture}{-2.00mm}{3.70mm}{12.50mm}{14.00mm}
\pgfsetxvec{\pgfpoint{0.70mm}{0mm}}
\pgfsetyvec{\pgfpoint{0mm}{0.70mm}}
\color[rgb]{0,0,0}\pgfsetlinewidth{0.30mm}\pgfsetdash{}{0mm}
\pgfcircle[fill]{\pgfxy(5.00,10.00)}{0.70mm}
\pgfcircle[stroke]{\pgfxy(5.00,10.00)}{0.70mm}
\pgfputat{\pgfxy(2.00,9.00)}{\pgfbox[bottom,left]{\fontsize{7.97}{9.56}\selectfont \makebox[0pt][r]{$3$}}}
\pgfmoveto{\pgfxy(5.00,15.00)}\pgflineto{\pgfxy(5.00,10.00)}\pgfstroke
\pgfcircle[fill]{\pgfxy(5.00,15.00)}{0.70mm}
\pgfcircle[stroke]{\pgfxy(5.00,15.00)}{0.70mm}
\pgfcircle[fill]{\pgfxy(10.00,15.00)}{0.70mm}
\pgfcircle[stroke]{\pgfxy(10.00,15.00)}{0.70mm}
\pgfmoveto{\pgfxy(10.00,15.00)}\pgflineto{\pgfxy(5.00,10.00)}\pgfstroke
\pgfputat{\pgfxy(13.00,14.00)}{\pgfbox[bottom,left]{\fontsize{7.97}{9.56}\selectfont $9$}}
\pgfputat{\pgfxy(2.00,14.00)}{\pgfbox[bottom,left]{\fontsize{7.97}{9.56}\selectfont \makebox[0pt][r]{$8$}}}
\end{pgfpicture}%
$ &
$
\centering
\begin{pgfpicture}{4.30mm}{3.70mm}{12.50mm}{17.50mm}
\pgfsetxvec{\pgfpoint{0.70mm}{0mm}}
\pgfsetyvec{\pgfpoint{0mm}{0.70mm}}
\color[rgb]{0,0,0}\pgfsetlinewidth{0.30mm}\pgfsetdash{}{0mm}
\pgfcircle[fill]{\pgfxy(10.00,10.00)}{0.70mm}
\pgfcircle[stroke]{\pgfxy(10.00,10.00)}{0.70mm}
\pgfputat{\pgfxy(13.00,9.00)}{\pgfbox[bottom,left]{\fontsize{7.97}{9.56}\selectfont $3$}}
\pgfmoveto{\pgfxy(10.00,20.00)}\pgflineto{\pgfxy(10.00,15.00)}\pgfstroke
\pgfcircle[fill]{\pgfxy(10.00,20.00)}{0.70mm}
\pgfcircle[stroke]{\pgfxy(10.00,20.00)}{0.70mm}
\pgfcircle[fill]{\pgfxy(10.00,15.00)}{0.70mm}
\pgfcircle[stroke]{\pgfxy(10.00,15.00)}{0.70mm}
\pgfmoveto{\pgfxy(10.00,15.00)}\pgflineto{\pgfxy(10.00,10.00)}\pgfstroke
\pgfputat{\pgfxy(13.00,14.00)}{\pgfbox[bottom,left]{\fontsize{7.97}{9.56}\selectfont $8$}}
\pgfputat{\pgfxy(13.00,19.00)}{\pgfbox[bottom,left]{\fontsize{7.97}{9.56}\selectfont $9$}}
\end{pgfpicture}%
$ &
$
\centering
\begin{pgfpicture}{-2.00mm}{0.20mm}{12.50mm}{17.50mm}
\pgfsetxvec{\pgfpoint{0.70mm}{0mm}}
\pgfsetyvec{\pgfpoint{0mm}{0.70mm}}
\color[rgb]{0,0,0}\pgfsetlinewidth{0.30mm}\pgfsetdash{}{0mm}
\pgfcircle[fill]{\pgfxy(10.00,10.00)}{0.70mm}
\pgfcircle[stroke]{\pgfxy(10.00,10.00)}{0.70mm}
\pgfputat{\pgfxy(2.00,4.00)}{\pgfbox[bottom,left]{\fontsize{7.97}{9.56}\selectfont \makebox[0pt][r]{$1$}}}
\pgfmoveto{\pgfxy(10.00,20.00)}\pgflineto{\pgfxy(10.00,15.00)}\pgfstroke
\pgfcircle[fill]{\pgfxy(10.00,20.00)}{0.70mm}
\pgfcircle[stroke]{\pgfxy(10.00,20.00)}{0.70mm}
\pgfcircle[fill]{\pgfxy(10.00,15.00)}{0.70mm}
\pgfcircle[stroke]{\pgfxy(10.00,15.00)}{0.70mm}
\pgfmoveto{\pgfxy(10.00,15.00)}\pgflineto{\pgfxy(10.00,10.00)}\pgfstroke
\pgfputat{\pgfxy(13.00,19.00)}{\pgfbox[bottom,left]{\fontsize{7.97}{9.56}\selectfont $9$}}
\pgfputat{\pgfxy(13.00,9.00)}{\pgfbox[bottom,left]{\fontsize{7.97}{9.56}\selectfont $3$}}
\pgfcircle[fill]{\pgfxy(5.00,5.00)}{0.70mm}
\pgfcircle[stroke]{\pgfxy(5.00,5.00)}{0.70mm}
\pgfcircle[fill]{\pgfxy(5.00,10.00)}{0.70mm}
\pgfcircle[stroke]{\pgfxy(5.00,10.00)}{0.70mm}
\pgfmoveto{\pgfxy(5.00,5.00)}\pgflineto{\pgfxy(5.00,10.00)}\pgfstroke
\pgfmoveto{\pgfxy(5.00,5.00)}\pgflineto{\pgfxy(10.00,10.00)}\pgfstroke
\pgfputat{\pgfxy(2.00,9.00)}{\pgfbox[bottom,left]{\fontsize{7.97}{9.56}\selectfont \makebox[0pt][r]{$2$}}}
\pgfputat{\pgfxy(13.00,14.00)}{\pgfbox[bottom,left]{\fontsize{7.97}{9.56}\selectfont $8$}}
\end{pgfpicture}%
$ &
$
\centering
\begin{pgfpicture}{-2.00mm}{0.20mm}{12.50mm}{14.00mm}
\pgfsetxvec{\pgfpoint{0.70mm}{0mm}}
\pgfsetyvec{\pgfpoint{0mm}{0.70mm}}
\color[rgb]{0,0,0}\pgfsetlinewidth{0.30mm}\pgfsetdash{}{0mm}
\pgfcircle[fill]{\pgfxy(5.00,15.00)}{0.70mm}
\pgfcircle[stroke]{\pgfxy(5.00,15.00)}{0.70mm}
\pgfputat{\pgfxy(2.00,4.00)}{\pgfbox[bottom,left]{\fontsize{7.97}{9.56}\selectfont \makebox[0pt][r]{$1$}}}
\pgfmoveto{\pgfxy(10.00,15.00)}\pgflineto{\pgfxy(10.00,10.00)}\pgfstroke
\pgfcircle[fill]{\pgfxy(10.00,15.00)}{0.70mm}
\pgfcircle[stroke]{\pgfxy(10.00,15.00)}{0.70mm}
\pgfcircle[fill]{\pgfxy(10.00,10.00)}{0.70mm}
\pgfcircle[stroke]{\pgfxy(10.00,10.00)}{0.70mm}
\pgfmoveto{\pgfxy(5.00,15.00)}\pgflineto{\pgfxy(5.00,9.50)}\pgfstroke
\pgfputat{\pgfxy(13.00,14.00)}{\pgfbox[bottom,left]{\fontsize{7.97}{9.56}\selectfont $9$}}
\pgfputat{\pgfxy(2.00,14.00)}{\pgfbox[bottom,left]{\fontsize{7.97}{9.56}\selectfont \makebox[0pt][r]{$3$}}}
\pgfcircle[fill]{\pgfxy(5.00,5.00)}{0.70mm}
\pgfcircle[stroke]{\pgfxy(5.00,5.00)}{0.70mm}
\pgfcircle[fill]{\pgfxy(5.00,10.00)}{0.70mm}
\pgfcircle[stroke]{\pgfxy(5.00,10.00)}{0.70mm}
\pgfmoveto{\pgfxy(5.00,5.00)}\pgflineto{\pgfxy(5.00,10.00)}\pgfstroke
\pgfmoveto{\pgfxy(5.00,5.00)}\pgflineto{\pgfxy(10.00,10.00)}\pgfstroke
\pgfputat{\pgfxy(2.00,9.00)}{\pgfbox[bottom,left]{\fontsize{7.97}{9.56}\selectfont \makebox[0pt][r]{$2$}}}
\pgfputat{\pgfxy(13.00,9.00)}{\pgfbox[bottom,left]{\fontsize{7.97}{9.56}\selectfont $8$}}
\end{pgfpicture}%
$ &
$
\centering
\begin{pgfpicture}{-2.00mm}{0.20mm}{12.50mm}{14.00mm}
\pgfsetxvec{\pgfpoint{0.70mm}{0mm}}
\pgfsetyvec{\pgfpoint{0mm}{0.70mm}}
\color[rgb]{0,0,0}\pgfsetlinewidth{0.30mm}\pgfsetdash{}{0mm}
\pgfcircle[fill]{\pgfxy(5.00,15.00)}{0.70mm}
\pgfcircle[stroke]{\pgfxy(5.00,15.00)}{0.70mm}
\pgfputat{\pgfxy(2.00,4.00)}{\pgfbox[bottom,left]{\fontsize{7.97}{9.56}\selectfont \makebox[0pt][r]{$1$}}}
\pgfmoveto{\pgfxy(10.00,15.00)}\pgflineto{\pgfxy(10.00,10.00)}\pgfstroke
\pgfcircle[fill]{\pgfxy(10.00,15.00)}{0.70mm}
\pgfcircle[stroke]{\pgfxy(10.00,15.00)}{0.70mm}
\pgfcircle[fill]{\pgfxy(10.00,10.00)}{0.70mm}
\pgfcircle[stroke]{\pgfxy(10.00,10.00)}{0.70mm}
\pgfmoveto{\pgfxy(5.00,15.00)}\pgflineto{\pgfxy(5.00,9.50)}\pgfstroke
\pgfputat{\pgfxy(13.00,14.00)}{\pgfbox[bottom,left]{\fontsize{7.97}{9.56}\selectfont $9$}}
\pgfputat{\pgfxy(2.00,14.00)}{\pgfbox[bottom,left]{\fontsize{7.97}{9.56}\selectfont \makebox[0pt][r]{$8$}}}
\pgfcircle[fill]{\pgfxy(5.00,5.00)}{0.70mm}
\pgfcircle[stroke]{\pgfxy(5.00,5.00)}{0.70mm}
\pgfcircle[fill]{\pgfxy(5.00,10.00)}{0.70mm}
\pgfcircle[stroke]{\pgfxy(5.00,10.00)}{0.70mm}
\pgfmoveto{\pgfxy(5.00,5.00)}\pgflineto{\pgfxy(5.00,10.00)}\pgfstroke
\pgfmoveto{\pgfxy(5.00,5.00)}\pgflineto{\pgfxy(10.00,10.00)}\pgfstroke
\pgfputat{\pgfxy(2.00,9.00)}{\pgfbox[bottom,left]{\fontsize{7.97}{9.56}\selectfont \makebox[0pt][r]{$2$}}}
\pgfputat{\pgfxy(13.00,9.00)}{\pgfbox[bottom,left]{\fontsize{7.97}{9.56}\selectfont $3$}}
\end{pgfpicture}%
$
\\
\hline
$\textsc{Leaf}(T_{m})$ & 9 & 8 & 9 & 2 & 3 & 8
\\
\hline
\hline
$\pi^{(m)}$ & $91823$ & $7691823$ & $8691723$ & $548691723$ & $648591723$ & $748591623$
\\
\hline
$\textsc{First}(\pi^{(m)})$ & 9 & 7 & 8 & 5 & 6 & 7
\\
\hline
$\Delta_m$\rule[-7pt]{0pt}{20pt} & $(9,8)$ & $(7,6)^{*} $ & $(8,7)$ & $(5,4)^{*} $ & $(6,5)$ & $(7,6)$
\\
\hline
$
\centering
\begin{pgfpicture}{-2.00mm}{-2.01mm}{11.51mm}{15.98mm}
\pgfsetxvec{\pgfpoint{0.70mm}{0mm}}
\pgfsetyvec{\pgfpoint{0mm}{0.70mm}}
\color[rgb]{0,0,0}\pgfsetlinewidth{0.30mm}\pgfsetdash{}{0mm}
\pgfputat{\pgfxy(7.00,9.00)}{\pgfbox[bottom,left]{\fontsize{11}{9.56}\selectfont \makebox[0pt]{$T_{m}=\Phi(\pi^{(m)})$}}}
\end{pgfpicture}%
$ &
$
\centering
\begin{pgfpicture}{-2.00mm}{0.20mm}{12.50mm}{14.00mm}
\pgfsetxvec{\pgfpoint{0.70mm}{0mm}}
\pgfsetyvec{\pgfpoint{0mm}{0.70mm}}
\color[rgb]{0,0,0}\pgfsetlinewidth{0.30mm}\pgfsetdash{}{0mm}
\pgfcircle[fill]{\pgfxy(5.00,15.00)}{0.70mm}
\pgfcircle[stroke]{\pgfxy(5.00,15.00)}{0.70mm}
\pgfputat{\pgfxy(2.00,4.00)}{\pgfbox[bottom,left]{\fontsize{7.97}{9.56}\selectfont \makebox[0pt][r]{$1$}}}
\pgfmoveto{\pgfxy(10.00,15.00)}\pgflineto{\pgfxy(10.00,10.00)}\pgfstroke
\pgfcircle[fill]{\pgfxy(10.00,15.00)}{0.70mm}
\pgfcircle[stroke]{\pgfxy(10.00,15.00)}{0.70mm}
\pgfcircle[fill]{\pgfxy(10.00,10.00)}{0.70mm}
\pgfcircle[stroke]{\pgfxy(10.00,10.00)}{0.70mm}
\pgfmoveto{\pgfxy(5.00,15.00)}\pgflineto{\pgfxy(5.00,9.50)}\pgfstroke
\pgfputat{\pgfxy(13.00,14.00)}{\pgfbox[bottom,left]{\fontsize{7.97}{9.56}\selectfont $8$}}
\pgfputat{\pgfxy(2.00,14.00)}{\pgfbox[bottom,left]{\fontsize{7.97}{9.56}\selectfont \makebox[0pt][r]{$9$}}}
\pgfcircle[fill]{\pgfxy(5.00,5.00)}{0.70mm}
\pgfcircle[stroke]{\pgfxy(5.00,5.00)}{0.70mm}
\pgfcircle[fill]{\pgfxy(5.00,10.00)}{0.70mm}
\pgfcircle[stroke]{\pgfxy(5.00,10.00)}{0.70mm}
\pgfmoveto{\pgfxy(5.00,5.00)}\pgflineto{\pgfxy(5.00,10.00)}\pgfstroke
\pgfmoveto{\pgfxy(5.00,5.00)}\pgflineto{\pgfxy(10.00,10.00)}\pgfstroke
\pgfputat{\pgfxy(2.00,9.00)}{\pgfbox[bottom,left]{\fontsize{7.97}{9.56}\selectfont \makebox[0pt][r]{$2$}}}
\pgfputat{\pgfxy(13.00,9.00)}{\pgfbox[bottom,left]{\fontsize{7.97}{9.56}\selectfont $3$}}
\end{pgfpicture}%
$ &
$
\centering
\begin{pgfpicture}{-2.00mm}{0.20mm}{12.50mm}{17.50mm}
\pgfsetxvec{\pgfpoint{0.70mm}{0mm}}
\pgfsetyvec{\pgfpoint{0mm}{0.70mm}}
\color[rgb]{0,0,0}\pgfsetlinewidth{0.30mm}\pgfsetdash{}{0mm}
\pgfcircle[fill]{\pgfxy(5.00,15.00)}{0.70mm}
\pgfcircle[stroke]{\pgfxy(5.00,15.00)}{0.70mm}
\pgfputat{\pgfxy(2.00,4.00)}{\pgfbox[bottom,left]{\fontsize{7.97}{9.56}\selectfont \makebox[0pt][r]{$1$}}}
\pgfmoveto{\pgfxy(10.00,15.00)}\pgflineto{\pgfxy(10.00,10.00)}\pgfstroke
\pgfcircle[fill]{\pgfxy(10.00,15.00)}{0.70mm}
\pgfcircle[stroke]{\pgfxy(10.00,15.00)}{0.70mm}
\pgfcircle[fill]{\pgfxy(10.00,10.00)}{0.70mm}
\pgfcircle[stroke]{\pgfxy(10.00,10.00)}{0.70mm}
\pgfmoveto{\pgfxy(5.00,15.00)}\pgflineto{\pgfxy(5.00,9.50)}\pgfstroke
\pgfputat{\pgfxy(2.00,14.00)}{\pgfbox[bottom,left]{\fontsize{7.97}{9.56}\selectfont \makebox[0pt][r]{$6$}}}
\pgfputat{\pgfxy(2.00,19.00)}{\pgfbox[bottom,left]{\fontsize{7.97}{9.56}\selectfont \makebox[0pt][r]{$7$}}}
\pgfcircle[fill]{\pgfxy(5.00,5.00)}{0.70mm}
\pgfcircle[stroke]{\pgfxy(5.00,5.00)}{0.70mm}
\pgfcircle[fill]{\pgfxy(5.00,10.00)}{0.70mm}
\pgfcircle[stroke]{\pgfxy(5.00,10.00)}{0.70mm}
\pgfmoveto{\pgfxy(5.00,5.00)}\pgflineto{\pgfxy(5.00,10.00)}\pgfstroke
\pgfmoveto{\pgfxy(5.00,5.00)}\pgflineto{\pgfxy(10.00,10.00)}\pgfstroke
\pgfputat{\pgfxy(2.00,9.00)}{\pgfbox[bottom,left]{\fontsize{7.97}{9.56}\selectfont \makebox[0pt][r]{$2$}}}
\pgfputat{\pgfxy(13.00,9.00)}{\pgfbox[bottom,left]{\fontsize{7.97}{9.56}\selectfont $3$}}
\pgfcircle[fill]{\pgfxy(10.00,20.00)}{0.70mm}
\pgfcircle[stroke]{\pgfxy(10.00,20.00)}{0.70mm}
\pgfmoveto{\pgfxy(5.00,20.00)}\pgflineto{\pgfxy(5.00,15.00)}\pgfstroke
\pgfmoveto{\pgfxy(5.00,15.00)}\pgflineto{\pgfxy(10.00,20.00)}\pgfstroke
\pgfcircle[fill]{\pgfxy(5.00,20.00)}{0.70mm}
\pgfcircle[stroke]{\pgfxy(5.00,20.00)}{0.70mm}
\pgfputat{\pgfxy(13.00,14.00)}{\pgfbox[bottom,left]{\fontsize{7.97}{9.56}\selectfont $8$}}
\pgfputat{\pgfxy(13.00,19.00)}{\pgfbox[bottom,left]{\fontsize{7.97}{9.56}\selectfont $9$}}
\end{pgfpicture}%
$ &
$
\centering
\begin{pgfpicture}{-2.00mm}{0.20mm}{12.50mm}{17.50mm}
\pgfsetxvec{\pgfpoint{0.70mm}{0mm}}
\pgfsetyvec{\pgfpoint{0mm}{0.70mm}}
\color[rgb]{0,0,0}\pgfsetlinewidth{0.30mm}\pgfsetdash{}{0mm}
\pgfcircle[fill]{\pgfxy(5.00,15.00)}{0.70mm}
\pgfcircle[stroke]{\pgfxy(5.00,15.00)}{0.70mm}
\pgfputat{\pgfxy(2.00,4.00)}{\pgfbox[bottom,left]{\fontsize{7.97}{9.56}\selectfont \makebox[0pt][r]{$1$}}}
\pgfmoveto{\pgfxy(10.00,15.00)}\pgflineto{\pgfxy(10.00,10.00)}\pgfstroke
\pgfcircle[fill]{\pgfxy(10.00,15.00)}{0.70mm}
\pgfcircle[stroke]{\pgfxy(10.00,15.00)}{0.70mm}
\pgfcircle[fill]{\pgfxy(10.00,10.00)}{0.70mm}
\pgfcircle[stroke]{\pgfxy(10.00,10.00)}{0.70mm}
\pgfmoveto{\pgfxy(5.00,15.00)}\pgflineto{\pgfxy(5.00,9.50)}\pgfstroke
\pgfputat{\pgfxy(2.00,14.00)}{\pgfbox[bottom,left]{\fontsize{7.97}{9.56}\selectfont \makebox[0pt][r]{$6$}}}
\pgfputat{\pgfxy(2.00,19.00)}{\pgfbox[bottom,left]{\fontsize{7.97}{9.56}\selectfont \makebox[0pt][r]{$8$}}}
\pgfcircle[fill]{\pgfxy(5.00,5.00)}{0.70mm}
\pgfcircle[stroke]{\pgfxy(5.00,5.00)}{0.70mm}
\pgfcircle[fill]{\pgfxy(5.00,10.00)}{0.70mm}
\pgfcircle[stroke]{\pgfxy(5.00,10.00)}{0.70mm}
\pgfmoveto{\pgfxy(5.00,5.00)}\pgflineto{\pgfxy(5.00,10.00)}\pgfstroke
\pgfmoveto{\pgfxy(5.00,5.00)}\pgflineto{\pgfxy(10.00,10.00)}\pgfstroke
\pgfputat{\pgfxy(2.00,9.00)}{\pgfbox[bottom,left]{\fontsize{7.97}{9.56}\selectfont \makebox[0pt][r]{$2$}}}
\pgfputat{\pgfxy(13.00,9.00)}{\pgfbox[bottom,left]{\fontsize{7.97}{9.56}\selectfont $3$}}
\pgfcircle[fill]{\pgfxy(10.00,20.00)}{0.70mm}
\pgfcircle[stroke]{\pgfxy(10.00,20.00)}{0.70mm}
\pgfmoveto{\pgfxy(5.00,20.00)}\pgflineto{\pgfxy(5.00,15.00)}\pgfstroke
\pgfmoveto{\pgfxy(5.00,15.00)}\pgflineto{\pgfxy(10.00,20.00)}\pgfstroke
\pgfcircle[fill]{\pgfxy(5.00,20.00)}{0.70mm}
\pgfcircle[stroke]{\pgfxy(5.00,20.00)}{0.70mm}
\pgfputat{\pgfxy(13.00,14.00)}{\pgfbox[bottom,left]{\fontsize{7.97}{9.56}\selectfont $7$}}
\pgfputat{\pgfxy(13.00,19.00)}{\pgfbox[bottom,left]{\fontsize{7.97}{9.56}\selectfont $9$}}
\end{pgfpicture}%
$ &
$
\centering
\begin{pgfpicture}{-2.00mm}{0.20mm}{16.00mm}{21.00mm}
\pgfsetxvec{\pgfpoint{0.70mm}{0mm}}
\pgfsetyvec{\pgfpoint{0mm}{0.70mm}}
\color[rgb]{0,0,0}\pgfsetlinewidth{0.30mm}\pgfsetdash{}{0mm}
\pgfputat{\pgfxy(2.00,4.00)}{\pgfbox[bottom,left]{\fontsize{7.97}{9.56}\selectfont \makebox[0pt][r]{$1$}}}
\pgfmoveto{\pgfxy(10.00,15.00)}\pgflineto{\pgfxy(10.00,10.00)}\pgfstroke
\pgfcircle[fill]{\pgfxy(10.00,15.00)}{0.70mm}
\pgfcircle[stroke]{\pgfxy(10.00,15.00)}{0.70mm}
\pgfcircle[fill]{\pgfxy(10.00,10.00)}{0.70mm}
\pgfcircle[stroke]{\pgfxy(10.00,10.00)}{0.70mm}
\pgfmoveto{\pgfxy(5.00,15.00)}\pgflineto{\pgfxy(5.00,9.50)}\pgfstroke
\pgfcircle[fill]{\pgfxy(5.00,5.00)}{0.70mm}
\pgfcircle[stroke]{\pgfxy(5.00,5.00)}{0.70mm}
\pgfcircle[fill]{\pgfxy(5.00,10.00)}{0.70mm}
\pgfcircle[stroke]{\pgfxy(5.00,10.00)}{0.70mm}
\pgfmoveto{\pgfxy(5.00,5.00)}\pgflineto{\pgfxy(5.00,10.00)}\pgfstroke
\pgfmoveto{\pgfxy(5.00,5.00)}\pgflineto{\pgfxy(10.00,10.00)}\pgfstroke
\pgfputat{\pgfxy(2.00,9.00)}{\pgfbox[bottom,left]{\fontsize{7.97}{9.56}\selectfont \makebox[0pt][r]{$2$}}}
\pgfputat{\pgfxy(13.00,9.00)}{\pgfbox[bottom,left]{\fontsize{7.97}{9.56}\selectfont $3$}}
\pgfputat{\pgfxy(13.00,14.00)}{\pgfbox[bottom,left]{\fontsize{7.97}{9.56}\selectfont $7$}}
\pgfcircle[fill]{\pgfxy(10.00,20.00)}{0.70mm}
\pgfcircle[stroke]{\pgfxy(10.00,20.00)}{0.70mm}
\pgfputat{\pgfxy(13.00,19.00)}{\pgfbox[bottom,left]{\fontsize{7.97}{9.56}\selectfont $6$}}
\pgfputat{\pgfxy(7.00,24.00)}{\pgfbox[bottom,left]{\fontsize{7.97}{9.56}\selectfont \makebox[0pt][r]{$8$}}}
\pgfcircle[fill]{\pgfxy(15.00,25.00)}{0.70mm}
\pgfcircle[stroke]{\pgfxy(15.00,25.00)}{0.70mm}
\pgfmoveto{\pgfxy(10.00,25.00)}\pgflineto{\pgfxy(10.00,20.00)}\pgfstroke
\pgfmoveto{\pgfxy(10.00,20.00)}\pgflineto{\pgfxy(15.00,25.00)}\pgfstroke
\pgfcircle[fill]{\pgfxy(10.00,25.00)}{0.70mm}
\pgfcircle[stroke]{\pgfxy(10.00,25.00)}{0.70mm}
\pgfputat{\pgfxy(18.00,24.00)}{\pgfbox[bottom,left]{\fontsize{7.97}{9.56}\selectfont $9$}}
\pgfcircle[fill]{\pgfxy(5.00,15.00)}{0.70mm}
\pgfcircle[stroke]{\pgfxy(5.00,15.00)}{0.70mm}
\pgfcircle[fill]{\pgfxy(5.00,20.00)}{0.70mm}
\pgfcircle[stroke]{\pgfxy(5.00,20.00)}{0.70mm}
\pgfmoveto{\pgfxy(10.00,20.00)}\pgflineto{\pgfxy(5.00,15.00)}\pgfstroke
\pgfmoveto{\pgfxy(5.00,20.00)}\pgflineto{\pgfxy(5.00,15.00)}\pgfstroke
\pgfputat{\pgfxy(2.00,19.00)}{\pgfbox[bottom,left]{\fontsize{7.97}{9.56}\selectfont \makebox[0pt][r]{$5$}}}
\pgfputat{\pgfxy(2.00,14.00)}{\pgfbox[bottom,left]{\fontsize{7.97}{9.56}\selectfont \makebox[0pt][r]{$4$}}}
\end{pgfpicture}%
$ &
$
\centering
\begin{pgfpicture}{-2.00mm}{0.20mm}{12.50mm}{21.00mm}
\pgfsetxvec{\pgfpoint{0.70mm}{0mm}}
\pgfsetyvec{\pgfpoint{0mm}{0.70mm}}
\color[rgb]{0,0,0}\pgfsetlinewidth{0.30mm}\pgfsetdash{}{0mm}
\pgfputat{\pgfxy(2.00,4.00)}{\pgfbox[bottom,left]{\fontsize{7.97}{9.56}\selectfont \makebox[0pt][r]{$1$}}}
\pgfmoveto{\pgfxy(10.00,15.00)}\pgflineto{\pgfxy(10.00,10.00)}\pgfstroke
\pgfcircle[fill]{\pgfxy(10.00,15.00)}{0.70mm}
\pgfcircle[stroke]{\pgfxy(10.00,15.00)}{0.70mm}
\pgfcircle[fill]{\pgfxy(10.00,10.00)}{0.70mm}
\pgfcircle[stroke]{\pgfxy(10.00,10.00)}{0.70mm}
\pgfmoveto{\pgfxy(5.00,15.00)}\pgflineto{\pgfxy(5.00,9.50)}\pgfstroke
\pgfcircle[fill]{\pgfxy(5.00,5.00)}{0.70mm}
\pgfcircle[stroke]{\pgfxy(5.00,5.00)}{0.70mm}
\pgfcircle[fill]{\pgfxy(5.00,10.00)}{0.70mm}
\pgfcircle[stroke]{\pgfxy(5.00,10.00)}{0.70mm}
\pgfmoveto{\pgfxy(5.00,5.00)}\pgflineto{\pgfxy(5.00,10.00)}\pgfstroke
\pgfmoveto{\pgfxy(5.00,5.00)}\pgflineto{\pgfxy(10.00,10.00)}\pgfstroke
\pgfputat{\pgfxy(2.00,9.00)}{\pgfbox[bottom,left]{\fontsize{7.97}{9.56}\selectfont \makebox[0pt][r]{$2$}}}
\pgfputat{\pgfxy(13.00,9.00)}{\pgfbox[bottom,left]{\fontsize{7.97}{9.56}\selectfont $3$}}
\pgfputat{\pgfxy(13.00,14.00)}{\pgfbox[bottom,left]{\fontsize{7.97}{9.56}\selectfont $7$}}
\pgfcircle[fill]{\pgfxy(10.00,20.00)}{0.70mm}
\pgfcircle[stroke]{\pgfxy(10.00,20.00)}{0.70mm}
\pgfputat{\pgfxy(2.00,24.00)}{\pgfbox[bottom,left]{\fontsize{7.97}{9.56}\selectfont \makebox[0pt][r]{$6$}}}
\pgfputat{\pgfxy(13.00,19.00)}{\pgfbox[bottom,left]{\fontsize{7.97}{9.56}\selectfont $8$}}
\pgfcircle[fill]{\pgfxy(10.00,25.00)}{0.70mm}
\pgfcircle[stroke]{\pgfxy(10.00,25.00)}{0.70mm}
\pgfmoveto{\pgfxy(5.00,25.00)}\pgflineto{\pgfxy(5.00,20.00)}\pgfstroke
\pgfmoveto{\pgfxy(10.00,25.00)}\pgflineto{\pgfxy(5.00,20.00)}\pgfstroke
\pgfcircle[fill]{\pgfxy(5.00,25.00)}{0.70mm}
\pgfcircle[stroke]{\pgfxy(5.00,25.00)}{0.70mm}
\pgfputat{\pgfxy(13.00,24.00)}{\pgfbox[bottom,left]{\fontsize{7.97}{9.56}\selectfont $9$}}
\pgfcircle[fill]{\pgfxy(5.00,15.00)}{0.70mm}
\pgfcircle[stroke]{\pgfxy(5.00,15.00)}{0.70mm}
\pgfcircle[fill]{\pgfxy(5.00,20.00)}{0.70mm}
\pgfcircle[stroke]{\pgfxy(5.00,20.00)}{0.70mm}
\pgfmoveto{\pgfxy(10.00,20.00)}\pgflineto{\pgfxy(5.00,15.00)}\pgfstroke
\pgfmoveto{\pgfxy(5.00,20.00)}\pgflineto{\pgfxy(5.00,15.00)}\pgfstroke
\pgfputat{\pgfxy(2.00,19.00)}{\pgfbox[bottom,left]{\fontsize{7.97}{9.56}\selectfont \makebox[0pt][r]{$5$}}}
\pgfputat{\pgfxy(2.00,14.00)}{\pgfbox[bottom,left]{\fontsize{7.97}{9.56}\selectfont \makebox[0pt][r]{$4$}}}
\end{pgfpicture}%
$ &
$
\centering
\begin{pgfpicture}{-2.00mm}{0.20mm}{12.50mm}{21.00mm}
\pgfsetxvec{\pgfpoint{0.70mm}{0mm}}
\pgfsetyvec{\pgfpoint{0mm}{0.70mm}}
\color[rgb]{0,0,0}\pgfsetlinewidth{0.30mm}\pgfsetdash{}{0mm}
\pgfputat{\pgfxy(2.00,4.00)}{\pgfbox[bottom,left]{\fontsize{7.97}{9.56}\selectfont \makebox[0pt][r]{$1$}}}
\pgfmoveto{\pgfxy(10.00,15.00)}\pgflineto{\pgfxy(10.00,10.00)}\pgfstroke
\pgfcircle[fill]{\pgfxy(10.00,15.00)}{0.70mm}
\pgfcircle[stroke]{\pgfxy(10.00,15.00)}{0.70mm}
\pgfcircle[fill]{\pgfxy(10.00,10.00)}{0.70mm}
\pgfcircle[stroke]{\pgfxy(10.00,10.00)}{0.70mm}
\pgfmoveto{\pgfxy(5.00,15.00)}\pgflineto{\pgfxy(5.00,9.50)}\pgfstroke
\pgfcircle[fill]{\pgfxy(5.00,5.00)}{0.70mm}
\pgfcircle[stroke]{\pgfxy(5.00,5.00)}{0.70mm}
\pgfcircle[fill]{\pgfxy(5.00,10.00)}{0.70mm}
\pgfcircle[stroke]{\pgfxy(5.00,10.00)}{0.70mm}
\pgfmoveto{\pgfxy(5.00,5.00)}\pgflineto{\pgfxy(5.00,10.00)}\pgfstroke
\pgfmoveto{\pgfxy(5.00,5.00)}\pgflineto{\pgfxy(10.00,10.00)}\pgfstroke
\pgfputat{\pgfxy(2.00,9.00)}{\pgfbox[bottom,left]{\fontsize{7.97}{9.56}\selectfont \makebox[0pt][r]{$2$}}}
\pgfputat{\pgfxy(13.00,9.00)}{\pgfbox[bottom,left]{\fontsize{7.97}{9.56}\selectfont $3$}}
\pgfputat{\pgfxy(13.00,14.00)}{\pgfbox[bottom,left]{\fontsize{7.97}{9.56}\selectfont $6$}}
\pgfcircle[fill]{\pgfxy(10.00,20.00)}{0.70mm}
\pgfcircle[stroke]{\pgfxy(10.00,20.00)}{0.70mm}
\pgfputat{\pgfxy(2.00,24.00)}{\pgfbox[bottom,left]{\fontsize{7.97}{9.56}\selectfont \makebox[0pt][r]{$7$}}}
\pgfputat{\pgfxy(13.00,19.00)}{\pgfbox[bottom,left]{\fontsize{7.97}{9.56}\selectfont $8$}}
\pgfcircle[fill]{\pgfxy(10.00,25.00)}{0.70mm}
\pgfcircle[stroke]{\pgfxy(10.00,25.00)}{0.70mm}
\pgfmoveto{\pgfxy(5.00,25.00)}\pgflineto{\pgfxy(5.00,20.00)}\pgfstroke
\pgfmoveto{\pgfxy(10.00,25.00)}\pgflineto{\pgfxy(5.00,20.00)}\pgfstroke
\pgfcircle[fill]{\pgfxy(5.00,25.00)}{0.70mm}
\pgfcircle[stroke]{\pgfxy(5.00,25.00)}{0.70mm}
\pgfputat{\pgfxy(13.00,24.00)}{\pgfbox[bottom,left]{\fontsize{7.97}{9.56}\selectfont $9$}}
\pgfcircle[fill]{\pgfxy(5.00,15.00)}{0.70mm}
\pgfcircle[stroke]{\pgfxy(5.00,15.00)}{0.70mm}
\pgfcircle[fill]{\pgfxy(5.00,20.00)}{0.70mm}
\pgfcircle[stroke]{\pgfxy(5.00,20.00)}{0.70mm}
\pgfmoveto{\pgfxy(10.00,20.00)}\pgflineto{\pgfxy(5.00,15.00)}\pgfstroke
\pgfmoveto{\pgfxy(5.00,20.00)}\pgflineto{\pgfxy(5.00,15.00)}\pgfstroke
\pgfputat{\pgfxy(2.00,19.00)}{\pgfbox[bottom,left]{\fontsize{7.97}{9.56}\selectfont \makebox[0pt][r]{$5$}}}
\pgfputat{\pgfxy(2.00,14.00)}{\pgfbox[bottom,left]{\fontsize{7.97}{9.56}\selectfont \makebox[0pt][r]{$4$}}}
\end{pgfpicture}%
$
\\
\hline
$\textsc{Leaf}(T_{m})$ & 9 & 7 & 8 & 5 & 6 & 7
\\
\end{tabular}
\caption{\label{fig1} The construction of the tree $\Psi(7\,4\,8\,5\,9\,1\,6\,2\,3)$}
\end{figure}

\end{ex}

\subsection{Interpretation of Entringer's formula  in $\Tn$} \label{s-rmq}

Following the interpretation of \eqref{eq1} in $\Dn$ (cf Remark~\ref{s-rmq2}) and the bijection $\varphi$, we must consider the decomposition of the set $\Tnk$. The first step in the construction of $\varphi^{-1}$ would consist in either removing the elements $k-1$ and $k$, or the first step transforms the tree to obtain another tree of $\Tn$.

For $T$ an element of $\Tnk$, we say that the edge $(k-1,k)$ is \emph{removable} if  $k-1$ is the parent of $k$ and if $k-1$ has another child $m$ that is not greater than the sibling of $k-1$ (if such a sibling exists). For a visual representation, a tree $T$ has its edge $(k-1,k)$ removable if it corresponds to the case A-1 in the proof of Theorem \ref{thmvarphi}.

If the edge $(k-1,k)$ is not removable, the tree obtained after the first operation in the construction of $\varphi^{-1}$ will be an increasing tree with $n$ elements such that $k-1$ is the leaf of the main chain. Then, there are exactly $E_{n,k-1}$ trees such that the edge $(k-1,k)$ is not removable.

If the edge is removable, the tree obtained with the first operation in the construction of $\varphi^{-1}$ will be an increasing tree with $n-2$ elements (without the elements $k-1$ and $k$), and the end of the minimal path must be an element $i$ greater than the $k-2$ first elements. Thus, there are $E_{n-2,k-1}+E_{n-2,k}+\cdots + E_{n-2,n-2}=E_{n-1,n-k+1}$  increasing trees such that the edge $(k-1,k)$ is removable.

Finally, an interpretation of \eqref{eq1} appears in the model of $\mathcal{T}_n$. The decomposition according to the removability of the edge $(k-1,k)$ in $T \in \Tnk$ gives \eqref{eq1}.

\section{Poupard's other Entringer families}
\label{sec-poupard}

\subsection{Another interpretation in increasing trees}
\label{s-trees}
Let $\Tnk'$ be  the set of trees $T \in \Tn$ such that the parent of $n$ in $T$  is $k-1$.  By using recurrence relations Poupard proved that $E_{n,k}$ is also the number of trees in $\Tnk'$. A bijection $\varphi'$ between $\Tnk$ and $\Tnk'$ was given in \cite[\S 6]{KPP94} for a more general class of increasing trees that they call geometric.


\subsection{Another interpretation in down-up permutations} \label{s-theta}

If $\pi$ is a permutation of $\mathcal{DU}_{n,k}$, define $\theta(\pi)$ as follows:
\begin{itemize}
\item[$\bullet$] if $k < n-k+1+\pi_2$, then $\theta(\pi) = (n-k+1+\pi_2,n-k+\pi_2,\ldots,k+1,k) \circ \pi$,
\item[$\bullet$] if $k > n-k+1+\pi_2$, then $\theta(\pi) = (n-k+1+\pi_2,n-k+2+\pi_2,\ldots,k-1,k) \circ \pi$.
\end{itemize}

Since $\pi$ is down-up, $\pi_2 < k = \pi_1$. If $k < n-k+1+\pi_2$, $\pi_2$ is unchanged by the cycle and then $\sigma(\pi)_2=\pi_2$. Thus $\sigma(\pi)_2 < k < n-k+1+\pi_2 = \sigma(\pi)_1$ and $\theta(\pi)$ is still down-up. If $k > n-k+1+\pi_2$, since $k \leq n$, then $n-k+1+\pi_2 \geq \pi_2+1$, so $\pi_2$ is unchanged by the cycle, $\sigma(\pi)_1 = n-k+1+\pi_2 > \pi_2 = \sigma(\pi)_2$ and $\theta(\pi)$ is still down-up.

Let's denote by $\mathcal{DU}_{n,k}'$ the set of permutations $\pi \in \mathcal{DU}_n$ such that $\pi_1-\pi_2=n+1-k$.

\begin{thm} \label{thm6} For all $n \geq 1$ and $k \in [n]$, the mapping $\theta$ is a bijection from $\Ank$ to $\Ank'$. Moreover, for every $\pi \in \Ank$, 
we have $\theta(\pi)_2=\pi_2$.
\end{thm}

\begin{proof} By construction,  the mapping $\theta$ is clearly invertible. Moreover,  for  $\sigma \in \An$ with $\sigma_1-\sigma_2=n-k+1$,
\begin{itemize}
\item[$\bullet$] if $k < n-k+1+\sigma_2$, then $\theta^{-1}(\sigma) = (k,k+1,\ldots,n-k+\sigma_2,n-k+1+\sigma_2) \circ \sigma$,
\item[$\bullet$] if $k > n-k+1+\sigma_2$, then $\theta^{-1}(\sigma) = (k,k-1,\ldots,n-k+2+\sigma_2,n-k+1+\sigma_2) \circ \sigma$,
\end{itemize} 
thus  $\theta^{-1}(\sigma) \in \Ank$.
\end{proof}

With Theorem~\ref{thm6}, the following interpretation of Poupard, proved in \cite{Pou97} by recurrence relations, can be recovered.

\begin{cor} The sequence $(\mathcal{DU}_{n,k}')_{1 \leq k \leq n}$ is an Entringer family. \end{cor}

Since $\mathcal{DU}_{n,k}' \subset \mathcal{DU}_{n}$, we can define $\theta^2(\pi)$ for $\pi \in \mathcal{DU}_n$. Actually, it is easy to see that 
the mapping $\theta$ is  an involution on $\mathcal{DU}_{n}$.
The result can also be generalized with the following observation. For any $\pi \in \An$, define the \emph{complement permutation} $\overline{\pi}$ with $\overline{\pi}_i=n+1-\pi_i$ for $i \in [n]$. Denote by $\An^*$ the set of permutations $\pi$ such that $\overline{\pi}\in \An$.

\begin{cor} For $n \geq 1$, we have
\[ \sum\limits_{\pi \in \An^*} q^{\pi_1} p^{\pi_2-\pi_1} = \sum\limits_{\pi \in \An^*} p^{\pi_1} q^{\pi_2-\pi_1}. \] \end{cor}

\begin{proof} The mapping $\pi \mapsto \theta(\overline{\pi})$ is a bijection between $\{ \pi \in \An^*:\pi_1=k\}$ and $\{\pi \in \An^*:\pi_2-\pi_1=k\}$. Thus, the two statistics $\pi_1$ and $\pi_2-\pi_1$ are equidistributed on $\An^*$. Indeed, with proof of Theorem~\ref{thm6}, $\pi \mapsto \theta(\overline{\pi})$ is a bijection between $\{ \pi \in \An^*:\pi_1=k,\pi_2-\pi_1=\ell\}$ and $\{ \pi \in \An^*:\pi_1=\ell,\pi_2-\pi_1=k\}$. Thus, the distribution of the two  statistics is symmetric.
 \end{proof}

\subsection{Interpretations in min-max alternating permutations} \label{s-direct}
Recall that a permutation $\pi$ on $[n]$ is an alternating   permutation if 
$\pi_1 > \pi_2 < \pi_3 > \cdots$  or $\pi_1 < \pi_2 > \pi_3 < \cdots$. 
A \emph{min-max alternating permutation of $[n]$} is an alternating permutation in which $1$ precedes $n$. 
For example, the min-max alternating permutations of $[4]$ are \[ 1 \, 4 \, 2 \, 3, \quad 1 \, 3 \, 2 \, 4, \quad 3 \, 1 \, 4 \, 2, \quad 2 \, 3 \, 1 \, 4, \quad 2 \, 1 \, 4 \, 3. \]
Let $\mathcal{MM}_n$ be the set of \emph{min-max alternating permutations of $[n]$}.
Denote by $\mathcal{MM}_{n,k}$ the set of $\pi \in \mathcal{MM}_n$ such that $|\pi_1-\pi_2|=n+1-k$.
The set $\mathcal{DU}'_{n,k}$ can be split in two disjoint subsets $\mathcal{DU}'_{n,k,1n}$ which is the set of permutations in $\mathcal{DU}'_{n,k} \cap \mathcal{MM}_n$ and $\mathcal{DU}'_{n,k,n1}= \mathcal{DU}'_{n,k} \setminus \mathcal{DU}'_{n,k,1n}$. If $\pi \in \mathcal{DU}'_{n,k,1n}$, define $\beta(\pi)=\pi$, and if $\pi \in \mathcal{DU}'_{n,k,n1}$, define $\beta(\pi)=\overline{\pi}$. Thus $\beta(\pi) \in \mathcal{MM}_n$ and $\beta(\pi)_1-\beta(\pi)_2=-(n+1-k)$.

\begin{thm} For all $n \geq 1$ and $k \in [n]$, the mapping $\beta$ is a bijection between $\Ank'$ and $\mathcal{MM}_{n,k}$. \end{thm}

With the previous theorem, the interpretation of Poupard, proved in \cite{Pou97} by recurrence relations, can be recovered.

\begin{cor} \label{thmpou1}
 The sequence $(\mathcal{MM}_{n,k})_{1 \leq k \leq n}$ is an Entringer family.
\end{cor}

Denote by $\mathcal{MM}_{n,k}'$ the set of $\pi \in \mathcal{MM}_n$ such that the term immediately before $1$ is $k$, if $k \leq n-1$, and by $\mathcal{MM}_{n,n}'$ the set of $\pi \in \mathcal{MM}_n$ such that $\pi_1=1$.

We want to construct a bijection $\rho$ between $\Ank$ and $\mathcal{MM}_{n,k}'$.

If $k=n$, it suffices to define for $\pi \in \mathcal{DU}_{n,n}$, $\rho(\pi)=\overline{\pi}$.
Then, $\rho(\pi) \in \mathcal{MM}_{n,n}'$.

Assume that $k \leq n-1$. The set $\Ank$ can be split in two disjoint subsets $\mathcal{DU}_{n,k,1n}$ which is $\Ank \cap \mathcal{MM}_n$ and $\mathcal{DU}_{n,k,n1}:=\Ank \setminus \mathcal{DU}_{n,k,1n}$.
For an ordered set $I=\{a_1,\ldots,a_n\}$ with $a_1<\cdots<a_n$, denote by $\sigma_I$ the permutation:
\[ \sigma_I=\left(
     \begin{array}{cccc}
       a_1 & a_2 & \cdots & a_n \\
       a_n & a_{n-1} & \cdots & a_1 \\
     \end{array}
   \right) \] Then, for a permutation $\pi=\pi_1\ldots\pi_n$ on the ordered set $I$, denote by $\overline{\pi}$ the \emph{complement permutation on the set $I$}, that is $\overline{\pi}:=\sigma_I \circ \pi$, and $\pi^R$ the \emph{reverse permutation}: $$\pi^R:=\pi_n\pi_{n-1}\ldots\pi_1.$$ Note that when $I=[n]$, the definition of the complement permutation coincides with the one in the Remark of Subsection~\ref{s-theta}.

Then, for a permutation $\pi \in \Ank$,
\begin{itemize}
\item[$\bullet$] If $\pi \in \mathcal{DU}_{n,k,1n}$, we can write $\pi=\sigma_1 \,1 \,\sigma_2$. Then, define $\rho(\pi)=\sigma_1^R \,1\, \sigma_2$. Since $1 < \pi_1 > \pi_2$, $\rho(\pi)$ is still down-up, and the term just before $1$ in $\rho(\pi)$ is $\pi_1=k$.
\item[$\bullet$] If $\pi \in \mathcal{DU}_{n,k,n1}$, we can write $\pi=\sigma_1 \, n \, \sigma_2$. Then, define $\rho(\pi) = \sigma_1^R \, 1 \, \overline{\sigma_2}$. Since $1 < \pi_1 > \pi_2$ and $\overline{\sigma_2}$ is down-up, $\rho(\sigma_1)$ is still down-up, and the term just before $1$ in $\rho(\pi)$ is $\pi_1=k$.
\end{itemize}

\begin{thm} \label{thmpou2}
For all $n \geq 1$ and $k \in [n]$, the mapping $\rho$ is a bijection between $\mathcal{DU}_{n,k}$ and $\mathcal{MM}_{n,k}'$.
\end{thm}

\begin{proof}
In order to prove that $\rho$ is a bijection, it suffices to describe the inverse of $\rho$. Let $\pi$ be an element in $\mathcal{MM}_n$
such that the term immediately before $1$ is $k$. Following the construction of $\rho$, we have:
\begin{itemize}
\item[$\bullet$] If $\pi \in \Ank$, write $\pi=\tau_1 \, 1 \, \tau_2$. Then, $\rho^{-1}(\pi)=\tau_1^R \, 1 \, \tau_2$.
\item[$\bullet$] If $\pi \not\in \Ank$, write $\pi=\tau_1 \, 1 \, \tau_2$. Then, $\rho^{-1}(\pi)=\tau_1^R \, n \, \overline{\tau_2}$.
\end{itemize}
\end{proof}
With the previous theorem, the following interpretation of Poupard, proved in \cite{Pou97} by recurrence relations, can be recovered.
\begin{cor} \label{thmpou2}
 The sequence $(\mathcal{MM}_{n,k}')_{1 \leq k \leq n}$ is an Entringer family.
\end{cor}

Denote by $\mathcal{MM}_{n,k}''$ the set of $\pi \in \mathcal{MM}_n$ such that the term immediately after $n$ is $n+1-k$, if $k \leq n-1$, and $\mathcal{MM}_{n,n}''$ the set of $\pi \in \mathcal{MM}_n$ such that $\pi_n=n$.

Denote by $\rho'$ the mapping defined for $\pi \in \mathcal{MM}_{n,k}''$ by $\rho'(\pi)=\overline{\pi^R}$.

\begin{thm} \label{thmpou3}
For all $n \geq 1$ and $k \in [n]$, the mapping $\rho'$ is a bijection between $\mathcal{MM}_{n,k}'$ and $\mathcal{MM}_{n,k}''$.
 Therefore, the sequence $(\mathcal{MM}_{n,k}'')_{1 \leq k \leq n}$ is an Entringer family.
\end{thm}

\begin{proof} For $k \leq n-1$, $\pi \in \mathcal{MM}_n$ has $k$ just before $1$ if and only if $\rho'(\pi)$ has $n+1-k$ just after $n$.
\end{proof}

\section{New Entringer families}
\label{s-2new}

\subsection{Interpretations in G-words and R-words}
\label{s-gword}

A permutation $\pi$ of $I=\{a_1,\ldots,a_n\}$ with $a_1<\cdots<a_n$ is called a \emph{G-word} if \begin{itemize}
\item[(i)] $\pi_1=a_n$, $\pi_n=a_{n-1},$
\item[(ii)] $\pi_2 > \pi_{n-1}$ (if $n \geq 4$). \end{itemize}
Similarly, a permutation $\pi$ of $I$ is called an \emph{R-word} if previous conditions are satisfied when (ii) is replaced by
\begin{itemize}
\item[(ii')] $\pi_2 < \pi_{n-1}$ (if $n \geq 4$). \end{itemize}
A G-word (resp. an R-word) is said to be \emph{primitive} if for all $1 \leq i < j \leq n$, neither the word $\pi_i \pi_{i+1} \ldots \pi_{j}$ nor the word $\pi_{j} \pi_{j-1} \ldots \pi_{i}$ is a G-word (resp. an R-word).
Denote respectively by $\mathcal{GW}_n$ and $\mathcal{RW}_n$ the set of primitive G-words on $[n+2]$ and primitive R-words on $[n]$.
For examples, the G-words in $\mathcal{GW}_4$ are:
\[ 6 \, 3 \, 4 \, 2 \, 1 \, 5, \quad 6 \, 4 \, 2 \, 3 \, 1 \, 5, \quad 6 \, 2 \, 3 \, 4 \, 1 \, 5, \quad 6 \, 4 \, 3 \, 2 \, 1 \, 5, \quad 6 \, 2 \, 4 \, 3 \, 1 \,5,\]
and the R-words in $\mathcal{RW}_4$ are:
\[ 6 \, 2 \, 1 \, 4 \, 3 \, 5, \quad 6 \, 2 \, 3 \, 1 \, 4 \, 5, \quad 6 \, 1 \, 4 \, 2 \, 3 \, 5, \quad 6 \, 3 \, 1 \, 2 \, 4 \, 5, \quad 6 \, 2 \, 4 \, 1 \, 3 \,5.\]

These permutations were introduced in \cite{Mar06} with the following problem. Let $I_n$ be the ideal of all algebraic relations on the slopes of all lines that can be formed by placing $n$ points in a plane. Then, under two orders, $I_n$ is generated by monomials corresponding to respectively primitive G-words and primitive R-words.

Martin and Wagner proved \cite{MW09} that $E_n$ is the number of primitive G-words (resp. the number of primitive R-words) on $[n+2]$. 
Actually, this result can be refined to Entringer numbers by  introducing a statistic on G-words.

Given a primitive G-word or an R-word $\pi$ on $[n+2]$,  the \emph{route} of $\pi$ is  the sequence $(\alpha_i)$ defined by  the following procedure:
\begin{itemize}
\item[$\bullet$] $\alpha_1=n+2 (=\pi_1)$, $\alpha_2=n+1(=\pi_{n+2})$,
\item[$\bullet$] for $k \geq 2$, if $\alpha_k=\pi_{i}$, define $A_{k}=\{\alpha_{1}, \ldots, \alpha_{k-1}\}$ and 
$$
\alpha_{k+1}=\begin{cases}
\alpha_{k}&\textrm{if $\{\pi_{i-1}, \pi_{i+1}\}\subset A_{k}$},\\
\max \left[ \begin{array}{l} \left\{
    \pi_{j}  | j < i \mbox{~and~}   \pi_j, \pi_{j+1}, \ldots,\pi_{i-1}\not\in A_k \right\} \\  \ \cup
\left\{\pi_{j}  | j > i \mbox{~and~}   \pi_{i+1}, \ldots,\pi_{j-1}, \pi_j \not\in A_{k}\right\} \end{array} \right]&\textrm{otherwise}.
\end{cases}
$$
\end{itemize}
One can represent the route of a G-word or an R-word $\pi$ as a graph with the vertices $\pi_1, \pi_2,\ldots,\pi_n$ ordered in a line, with only one path starting from $n$ drawn upon the line and going successively, if it's possible, to $n-1$, $n-2$, \ldots, $1$ without crossings (see Figure~\ref{fig-gword} for an example). Denote $\mathcal{GW}_{n,k}$ (resp. $\mathcal{RW}_{n,k}$) the set of primitive G-words $\pi$ on $[n+2]$ (resp. primitive R-words $\pi$ on $[n+2]$) such that $\alpha_{n+2} = n+1-k$.

\begin{figure}[!h]
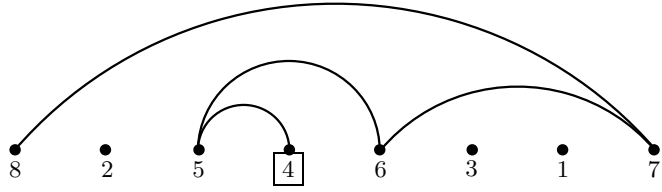

$
\centering
\begin{pgfpicture}{21.27mm}{29.73mm}{110.74mm}{57.38mm}
\pgfsetxvec{\pgfpoint{0.60mm}{0mm}}
\pgfsetyvec{\pgfpoint{0mm}{0.60mm}}
\color[rgb]{0,0,0}\pgfsetlinewidth{0.30mm}\pgfsetdash{}{0mm}
\pgfcircle[fill]{\pgfxy(39.97,60.01)}{0.60mm}
\pgfcircle[stroke]{\pgfxy(39.97,60.01)}{0.60mm}
\pgfcircle[fill]{\pgfxy(59.77,60.01)}{0.60mm}
\pgfcircle[stroke]{\pgfxy(59.77,60.01)}{0.60mm}
\pgfcircle[fill]{\pgfxy(80.28,60.01)}{0.60mm}
\pgfcircle[stroke]{\pgfxy(80.28,60.01)}{0.60mm}
\pgfcircle[fill]{\pgfxy(100.09,60.01)}{0.60mm}
\pgfcircle[stroke]{\pgfxy(100.09,60.01)}{0.60mm}
\pgfcircle[fill]{\pgfxy(119.90,60.01)}{0.60mm}
\pgfcircle[stroke]{\pgfxy(119.90,60.01)}{0.60mm}
\pgfcircle[fill]{\pgfxy(140.17,60.01)}{0.60mm}
\pgfcircle[stroke]{\pgfxy(140.17,60.01)}{0.60mm}
\pgfcircle[fill]{\pgfxy(159.98,60.01)}{0.60mm}
\pgfcircle[stroke]{\pgfxy(159.98,60.01)}{0.60mm}
\pgfcircle[fill]{\pgfxy(180.02,60.01)}{0.60mm}
\pgfcircle[stroke]{\pgfxy(180.02,60.01)}{0.60mm}
\pgfmoveto{\pgfxy(39.97,60.01)}\pgfstroke
\pgfmoveto{\pgfxy(39.97,60.15)}\pgfcurveto{\pgfxy(39.97,60.09)}{\pgfxy(39.99,60.03)}{\pgfxy(40.03,59.98)}\pgfcurveto{\pgfxy(40.08,59.94)}{\pgfxy(40.14,59.91)}{\pgfxy(40.20,59.91)}\pgfcurveto{\pgfxy(40.26,59.91)}{\pgfxy(40.32,59.94)}{\pgfxy(40.36,59.98)}\pgfcurveto{\pgfxy(40.41,60.03)}{\pgfxy(40.43,60.08)}{\pgfxy(40.43,60.15)}\pgfstroke
\pgfmoveto{\pgfxy(180.26,60.14)}\pgfcurveto{\pgfxy(168.63,73.65)}{\pgfxy(153.35,83.51)}{\pgfxy(136.25,88.54)}\pgfcurveto{\pgfxy(119.15,93.57)}{\pgfxy(100.96,93.55)}{\pgfxy(83.87,88.49)}\pgfcurveto{\pgfxy(66.78,83.43)}{\pgfxy(51.52,73.53)}{\pgfxy(39.92,60.00)}\pgfstroke
\pgfmoveto{\pgfxy(180.02,60.15)}\pgfcurveto{\pgfxy(175.08,65.96)}{\pgfxy(168.56,70.21)}{\pgfxy(161.24,72.37)}\pgfcurveto{\pgfxy(153.92,74.53)}{\pgfxy(146.13,74.51)}{\pgfxy(138.83,72.31)}\pgfcurveto{\pgfxy(131.52,70.11)}{\pgfxy(125.02,65.83)}{\pgfxy(120.11,59.98)}\pgfstroke
\pgfmoveto{\pgfxy(119.90,59.91)}\pgfcurveto{\pgfxy(119.85,65.11)}{\pgfxy(117.77,70.08)}{\pgfxy(114.10,73.76)}\pgfcurveto{\pgfxy(110.44,77.45)}{\pgfxy(105.48,79.56)}{\pgfxy(100.29,79.64)}\pgfcurveto{\pgfxy(95.09,79.72)}{\pgfxy(90.07,77.77)}{\pgfxy(86.29,74.21)}\pgfcurveto{\pgfxy(82.52,70.64)}{\pgfxy(80.28,65.74)}{\pgfxy(80.06,60.55)}\pgfstroke
\pgfmoveto{\pgfxy(100.09,59.91)}\pgfcurveto{\pgfxy(100.09,62.55)}{\pgfxy(99.05,65.08)}{\pgfxy(97.20,66.96)}\pgfcurveto{\pgfxy(95.34,68.83)}{\pgfxy(92.82,69.90)}{\pgfxy(90.19,69.93)}\pgfcurveto{\pgfxy(87.55,69.96)}{\pgfxy(85.01,68.95)}{\pgfxy(83.11,67.12)}\pgfcurveto{\pgfxy(81.21,65.29)}{\pgfxy(80.11,62.78)}{\pgfxy(80.05,60.15)}\pgfstroke
\pgfputat{\pgfxy(40.04,53.97)}{\pgfbox[bottom,left]{\fontsize{8.54}{10.24}\selectfont \makebox[0pt]{8}}}
\pgfputat{\pgfxy(60.20,54.03)}{\pgfbox[bottom,left]{\fontsize{8.54}{10.24}\selectfont \makebox[0pt]{2}}}
\pgfputat{\pgfxy(80.16,54.03)}{\pgfbox[bottom,left]{\fontsize{8.54}{10.24}\selectfont \makebox[0pt]{5}}}
\pgfputat{\pgfxy(99.83,54.09)}{\pgfbox[bottom,left]{\fontsize{8.54}{10.24}\selectfont \makebox[0pt]{$\boxed{4}$}}}
\pgfputat{\pgfxy(120.08,54.08)}{\pgfbox[bottom,left]{\fontsize{8.54}{10.24}\selectfont \makebox[0pt]{6}}}
\pgfputat{\pgfxy(140.02,54.30)}{\pgfbox[bottom,left]{\fontsize{8.54}{10.24}\selectfont \makebox[0pt]{3}}}
\pgfputat{\pgfxy(160.04,54.08)}{\pgfbox[bottom,left]{\fontsize{8.54}{10.24}\selectfont \makebox[0pt]{1}}}
\pgfputat{\pgfxy(179.98,54.03)}{\pgfbox[bottom,left]{\fontsize{8.54}{10.24}\selectfont \makebox[0pt]{7}}}
\end{pgfpicture}%
$
\caption{\label{fig-gword} The route of the G-word $\pi=82546317$}
\end{figure}

\begin{thm}
 The sequences $(\mathcal{GW}_{n,k})_{1 \leq k \leq n}$ and $(\mathcal{RW}_{n,k})_{1 \leq k \leq n}$ are Entringer families.
\end{thm}

\begin{proof}
Use the bijection $\delta$ between $\mathcal{GW}_{n}$ and $\Tn$ present in \cite{MW09}. For $\pi$ a primitive G-word on $\{a_1,\ldots,a_{n+2}\}$ with $a_1 < \cdots < a_{n+2}$, denote by $\pi'$ the word $\pi_2\ldots\pi_{n+1}$. If $\pi'$ is a word on $\{a_1,\ldots,a_n \}$, with $a_1 < \cdots < a_n$ and $a_{n}=\pi'_k$ for $k \in \{1,\ldots,n\}$, define $T=\alpha(\pi')$ as the tree with root $a_1$, from which two subgraphs go out, that are $\alpha(\pi'_1 \pi'_2 \ldots \pi'_{k-1})$ and $\alpha(\pi'_{k+1}\pi'_{k+2}\ldots\pi'_n)$ (eventually one of them or both are empty). The tree $\delta(\pi)=\alpha(\pi')$ is a binary  increasing increasing tree and the application $\delta$ is a bijection from $\mathcal{GW}_n$ to $\Tn$ (see \cite{MW09} for further details).

Moreover, it is easy to see that the labels upon the minimal path of $T=\delta(\pi)$ are successively $(n+1-a_1),(n+1-a_2),\ldots,(n+1-a_m)$, where $a_1\ldots,a_m$ ($a_1>\cdots>a_m$) are the different values that appear in the route of $\pi$. Thus, the leaf of the minimal path is $k$. Then, $\delta$ is a bijection between $\mathcal{GW}_{n,k}$ and $\mathcal{T}_{n,k}$.

For example, one can construct the tree that corresponds with the G-word $\pi=82546317$ with this construction:
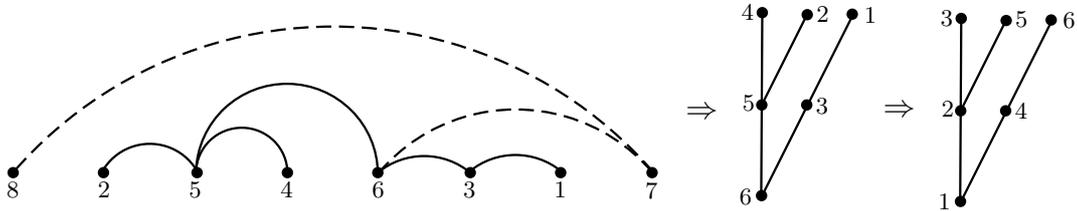
\[
\centering
\begin{pgfpicture}{21.20mm}{27.39mm}{165.57mm}{60.55mm}
\pgfsetxvec{\pgfpoint{0.60mm}{0mm}}
\pgfsetyvec{\pgfpoint{0mm}{0.60mm}}
\color[rgb]{0,0,0}\pgfsetlinewidth{0.30mm}\pgfsetdash{}{0mm}
\pgfcircle[fill]{\pgfxy(39.97,60.01)}{0.60mm}
\pgfcircle[stroke]{\pgfxy(39.97,60.01)}{0.60mm}
\pgfcircle[fill]{\pgfxy(59.77,60.01)}{0.60mm}
\pgfcircle[stroke]{\pgfxy(59.77,60.01)}{0.60mm}
\pgfcircle[fill]{\pgfxy(80.28,60.01)}{0.60mm}
\pgfcircle[stroke]{\pgfxy(80.28,60.01)}{0.60mm}
\pgfcircle[fill]{\pgfxy(100.09,60.01)}{0.60mm}
\pgfcircle[stroke]{\pgfxy(100.09,60.01)}{0.60mm}
\pgfcircle[fill]{\pgfxy(119.90,60.01)}{0.60mm}
\pgfcircle[stroke]{\pgfxy(119.90,60.01)}{0.60mm}
\pgfcircle[fill]{\pgfxy(140.17,60.01)}{0.60mm}
\pgfcircle[stroke]{\pgfxy(140.17,60.01)}{0.60mm}
\pgfcircle[fill]{\pgfxy(159.98,60.01)}{0.60mm}
\pgfcircle[stroke]{\pgfxy(159.98,60.01)}{0.60mm}
\pgfcircle[fill]{\pgfxy(180.02,60.01)}{0.60mm}
\pgfcircle[stroke]{\pgfxy(180.02,60.01)}{0.60mm}
\pgfmoveto{\pgfxy(39.97,60.01)}\pgfstroke
\pgfmoveto{\pgfxy(39.97,60.15)}\pgfcurveto{\pgfxy(39.97,60.09)}{\pgfxy(39.99,60.03)}{\pgfxy(40.03,59.98)}\pgfcurveto{\pgfxy(40.08,59.94)}{\pgfxy(40.14,59.91)}{\pgfxy(40.20,59.91)}\pgfcurveto{\pgfxy(40.26,59.91)}{\pgfxy(40.32,59.94)}{\pgfxy(40.36,59.98)}\pgfcurveto{\pgfxy(40.41,60.03)}{\pgfxy(40.43,60.08)}{\pgfxy(40.43,60.15)}\pgfstroke
\pgfsetdash{{2.00mm}{1.00mm}}{0mm}\pgfmoveto{\pgfxy(180.26,60.14)}\pgfcurveto{\pgfxy(168.63,73.65)}{\pgfxy(153.35,83.51)}{\pgfxy(136.25,88.54)}\pgfcurveto{\pgfxy(119.15,93.57)}{\pgfxy(100.96,93.55)}{\pgfxy(83.87,88.49)}\pgfcurveto{\pgfxy(66.78,83.43)}{\pgfxy(51.52,73.53)}{\pgfxy(39.92,60.00)}\pgfstroke
\pgfmoveto{\pgfxy(180.02,60.15)}\pgfcurveto{\pgfxy(175.08,65.96)}{\pgfxy(168.56,70.21)}{\pgfxy(161.24,72.37)}\pgfcurveto{\pgfxy(153.92,74.53)}{\pgfxy(146.13,74.51)}{\pgfxy(138.83,72.31)}\pgfcurveto{\pgfxy(131.52,70.11)}{\pgfxy(125.02,65.83)}{\pgfxy(120.11,59.98)}\pgfstroke
\pgfsetdash{}{0mm}\pgfmoveto{\pgfxy(119.90,59.91)}\pgfcurveto{\pgfxy(119.85,65.11)}{\pgfxy(117.77,70.08)}{\pgfxy(114.10,73.76)}\pgfcurveto{\pgfxy(110.44,77.45)}{\pgfxy(105.48,79.56)}{\pgfxy(100.29,79.64)}\pgfcurveto{\pgfxy(95.09,79.72)}{\pgfxy(90.07,77.77)}{\pgfxy(86.29,74.21)}\pgfcurveto{\pgfxy(82.52,70.64)}{\pgfxy(80.28,65.74)}{\pgfxy(80.06,60.55)}\pgfstroke
\pgfmoveto{\pgfxy(100.09,59.91)}\pgfcurveto{\pgfxy(100.09,62.55)}{\pgfxy(99.05,65.08)}{\pgfxy(97.20,66.96)}\pgfcurveto{\pgfxy(95.34,68.83)}{\pgfxy(92.82,69.90)}{\pgfxy(90.19,69.93)}\pgfcurveto{\pgfxy(87.55,69.96)}{\pgfxy(85.01,68.95)}{\pgfxy(83.11,67.12)}\pgfcurveto{\pgfxy(81.21,65.29)}{\pgfxy(80.11,62.78)}{\pgfxy(80.05,60.15)}\pgfstroke
\pgfputat{\pgfxy(39.92,54.61)}{\pgfbox[bottom,left]{\fontsize{8.54}{10.24}\selectfont \makebox[0pt]{8}}}
\pgfputat{\pgfxy(59.87,54.59)}{\pgfbox[bottom,left]{\fontsize{8.54}{10.24}\selectfont \makebox[0pt]{2}}}
\pgfputat{\pgfxy(79.88,54.47)}{\pgfbox[bottom,left]{\fontsize{8.54}{10.24}\selectfont \makebox[0pt]{5}}}
\pgfputat{\pgfxy(99.92,54.59)}{\pgfbox[bottom,left]{\fontsize{8.54}{10.24}\selectfont \makebox[0pt]{4}}}
\pgfputat{\pgfxy(119.98,54.61)}{\pgfbox[bottom,left]{\fontsize{8.54}{10.24}\selectfont \makebox[0pt]{6}}}
\pgfputat{\pgfxy(140.00,54.61)}{\pgfbox[bottom,left]{\fontsize{8.54}{10.24}\selectfont \makebox[0pt]{3}}}
\pgfputat{\pgfxy(159.91,54.45)}{\pgfbox[bottom,left]{\fontsize{8.54}{10.24}\selectfont \makebox[0pt]{1}}}
\pgfputat{\pgfxy(179.87,54.36)}{\pgfbox[bottom,left]{\fontsize{8.54}{10.24}\selectfont \makebox[0pt]{7}}}
\pgfmoveto{\pgfxy(140.30,59.89)}\pgfcurveto{\pgfxy(137.45,62.32)}{\pgfxy(133.84,63.66)}{\pgfxy(130.10,63.67)}\pgfcurveto{\pgfxy(126.36,63.69)}{\pgfxy(122.73,62.37)}{\pgfxy(119.87,59.97)}\pgfstroke
\pgfmoveto{\pgfxy(150.24,48.99)}\pgfstroke
\pgfmoveto{\pgfxy(160.02,60.37)}\pgfcurveto{\pgfxy(157.20,62.67)}{\pgfxy(153.67,63.92)}{\pgfxy(150.03,63.91)}\pgfcurveto{\pgfxy(146.39,63.90)}{\pgfxy(142.87,62.63)}{\pgfxy(140.06,60.31)}\pgfstroke
\pgfmoveto{\pgfxy(80.18,60.53)}\pgfcurveto{\pgfxy(78.74,62.96)}{\pgfxy(76.49,64.79)}{\pgfxy(73.83,65.72)}\pgfcurveto{\pgfxy(71.16,66.64)}{\pgfxy(68.25,66.59)}{\pgfxy(65.62,65.57)}\pgfcurveto{\pgfxy(62.99,64.56)}{\pgfxy(60.81,62.64)}{\pgfxy(59.45,60.17)}\pgfstroke
\pgfputat{\pgfxy(187.44,71.85)}{\pgfbox[bottom,left]{\fontsize{11.95}{14.34}\selectfont $\Rightarrow$}}
\pgfcircle[fill]{\pgfxy(204.01,54.93)}{0.60mm}
\pgfcircle[stroke]{\pgfxy(204.01,54.93)}{0.60mm}
\pgfputat{\pgfxy(199.18,53.09)}{\pgfbox[bottom,left]{\fontsize{8.54}{10.24}\selectfont 6}}
\pgfcircle[fill]{\pgfxy(204.06,74.97)}{0.60mm}
\pgfcircle[stroke]{\pgfxy(204.06,74.97)}{0.60mm}
\pgfcircle[fill]{\pgfxy(204.17,95.37)}{0.60mm}
\pgfcircle[stroke]{\pgfxy(204.17,95.37)}{0.60mm}
\pgfcircle[fill]{\pgfxy(214.03,94.92)}{0.60mm}
\pgfcircle[stroke]{\pgfxy(214.03,94.92)}{0.60mm}
\pgfcircle[fill]{\pgfxy(213.92,74.97)}{0.60mm}
\pgfcircle[stroke]{\pgfxy(213.92,74.97)}{0.60mm}
\pgfcircle[fill]{\pgfxy(223.90,95.03)}{0.60mm}
\pgfcircle[stroke]{\pgfxy(223.90,95.03)}{0.60mm}
\pgfmoveto{\pgfxy(203.94,55.13)}\pgflineto{\pgfxy(204.17,95.26)}\pgfstroke
\pgfmoveto{\pgfxy(204.06,54.79)}\pgflineto{\pgfxy(224.12,95.03)}\pgfstroke
\pgfmoveto{\pgfxy(204.06,74.97)}\pgflineto{\pgfxy(214.03,95.03)}\pgfstroke
\pgfputat{\pgfxy(199.86,73.72)}{\pgfbox[bottom,left]{\fontsize{8.54}{10.24}\selectfont 5}}
\pgfputat{\pgfxy(199.64,93.67)}{\pgfbox[bottom,left]{\fontsize{8.54}{10.24}\selectfont 4}}
\pgfputat{\pgfxy(216.07,93.44)}{\pgfbox[bottom,left]{\fontsize{8.54}{10.24}\selectfont 2}}
\pgfputat{\pgfxy(226.50,93.22)}{\pgfbox[bottom,left]{\fontsize{8.54}{10.24}\selectfont 1}}
\pgfputat{\pgfxy(215.96,73.15)}{\pgfbox[bottom,left]{\fontsize{8.54}{10.24}\selectfont 3}}
\pgfcircle[fill]{\pgfxy(247.62,53.64)}{0.60mm}
\pgfcircle[stroke]{\pgfxy(247.62,53.64)}{0.60mm}
\pgfputat{\pgfxy(242.79,51.80)}{\pgfbox[bottom,left]{\fontsize{8.54}{10.24}\selectfont 1}}
\pgfcircle[fill]{\pgfxy(247.66,73.68)}{0.60mm}
\pgfcircle[stroke]{\pgfxy(247.66,73.68)}{0.60mm}
\pgfcircle[fill]{\pgfxy(247.78,94.09)}{0.60mm}
\pgfcircle[stroke]{\pgfxy(247.78,94.09)}{0.60mm}
\pgfcircle[fill]{\pgfxy(257.64,93.63)}{0.60mm}
\pgfcircle[stroke]{\pgfxy(257.64,93.63)}{0.60mm}
\pgfcircle[fill]{\pgfxy(257.53,73.68)}{0.60mm}
\pgfcircle[stroke]{\pgfxy(257.53,73.68)}{0.60mm}
\pgfcircle[fill]{\pgfxy(267.50,93.75)}{0.60mm}
\pgfcircle[stroke]{\pgfxy(267.50,93.75)}{0.60mm}
\pgfmoveto{\pgfxy(247.55,53.84)}\pgflineto{\pgfxy(247.78,93.97)}\pgfstroke
\pgfmoveto{\pgfxy(247.66,53.50)}\pgflineto{\pgfxy(267.73,93.75)}\pgfstroke
\pgfmoveto{\pgfxy(247.66,73.68)}\pgflineto{\pgfxy(257.64,93.75)}\pgfstroke
\pgfputat{\pgfxy(243.47,72.44)}{\pgfbox[bottom,left]{\fontsize{8.54}{10.24}\selectfont 2}}
\pgfputat{\pgfxy(243.24,92.39)}{\pgfbox[bottom,left]{\fontsize{8.54}{10.24}\selectfont 3}}
\pgfputat{\pgfxy(259.68,92.16)}{\pgfbox[bottom,left]{\fontsize{8.54}{10.24}\selectfont 5}}
\pgfputat{\pgfxy(270.11,91.93)}{\pgfbox[bottom,left]{\fontsize{8.54}{10.24}\selectfont 6}}
\pgfputat{\pgfxy(259.57,71.87)}{\pgfbox[bottom,left]{\fontsize{8.54}{10.24}\selectfont 4}}
\pgfputat{\pgfxy(230.73,72.49)}{\pgfbox[bottom,left]{\fontsize{11.95}{14.34}\selectfont $\Rightarrow$}}
\end{pgfpicture}%
\]

The analogous  result for the R-word can be proved using the same method with the bijection $\delta'$ between $\mathcal{RW}_n$ and $\Tn$ present in \cite{MW09}.
\end{proof}

\subsection{Interpretations in U-words} \label{s-sequence}

We introduce here two  new  Entringer familes.

 \begin{defn}
 A \emph{U-word of length $n$} is a sequence $u=(u_i)_{1 \leq i \leq n}$ 
 such that $u_1=1$ and $u_i+u_{i-1} \leq i$ for $i \in \{2,\ldots,n \}$. We denote by $\mathcal{UW}_n$ the set of U-words of length $n$.
\end{defn}

For example, the U-words of length $4$ are:
\[ 1 \, 1 \, 1 \, 1, \quad 1 \, 1 \, 1 \, 2, \quad 1 \, 1 \, 1 \, 3, \quad 1 \, 1 \, 2 \, 1, \quad 1 \, 1 \, 2 \, 2. \]

Denote by $\mathcal{UW}_{n,k}$ the set of U-words $(u_i) \in \mathcal{UW}_n$ such that $u_n=n+1-k$.
\begin{thm} \label{thmunk}
 The sequence $(\mathcal{UW}_{n,k})_{1 \leq k \leq n}$ is an Entringer family.
\end{thm}

\begin{proof}
For any finite set $X$, let $\# X$ denotes its  cardinality.
For $\pi \in \mathcal{DU}_{n,k}$,
let $\gamma(\pi)=w^{RW}$, where  $w=w_{1}\ldots w_{n}$ is the  word defined by
$$
w_{i}=\left\lbrace\begin{array}{ll}
\#\{j \geq \pi_{i}, j \not\in \{\pi_{1},\pi_{2},\ldots,\pi_{i-1}\} \},&\quad \textrm{if $i$ is odd},\\
\#\{j \leq \pi_{i}, j \not\in \{\pi_{1},\pi_{2},\ldots,\pi_{i-1}\} \},&\quad \textrm{if $i$ is even}.
\end{array}\right.
$$
For example, if $\pi= 6 \, 3 \, 5 \, 1 \, 7 \, 2 \, 4 \in \mathcal{DU}_{7,6}$, then the word $w$ is computed as follows:
\begin{itemize}
\item $\{ j \geq 6 \} = \{ 6,7 \}$, so $w_1=2$,
\item $\{ j \leq 3, j \neq 6 \} = \{1,2,3 \}$, so $w_2=3$,
\item $\{ j \geq 5, j \not\in\{3,6 \} \} = \{5,7 \}$, so $w_3=2$,
\item $\{ j \leq 1, j \not\in\{3,5,6\} \} = \{1\}$, so $w_4=1$,
\item $\{ j \geq 7, j \not\in\{1,3,5,6 \} \} = \{7 \}$, so $w_5=1$,
\item $\{ j \leq 2, j \not\in\{1,3,5,6,7 \} \} = \{2 \}$, so $w_6=1$,
\item $\{ j \geq 4, j \not\in\{1,2,3,5,6,7 \} \} = \{4 \}$, so $w_7=1$,
\end{itemize}
Then, $w= 2 \, 3 \, 2 \, 1 \, 1 \, 1 \, 1$ and $\gamma(\pi) = 1 \, 1 \, 1 \, 1 \, 2 \, 3 \, 2$.

We show that the mapping  $\gamma$ is a bijection between  $\Ank$ and $\mathcal{UW}_{n,k}$.
Following the construction, $\gamma(\pi)_n=w_1 = n+1-\pi_1=n+1-k$. Moreover, when  $\gamma(\pi)_i=w_{n+1-i}$ is written, $n-i$ elements have been read in $\pi$ before, thus the number of elements counted by $\gamma(\pi)_i$ must be less than $i$. Moreover, the numbers counted by $\gamma(\pi)_{i-1}$ and $\gamma(\pi)_i$ are in the $n-i$ elements that have not been read in $\pi$ and are two disjoint sets since $\pi$ is down-up. Thus $\gamma(\pi)_i+\gamma(\pi)_{i-1}$ must be less than $i$. Finally, $\gamma(\pi) \in \mathcal{UW}_{n,k}$.

Conversely, if $u \in \mathcal{UW}_{n,k}$, the permutation $\pi=\gamma^{-1}(u) \in \Ank$ can be recovered with:
\begin{itemize}
\item $\pi_1=n+1-u_n$
\item $\forall n \geq 1$, $\pi_{2i}$ is the $u_{n-2i+1}$-st smallest element in $[n] \setminus \{ \pi_1,\ldots,\pi_{2i-1} \}$.
\item $\forall n \geq 1$, $\pi_{2i+1}$ is the $u_{n-2i}$-st greatest element in $[n] \setminus \{ \pi_1,\ldots,\pi_{2i} \}$.
\end{itemize}
We are done.
\end{proof}

Denote by $\mathcal{UW}'_{n,k}$ the set of U-words $(u_i) \in \mathcal{UW}_n$ such that $u_{n-1}+u_n=k$.

\begin{thm} \label{thmvnk}
The sequence $(\mathcal{UW}'_{n,k})_{1 \leq k \leq n}$ is an Entringer family.
\end{thm}

\begin{proof}
There are two possibilities to prove this  result.

Firstly, the mapping $\gamma$ also induces a bijection between $\mathcal{DU}'_{n,k}$ and $\mathcal{UW}'_{n,k}$. For $\pi \in \mathcal{DU}_{n,k}'$, there exists $j \in [n]$ such that $\pi \in \mathcal{DU}_{n,j}$, so we can define $v=\gamma(\pi) \in \mathcal{UW}_{n,j} \subset \mathcal{UW}_n$. It suffices to show that $v \in \mathcal{UW}_{n,k}'$. In the construction of $\gamma(\pi)$, $v_n$ is the number of elements that are greater than $\pi_1$, and $v_{n-1}$ is the number of elements that are less than $\pi_2$. Then $v_n=n+1-\pi_1$ and $v_{n-1}=\pi_2$, and $v_{n-1}+v_n=n+1-(\pi_1-\pi_2)=k$ since $\pi \in \mathcal{DU}'_{n,k}$.

Secondly, it is easy to construct a bijection $\alpha: \mathcal{UW}_{n,k}\longrightarrow  \mathcal{UW}_{n,k}'$.
For $u=(u_1,\ldots,u_n) \in \mathcal{UW}_{n,k}$, let  $\alpha(u)=(u_1,u_2,\ldots,u_{n-1},n+1-u_{n-1}-u_{n})$. 
Since $u \in \mathcal{UW}_{n,k}$, $u_{n} - u_{n-1} \leq n+1$, so we have $\alpha(u) \in \mathcal{UW}_n$. 
Moreover, the last element $\alpha(u)_n=n+1-u_{n-1}-(n+1-k)=k-u_{n-1}=k-\alpha(u)_{n-1}$, so $\alpha(u) \in \mathcal{UW}_{n,k}'$. 
The mapping $\alpha$ is then clearly a bijection between $\mathcal{UW}_{n,k}$ and $\mathcal{UW}_{n,k}'$.
\end{proof}

It follows immediately from  the above theorems that the Euler number $E_{n}$ is the number of U-words of length $n$ for all integer $n\geq 1$.
\section{Concluding remarks}
\subsection{List of  bijections for Entringer families}

In what follows, we list all the twelve interpretations for Entringer families along with the bijections dicussed in this paper:
\begin{enumerate}
\item the permutation $\pi \in \Ank$ such that $\pi_1=k$,
\item the encoding sequence $\Delta \in \ESnk$, obtained by $\Delta=\psi(\pi)$, where $\psi$ is the bijection described in Section~\ref{s-dnk}, then $k$ is the first element read in $\Delta$,
\item the binary  increasing increasing tree $T \in \mathcal{BT}_{n,k}$, obtained by $T=\varphi(\Delta)$, where $\varphi$ is the bijection described in Section~\ref{s-bij}, then $k$ is the leaf of the minimal path of $T$,
\item the binary  increasing increasing tree $T' \in \mathcal{BT}_{n,k}'$, obtained by $T'=\varphi'(T)$, where $\varphi'$ is the bijection described in \cite[\S 6]{KPP94}, then $k-1$ is the parent of $n$ in $T'$,
\item the down-up permutation $\sigma \in \Ank'$, obtained by $\sigma=\theta(\pi)$, where $\theta$ is the bijection described in Subsection~\ref{s-theta}, then $k=n+1-\sigma_1+\sigma_2$,
\item the min-max alternating permutation $\sigma' \in \mathcal{MM}_{n,k}$, obtained by $\sigma'=\beta(\sigma)$, where $\beta$ is the bijection described in Subsection~\ref{s-direct}, then $k=n+1-|\sigma_1-\sigma_2|$,
\item the min-max alternating permutation $\tau_1 \in \mathcal{MM}_{n,k}'$, obtained by $\tau_1=\rho(\pi)$, where $\rho$ is the bijection described in Subsection~\ref{s-direct}, then $k$ is the term immediately before $1$ (or $n$ if $\tau_1$ starts with $1$),
\item the min-max alternating permutation $\tau_2 \in \mathcal{MM}_{n,k}''$, obtained by $\tau_2=\rho'(\tau_2)$, where $\rho'$ is the bijection described in Subsection~\ref{s-direct}, then $n+1-k$ is the term immediately after $n$ (or $1$ if $\tau_2$ ends with $n$),
\item the G-word $\pi' \in \mathcal{GW}_{n,k}$, obtained by $\pi'=\delta^{-1}(T)$, where $\delta$ is the bijection described in Subsection~\ref{s-gword}, then $n+1-k$ is the end of the route of $\pi'$,
\item the R-word $\pi'' \in \mathcal{RW}_{n,k}$, obtained by $\pi''=(\delta')^{-1}(T)$, where $\delta'$ is the bijection described in Subsection~\ref{s-gword}, then $n+1-k$ is the end of the route of $\pi'$,
\item the sequence $u \in \mathcal{UW}_{n,k}$, obtained by $u=\gamma(\pi)$, where $\gamma$ is the bijection described in Subsection~\ref{s-sequence}, then $n+1-k$ is the last element of $u$,
\item the sequence $v \in \mathcal{UW}_{n,k}'$, obtained by $v=\gamma(\sigma)=\alpha(u)$, where $\alpha$ and $\gamma$ are the bijections described in Subsection~\ref{s-sequence}, then $k$ is the sum of the two last elements of $v$.
\end{enumerate}

We summarize the bijections of this  paper in the diagram of Figure~\ref{fig:diag}, where at the left we gather all the models in down-up permutations, and at the right we gather the models in the increasing trees.

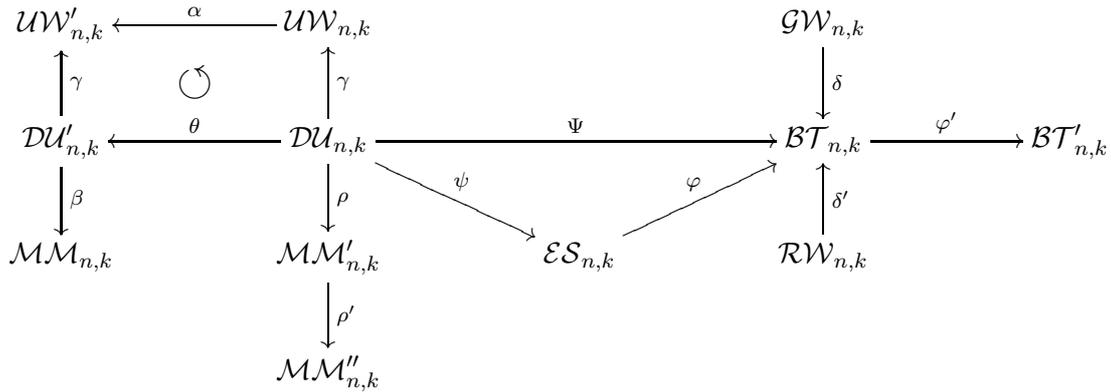
\begin{figure}[h]
\[
\xymatrix{
   \mathcal{UW}_{n,k}' \ar@{}[drr]|*+\txt{\LARGE$\circlearrowleft$}
&& \mathcal{UW}_{n,k} \ar[ll]_{\alpha}
&&
&& \mathcal{GW}_{n,k} \ar[d]^{\delta}\\
   \mathcal{DU}_{n,k}' \ar[d]^{\beta} \ar[u]_{\gamma}
&& \mathcal{DU}_{n,k} \ar[rrd]^{\psi} \ar[ll]_{\theta} \ar[u]_{\gamma} \ar[d]^{\rho} \ar[rrrr]^{\Psi}
&&
&& \mathcal{BT}_{n,k} \ar[rr]^{\varphi'}
&& \mathcal{BT}_{n,k}' \\
   \mathcal{MM}_{n,k}
&& \mathcal{MM}'_{n,k} \ar[d]^{\rho'}
&& \ESnk\ar[rru]^{\varphi}
&& \mathcal{RW}_{n,k} \ar[u]_{\delta'}
&& \\
&& \mathcal{MM}''_{n,k}
}
\]
\caption{\label{fig:diag}The bijections mentioned in the paper}
\end{figure}

\subsection{Illustration for $n=4$}

In Figure~\ref{fig:E4}, we summarize twelve interpretations for $E_{4,k}$, $k \in \{2,3,4\}$. In every column, the corresponding elements are described via the different bijections mentioned in the paper.
Moreover, in the table, boxes point out the statistic $k=\pi_1$ if $\pi \in \mathcal{DU}_{n,k}$ and the corresponding statistics in the other models.
\begin{figure}[p]
 \renewcommand{\arraystretch}{1.8}
\begin{tabular}{|c|c|c|c|c|c|c|}
\hline
&$k$ & $2$ & \multicolumn{2}{|c|}{$3$} & \multicolumn{2}{|c|}{$4$} \\
\hline
(1)&$\pi \in \A_{4,k}$ & $\boxed{2} 1 \, 4 \, 3$ & $\boxed{3} 2 \, 4 \, 1$ & $\boxed{3} 1 \, 4 \, 2$ & $\boxed{4} 2 \, 3 \, 1$ & $\boxed{4} 1 \, 3 \, 2$ \\
\hline
(2)&$\Delta \in \mathcal{ES}_{4,k}$ & $\begin{array}{c} (\boxed{2},1)^{*}  \\ (4,3)^{*}  \end{array}$ & $\begin{array}{c} (\boxed{3},2)^{*}  \\ (4,1)^{*}  \end{array}$ & $\begin{array}{c} (\boxed{3},2) \\ (2,1)^{*}  \\ (4,3)^{*}  \end{array}$ & $\begin{array}{c} (\boxed{4},3) \\ (3,2)^{*}  \\ (4,1)^{*}  \end{array}$ & $\begin{array}{c} (\boxed{4},3) \\ (3,2) \\ (2,1)^{*}  \\ (4,3)^{*}  \end{array}$ \\
\hline
\centering
\begin{pgfpicture}{4.03mm}{6.36mm}{10.43mm}{21.62mm}
\pgfsetxvec{\pgfpoint{0.70mm}{0mm}}
\pgfsetyvec{\pgfpoint{0mm}{0.70mm}}
\color[rgb]{0,0,0}\pgfsetlinewidth{0.30mm}\pgfsetdash{}{0mm}
\pgfputat{\pgfxy(7.43,17.95)}{\pgfbox[bottom,left]{\fontsize{9.96}{11.95}\selectfont $(3)$}}
\end{pgfpicture}&
\centering
\begin{pgfpicture}{4.03mm}{6.36mm}{22.43mm}{21.62mm}
\pgfsetxvec{\pgfpoint{0.70mm}{0mm}}
\pgfsetyvec{\pgfpoint{0mm}{0.70mm}}
\color[rgb]{0,0,0}\pgfsetlinewidth{0.30mm}\pgfsetdash{}{0mm}
\pgfputat{\pgfxy(11.43,17.95)}{\pgfbox[bottom,left]{\fontsize{9.96}{11.95}\selectfont $T \in \mathcal{BT}_{4,k}$}}
\end{pgfpicture}%
&
\centering
\begin{pgfpicture}{1.83mm}{3.70mm}{15.30mm}{17.50mm}
\pgfsetxvec{\pgfpoint{0.70mm}{0mm}}
\pgfsetyvec{\pgfpoint{0mm}{0.70mm}}
\color[rgb]{0,0,0}\pgfsetlinewidth{0.30mm}\pgfsetdash{}{0mm}
\pgfcircle[fill]{\pgfxy(10.00,10.00)}{0.70mm}
\pgfcircle[stroke]{\pgfxy(10.00,10.00)}{0.70mm}
\pgfcircle[fill]{\pgfxy(10.00,15.00)}{0.70mm}
\pgfcircle[stroke]{\pgfxy(10.00,15.00)}{0.70mm}
\pgfcircle[fill]{\pgfxy(15.00,15.00)}{0.70mm}
\pgfcircle[stroke]{\pgfxy(15.00,15.00)}{0.70mm}
\pgfcircle[fill]{\pgfxy(15.00,20.00)}{0.70mm}
\pgfcircle[stroke]{\pgfxy(15.00,20.00)}{0.70mm}
\pgfmoveto{\pgfxy(10.00,15.00)}\pgflineto{\pgfxy(10.00,10.00)}\pgfstroke
\pgfmoveto{\pgfxy(10.00,10.00)}\pgflineto{\pgfxy(15.00,15.00)}\pgfstroke
\pgfmoveto{\pgfxy(15.00,20.00)}\pgflineto{\pgfxy(15.00,15.00)}\pgfstroke
\pgfputat{\pgfxy(8.00,14.00)}{\pgfbox[bottom,left]{\fontsize{7.97}{9.56}\selectfont \makebox[0pt][r]{2}}}
\pgfputat{\pgfxy(8.00,9.00)}{\pgfbox[bottom,left]{\fontsize{7.97}{9.56}\selectfont \makebox[0pt][r]{1}}}
\pgfputat{\pgfxy(17.00,14.00)}{\pgfbox[bottom,left]{\fontsize{7.97}{9.56}\selectfont 3}}
\pgfputat{\pgfxy(17.00,19.00)}{\pgfbox[bottom,left]{\fontsize{7.97}{9.56}\selectfont 4}}
\pgfsetlinewidth{0.15mm}\pgfmoveto{\pgfxy(5.47,13.36)}\pgflineto{\pgfxy(8.52,13.36)}\pgflineto{\pgfxy(8.52,17.27)}\pgflineto{\pgfxy(5.47,17.27)}\pgfclosepath\pgfstroke
\end{pgfpicture}%
&
\centering
\begin{pgfpicture}{1.86mm}{3.70mm}{15.30mm}{17.61mm}
\pgfsetxvec{\pgfpoint{0.70mm}{0mm}}
\pgfsetyvec{\pgfpoint{0mm}{0.70mm}}
\color[rgb]{0,0,0}\pgfsetlinewidth{0.30mm}\pgfsetdash{}{0mm}
\pgfcircle[fill]{\pgfxy(10.00,10.00)}{0.70mm}
\pgfcircle[stroke]{\pgfxy(10.00,10.00)}{0.70mm}
\pgfcircle[fill]{\pgfxy(10.00,15.00)}{0.70mm}
\pgfcircle[stroke]{\pgfxy(10.00,15.00)}{0.70mm}
\pgfcircle[fill]{\pgfxy(10.00,20.00)}{0.70mm}
\pgfcircle[stroke]{\pgfxy(10.00,20.00)}{0.70mm}
\pgfcircle[fill]{\pgfxy(15.00,20.00)}{0.70mm}
\pgfcircle[stroke]{\pgfxy(15.00,20.00)}{0.70mm}
\pgfmoveto{\pgfxy(10.00,15.00)}\pgflineto{\pgfxy(10.00,10.00)}\pgfstroke
\pgfmoveto{\pgfxy(10.00,15.00)}\pgflineto{\pgfxy(15.00,20.00)}\pgfstroke
\pgfmoveto{\pgfxy(10.00,20.00)}\pgflineto{\pgfxy(10.00,15.00)}\pgfstroke
\pgfputat{\pgfxy(8.00,14.00)}{\pgfbox[bottom,left]{\fontsize{7.97}{9.56}\selectfont \makebox[0pt][r]{2}}}
\pgfputat{\pgfxy(8.00,9.00)}{\pgfbox[bottom,left]{\fontsize{7.97}{9.56}\selectfont \makebox[0pt][r]{1}}}
\pgfputat{\pgfxy(8.00,19.00)}{\pgfbox[bottom,left]{\fontsize{7.97}{9.56}\selectfont \makebox[0pt][r]{3}}}
\pgfputat{\pgfxy(17.00,19.00)}{\pgfbox[bottom,left]{\fontsize{7.97}{9.56}\selectfont 4}}
\pgfsetlinewidth{0.15mm}\pgfmoveto{\pgfxy(5.52,18.38)}\pgflineto{\pgfxy(8.28,18.38)}\pgflineto{\pgfxy(8.28,22.30)}\pgflineto{\pgfxy(5.52,22.30)}\pgfclosepath\pgfstroke
\end{pgfpicture}%
 &
 \centering
\begin{pgfpicture}{1.77mm}{3.70mm}{15.30mm}{17.60mm}
\pgfsetxvec{\pgfpoint{0.70mm}{0mm}}
\pgfsetyvec{\pgfpoint{0mm}{0.70mm}}
\color[rgb]{0,0,0}\pgfsetlinewidth{0.30mm}\pgfsetdash{}{0mm}
\pgfcircle[fill]{\pgfxy(10.00,10.00)}{0.70mm}
\pgfcircle[stroke]{\pgfxy(10.00,10.00)}{0.70mm}
\pgfcircle[fill]{\pgfxy(10.00,15.00)}{0.70mm}
\pgfcircle[stroke]{\pgfxy(10.00,15.00)}{0.70mm}
\pgfcircle[fill]{\pgfxy(10.00,20.00)}{0.70mm}
\pgfcircle[stroke]{\pgfxy(10.00,20.00)}{0.70mm}
\pgfcircle[fill]{\pgfxy(15.00,15.00)}{0.70mm}
\pgfcircle[stroke]{\pgfxy(15.00,15.00)}{0.70mm}
\pgfmoveto{\pgfxy(10.00,15.00)}\pgflineto{\pgfxy(10.00,10.00)}\pgfstroke
\pgfmoveto{\pgfxy(10.00,10.00)}\pgflineto{\pgfxy(15.00,15.00)}\pgfstroke
\pgfmoveto{\pgfxy(10.00,20.00)}\pgflineto{\pgfxy(10.00,15.00)}\pgfstroke
\pgfputat{\pgfxy(8.00,14.00)}{\pgfbox[bottom,left]{\fontsize{7.97}{9.56}\selectfont \makebox[0pt][r]{2}}}
\pgfputat{\pgfxy(8.00,9.00)}{\pgfbox[bottom,left]{\fontsize{7.97}{9.56}\selectfont \makebox[0pt][r]{1}}}
\pgfputat{\pgfxy(8.00,19.00)}{\pgfbox[bottom,left]{\fontsize{7.97}{9.56}\selectfont \makebox[0pt][r]{3}}}
\pgfputat{\pgfxy(17.00,14.00)}{\pgfbox[bottom,left]{\fontsize{7.97}{9.56}\selectfont 4}}
\pgfsetlinewidth{0.15mm}\pgfmoveto{\pgfxy(5.38,18.19)}\pgflineto{\pgfxy(8.36,18.19)}\pgflineto{\pgfxy(8.36,22.28)}\pgflineto{\pgfxy(5.38,22.28)}\pgfclosepath\pgfstroke
\end{pgfpicture}%
 &
 \centering
\begin{pgfpicture}{15.58mm}{3.70mm}{23.70mm}{20.91mm}
\pgfsetxvec{\pgfpoint{0.70mm}{0mm}}
\pgfsetyvec{\pgfpoint{0mm}{0.70mm}}
\color[rgb]{0,0,0}\pgfsetlinewidth{0.30mm}\pgfsetdash{}{0mm}
\pgfcircle[fill]{\pgfxy(30.00,10.00)}{0.70mm}
\pgfcircle[stroke]{\pgfxy(30.00,10.00)}{0.70mm}
\pgfcircle[fill]{\pgfxy(30.00,15.00)}{0.70mm}
\pgfcircle[stroke]{\pgfxy(30.00,15.00)}{0.70mm}
\pgfcircle[fill]{\pgfxy(30.00,20.00)}{0.70mm}
\pgfcircle[stroke]{\pgfxy(30.00,20.00)}{0.70mm}
\pgfcircle[fill]{\pgfxy(30.00,25.00)}{0.70mm}
\pgfcircle[stroke]{\pgfxy(30.00,25.00)}{0.70mm}
\pgfmoveto{\pgfxy(30.00,15.00)}\pgflineto{\pgfxy(30.00,10.00)}\pgfstroke
\pgfputat{\pgfxy(28.00,14.00)}{\pgfbox[bottom,left]{\fontsize{7.97}{9.56}\selectfont \makebox[0pt][r]{2}}}
\pgfputat{\pgfxy(28.00,9.00)}{\pgfbox[bottom,left]{\fontsize{7.97}{9.56}\selectfont \makebox[0pt][r]{1}}}
\pgfputat{\pgfxy(28.00,18.50)}{\pgfbox[bottom,left]{\fontsize{7.97}{9.56}\selectfont \makebox[0pt][r]{3}}}
\pgfputat{\pgfxy(27.84,23.48)}{\pgfbox[bottom,left]{\fontsize{7.97}{9.56}\selectfont \makebox[0pt][r]{4}}}
\pgfmoveto{\pgfxy(30.00,20.00)}\pgflineto{\pgfxy(30.00,15.00)}\pgfstroke
\pgfmoveto{\pgfxy(30.00,25.00)}\pgflineto{\pgfxy(30.00,20.00)}\pgfstroke
\pgfsetlinewidth{0.15mm}\pgfmoveto{\pgfxy(25.12,22.55)}\pgflineto{\pgfxy(28.38,22.55)}\pgflineto{\pgfxy(28.38,27.01)}\pgflineto{\pgfxy(25.12,27.01)}\pgfclosepath\pgfstroke
\end{pgfpicture}%
&
$
\centering
\begin{pgfpicture}{15.84mm}{3.70mm}{29.30mm}{17.68mm}
\pgfsetxvec{\pgfpoint{0.70mm}{0mm}}
\pgfsetyvec{\pgfpoint{0mm}{0.70mm}}
\color[rgb]{0,0,0}\pgfsetlinewidth{0.30mm}\pgfsetdash{}{0mm}
\pgfcircle[fill]{\pgfxy(30.00,10.00)}{0.70mm}
\pgfcircle[stroke]{\pgfxy(30.00,10.00)}{0.70mm}
\pgfcircle[fill]{\pgfxy(30.00,15.00)}{0.70mm}
\pgfcircle[stroke]{\pgfxy(30.00,15.00)}{0.70mm}
\pgfcircle[fill]{\pgfxy(30.00,20.00)}{0.70mm}
\pgfcircle[stroke]{\pgfxy(30.00,20.00)}{0.70mm}
\pgfcircle[fill]{\pgfxy(35.00,15.00)}{0.70mm}
\pgfcircle[stroke]{\pgfxy(35.00,15.00)}{0.70mm}
\pgfmoveto{\pgfxy(30.00,15.00)}\pgflineto{\pgfxy(30.00,10.00)}\pgfstroke
\pgfmoveto{\pgfxy(30.00,10.00)}\pgflineto{\pgfxy(35.00,15.00)}\pgfstroke
\pgfmoveto{\pgfxy(30.00,15.00)}\pgflineto{\pgfxy(30.00,20.00)}\pgfstroke
\pgfputat{\pgfxy(28.00,14.00)}{\pgfbox[bottom,left]{\fontsize{7.97}{9.56}\selectfont \makebox[0pt][r]{2}}}
\pgfputat{\pgfxy(28.00,9.00)}{\pgfbox[bottom,left]{\fontsize{7.97}{9.56}\selectfont \makebox[0pt][r]{1}}}
\pgfputat{\pgfxy(28.00,19.00)}{\pgfbox[bottom,left]{\fontsize{7.97}{9.56}\selectfont \makebox[0pt][r]{4}}}
\pgfputat{\pgfxy(37.00,14.00)}{\pgfbox[bottom,left]{\fontsize{7.97}{9.56}\selectfont 3}}
\pgfsetlinewidth{0.15mm}\pgfmoveto{\pgfxy(25.48,22.41)}\pgflineto{\pgfxy(25.48,17.98)}\pgflineto{\pgfxy(28.58,17.98)}\pgflineto{\pgfxy(28.58,22.41)}\pgfclosepath\pgfstroke
\end{pgfpicture}%
$ \\
 \hline
 \centering
\begin{pgfpicture}{5.03mm}{5.00mm}{10.32mm}{23.03mm}
\pgfsetxvec{\pgfpoint{0.70mm}{0mm}}
\pgfsetyvec{\pgfpoint{0mm}{0.70mm}}
\color[rgb]{0,0,0}\pgfsetlinewidth{0.30mm}\pgfsetdash{}{0mm}
\pgfputat{\pgfxy(8,18.13)}{\pgfbox[bottom,left]{\fontsize{9.96}{11.95}\selectfont $(4)$}}
\end{pgfpicture}&
\centering
\begin{pgfpicture}{5.03mm}{5.00mm}{25.32mm}{23.03mm}
\pgfsetxvec{\pgfpoint{0.70mm}{0mm}}
\pgfsetyvec{\pgfpoint{0mm}{0.70mm}}
\color[rgb]{0,0,0}\pgfsetlinewidth{0.30mm}\pgfsetdash{}{0mm}
\pgfputat{\pgfxy(13.46,18.13)}{\pgfbox[bottom,left]{\fontsize{9.96}{11.95}\selectfont $T' \in \mathcal{BT}_{4,k}'$}}
\end{pgfpicture}%
&
\centering
\begin{pgfpicture}{1.81mm}{3.70mm}{15.30mm}{17.50mm}
\pgfsetxvec{\pgfpoint{0.70mm}{0mm}}
\pgfsetyvec{\pgfpoint{0mm}{0.70mm}}
\color[rgb]{0,0,0}\pgfsetlinewidth{0.30mm}\pgfsetdash{}{0mm}
\pgfcircle[fill]{\pgfxy(10.00,10.00)}{0.70mm}
\pgfcircle[stroke]{\pgfxy(10.00,10.00)}{0.70mm}
\pgfcircle[fill]{\pgfxy(10.00,15.00)}{0.70mm}
\pgfcircle[stroke]{\pgfxy(10.00,15.00)}{0.70mm}
\pgfcircle[fill]{\pgfxy(10.00,20.00)}{0.70mm}
\pgfcircle[stroke]{\pgfxy(10.00,20.00)}{0.70mm}
\pgfcircle[fill]{\pgfxy(15.00,15.00)}{0.70mm}
\pgfcircle[stroke]{\pgfxy(15.00,15.00)}{0.70mm}
\pgfmoveto{\pgfxy(10.00,15.00)}\pgflineto{\pgfxy(10.00,10.00)}\pgfstroke
\pgfmoveto{\pgfxy(10.00,10.00)}\pgflineto{\pgfxy(15.00,15.00)}\pgfstroke
\pgfmoveto{\pgfxy(10.00,20.00)}\pgflineto{\pgfxy(10.00,15.00)}\pgfstroke
\pgfputat{\pgfxy(8.00,14.00)}{\pgfbox[bottom,left]{\fontsize{7.97}{9.56}\selectfont \makebox[0pt][r]{2}}}
\pgfputat{\pgfxy(8.00,9.00)}{\pgfbox[bottom,left]{\fontsize{7.97}{9.56}\selectfont \makebox[0pt][r]{1}}}
\pgfputat{\pgfxy(8.00,19.00)}{\pgfbox[bottom,left]{\fontsize{7.97}{9.56}\selectfont \makebox[0pt][r]{3}}}
\pgfputat{\pgfxy(17.00,14.00)}{\pgfbox[bottom,left]{\fontsize{7.97}{9.56}\selectfont 4}}
\pgfsetlinewidth{0.15mm}\pgfmoveto{\pgfxy(5.44,8.43)}\pgflineto{\pgfxy(8.42,8.43)}\pgflineto{\pgfxy(8.42,12.51)}\pgflineto{\pgfxy(5.44,12.51)}\pgfclosepath\pgfstroke
\end{pgfpicture}%
 &
 \centering
\begin{pgfpicture}{1.92mm}{3.70mm}{15.30mm}{17.50mm}
\pgfsetxvec{\pgfpoint{0.70mm}{0mm}}
\pgfsetyvec{\pgfpoint{0mm}{0.70mm}}
\color[rgb]{0,0,0}\pgfsetlinewidth{0.30mm}\pgfsetdash{}{0mm}
\pgfcircle[fill]{\pgfxy(10.00,10.00)}{0.70mm}
\pgfcircle[stroke]{\pgfxy(10.00,10.00)}{0.70mm}
\pgfcircle[fill]{\pgfxy(10.00,15.00)}{0.70mm}
\pgfcircle[stroke]{\pgfxy(10.00,15.00)}{0.70mm}
\pgfcircle[fill]{\pgfxy(10.00,20.00)}{0.70mm}
\pgfcircle[stroke]{\pgfxy(10.00,20.00)}{0.70mm}
\pgfcircle[fill]{\pgfxy(15.00,20.00)}{0.70mm}
\pgfcircle[stroke]{\pgfxy(15.00,20.00)}{0.70mm}
\pgfmoveto{\pgfxy(10.00,15.00)}\pgflineto{\pgfxy(10.00,10.00)}\pgfstroke
\pgfmoveto{\pgfxy(10.00,15.00)}\pgflineto{\pgfxy(15.00,20.00)}\pgfstroke
\pgfmoveto{\pgfxy(10.00,20.00)}\pgflineto{\pgfxy(10.00,15.00)}\pgfstroke
\pgfputat{\pgfxy(8.00,14.00)}{\pgfbox[bottom,left]{\fontsize{7.97}{9.56}\selectfont \makebox[0pt][r]{2}}}
\pgfputat{\pgfxy(8.00,9.00)}{\pgfbox[bottom,left]{\fontsize{7.97}{9.56}\selectfont \makebox[0pt][r]{1}}}
\pgfputat{\pgfxy(8.00,19.00)}{\pgfbox[bottom,left]{\fontsize{7.97}{9.56}\selectfont \makebox[0pt][r]{3}}}
\pgfputat{\pgfxy(17.00,19.00)}{\pgfbox[bottom,left]{\fontsize{7.97}{9.56}\selectfont 4}}
\pgfsetlinewidth{0.15mm}\pgfmoveto{\pgfxy(5.60,13.31)}\pgflineto{\pgfxy(8.36,13.31)}\pgflineto{\pgfxy(8.36,17.23)}\pgflineto{\pgfxy(5.60,17.23)}\pgfclosepath\pgfstroke
\end{pgfpicture}%
&
$
\centering
\begin{pgfpicture}{15.78mm}{3.70mm}{29.30mm}{17.50mm}
\pgfsetxvec{\pgfpoint{0.70mm}{0mm}}
\pgfsetyvec{\pgfpoint{0mm}{0.70mm}}
\color[rgb]{0,0,0}\pgfsetlinewidth{0.30mm}\pgfsetdash{}{0mm}
\pgfcircle[fill]{\pgfxy(30.00,10.00)}{0.70mm}
\pgfcircle[stroke]{\pgfxy(30.00,10.00)}{0.70mm}
\pgfcircle[fill]{\pgfxy(30.00,15.00)}{0.70mm}
\pgfcircle[stroke]{\pgfxy(30.00,15.00)}{0.70mm}
\pgfcircle[fill]{\pgfxy(30.00,20.00)}{0.70mm}
\pgfcircle[stroke]{\pgfxy(30.00,20.00)}{0.70mm}
\pgfcircle[fill]{\pgfxy(35.00,15.00)}{0.70mm}
\pgfcircle[stroke]{\pgfxy(35.00,15.00)}{0.70mm}
\pgfmoveto{\pgfxy(30.00,15.00)}\pgflineto{\pgfxy(30.00,10.00)}\pgfstroke
\pgfmoveto{\pgfxy(30.00,10.00)}\pgflineto{\pgfxy(35.00,15.00)}\pgfstroke
\pgfmoveto{\pgfxy(30.00,15.00)}\pgflineto{\pgfxy(30.00,20.00)}\pgfstroke
\pgfputat{\pgfxy(28.00,14.00)}{\pgfbox[bottom,left]{\fontsize{7.97}{9.56}\selectfont \makebox[0pt][r]{2}}}
\pgfputat{\pgfxy(28.00,9.00)}{\pgfbox[bottom,left]{\fontsize{7.97}{9.56}\selectfont \makebox[0pt][r]{1}}}
\pgfputat{\pgfxy(28.00,19.00)}{\pgfbox[bottom,left]{\fontsize{7.97}{9.56}\selectfont \makebox[0pt][r]{4}}}
\pgfputat{\pgfxy(37.00,14.00)}{\pgfbox[bottom,left]{\fontsize{7.97}{9.56}\selectfont 3}}
\pgfsetlinewidth{0.15mm}\pgfmoveto{\pgfxy(25.40,13.12)}\pgflineto{\pgfxy(28.50,13.12)}\pgflineto{\pgfxy(28.50,17.55)}\pgflineto{\pgfxy(25.40,17.55)}\pgfclosepath\pgfstroke
\end{pgfpicture}%
$ &
\centering
\begin{pgfpicture}{15.61mm}{3.70mm}{23.70mm}{20.63mm}
\pgfsetxvec{\pgfpoint{0.70mm}{0mm}}
\pgfsetyvec{\pgfpoint{0mm}{0.70mm}}
\color[rgb]{0,0,0}\pgfsetlinewidth{0.30mm}\pgfsetdash{}{0mm}
\pgfcircle[fill]{\pgfxy(30.00,10.00)}{0.70mm}
\pgfcircle[stroke]{\pgfxy(30.00,10.00)}{0.70mm}
\pgfcircle[fill]{\pgfxy(30.00,15.00)}{0.70mm}
\pgfcircle[stroke]{\pgfxy(30.00,15.00)}{0.70mm}
\pgfcircle[fill]{\pgfxy(30.00,20.00)}{0.70mm}
\pgfcircle[stroke]{\pgfxy(30.00,20.00)}{0.70mm}
\pgfcircle[fill]{\pgfxy(30.00,25.00)}{0.70mm}
\pgfcircle[stroke]{\pgfxy(30.00,25.00)}{0.70mm}
\pgfmoveto{\pgfxy(30.00,15.00)}\pgflineto{\pgfxy(30.00,10.00)}\pgfstroke
\pgfputat{\pgfxy(28.00,14.00)}{\pgfbox[bottom,left]{\fontsize{7.97}{9.56}\selectfont \makebox[0pt][r]{2}}}
\pgfputat{\pgfxy(28.00,9.00)}{\pgfbox[bottom,left]{\fontsize{7.97}{9.56}\selectfont \makebox[0pt][r]{1}}}
\pgfputat{\pgfxy(28.00,18.50)}{\pgfbox[bottom,left]{\fontsize{7.97}{9.56}\selectfont \makebox[0pt][r]{3}}}
\pgfputat{\pgfxy(27.84,23.48)}{\pgfbox[bottom,left]{\fontsize{7.97}{9.56}\selectfont \makebox[0pt][r]{4}}}
\pgfmoveto{\pgfxy(30.00,20.00)}\pgflineto{\pgfxy(30.00,15.00)}\pgfstroke
\pgfmoveto{\pgfxy(30.00,25.00)}\pgflineto{\pgfxy(30.00,20.00)}\pgfstroke
\pgfsetlinewidth{0.15mm}\pgfmoveto{\pgfxy(25.16,17.59)}\pgflineto{\pgfxy(28.42,17.59)}\pgflineto{\pgfxy(28.42,22.04)}\pgflineto{\pgfxy(25.16,22.04)}\pgfclosepath\pgfstroke
\end{pgfpicture}%
&
$
\centering
\begin{pgfpicture}{2.20mm}{3.70mm}{15.68mm}{17.50mm}
\pgfsetxvec{\pgfpoint{0.70mm}{0mm}}
\pgfsetyvec{\pgfpoint{0mm}{0.70mm}}
\color[rgb]{0,0,0}\pgfsetlinewidth{0.30mm}\pgfsetdash{}{0mm}
\pgfcircle[fill]{\pgfxy(10.00,10.00)}{0.70mm}
\pgfcircle[stroke]{\pgfxy(10.00,10.00)}{0.70mm}
\pgfcircle[fill]{\pgfxy(10.00,15.00)}{0.70mm}
\pgfcircle[stroke]{\pgfxy(10.00,15.00)}{0.70mm}
\pgfcircle[fill]{\pgfxy(15.00,15.00)}{0.70mm}
\pgfcircle[stroke]{\pgfxy(15.00,15.00)}{0.70mm}
\pgfcircle[fill]{\pgfxy(15.00,20.00)}{0.70mm}
\pgfcircle[stroke]{\pgfxy(15.00,20.00)}{0.70mm}
\pgfmoveto{\pgfxy(10.00,15.00)}\pgflineto{\pgfxy(10.00,10.00)}\pgfstroke
\pgfmoveto{\pgfxy(10.00,10.00)}\pgflineto{\pgfxy(15.00,15.00)}\pgfstroke
\pgfmoveto{\pgfxy(15.00,20.00)}\pgflineto{\pgfxy(15.00,15.00)}\pgfstroke
\pgfputat{\pgfxy(8.00,14.00)}{\pgfbox[bottom,left]{\fontsize{7.97}{9.56}\selectfont \makebox[0pt][r]{2}}}
\pgfputat{\pgfxy(8.00,9.00)}{\pgfbox[bottom,left]{\fontsize{7.97}{9.56}\selectfont \makebox[0pt][r]{1}}}
\pgfputat{\pgfxy(17.00,14.00)}{\pgfbox[bottom,left]{\fontsize{7.97}{9.56}\selectfont 3}}
\pgfputat{\pgfxy(17.00,19.00)}{\pgfbox[bottom,left]{\fontsize{7.97}{9.56}\selectfont 4}}
\pgfsetlinewidth{0.15mm}\pgfmoveto{\pgfxy(16.49,13.22)}\pgflineto{\pgfxy(19.54,13.22)}\pgflineto{\pgfxy(19.54,17.14)}\pgflineto{\pgfxy(16.49,17.14)}\pgfclosepath\pgfstroke
\end{pgfpicture}%
$  \\
 \hline
(5)&
$\sigma \in \A_{4,k}'$ & $\underbrace{4 \, 1}_{\boxed{3}} 3 \, 2$ & $\underbrace{4\, 2}_{\boxed{2}} 3 \, 1$ & $\underbrace{3\, 1}_{\boxed{2}} 4 \, 2$ & $\underbrace{3\, 2}_{\boxed{1}} 4 \, 1$ & $\underbrace{2 \, 1}_{\boxed{1}} 4 \, 3$ \\
\hline
(6)&
$\sigma' \in \mathcal{MM}_{4,k}$ & $\underbrace{1 \, 4}_{\boxed{3}} 2 \, 3$ & $\underbrace{1\, 3}_{\boxed{2}} 2 \, 4$ & $\underbrace{3\, 1}_{\boxed{2}} 4 \, 2$ & $\underbrace{2\, 3}_{\boxed{1}} 1 \, 4$ & $\underbrace{2 \, 1}_{\boxed{1}} 4 \, 3$ \\
\hline
(7)&
$\tau_1 \in \mathcal{MM}_{4,k}'$ &  $\boxed{2} 1 \, 4 \, 3$  & $2 \boxed{3} 1 \, 4$ & $\boxed{3} 1 \, 4 \, 2$ & $\boxed{(4)}1\,3\,2\,4$ & $\boxed{(4)} 1\,4\,2\,3$ \\
\hline
(8)&
$\tau_2 \in \mathcal{MM}_{4,k}''$ &  $2\, 1 \, 4 \boxed{3}$ & $1 \, 4 \boxed{2} 3$ & $3 \, 1 \, 4 \boxed{2}$  & $1 \, 3 \, 2 \, 4\boxed{(1)}$ & $2\, 3 \, 2 \, 4\boxed{(1)}$  \\
\hline
(9)&
$\pi' \in \mathcal{GW}_{4,k}$ & $6\boxed{3} 4 \, 2 \, 1 \, 5$ & $6 \, 4 \boxed{2} 3 \, 1 \, 5$ & $6 \boxed{2} 3 \, 4 \, 1 \, 5$ & $6 \, 4 \, 3 \, 2 \boxed{1} 5$  & $6 \, 2 \, 4 \, 3 \boxed{1} 5$ \\
\hline
(10)&
$\pi'' \in \mathcal{RW}_{4,k}$ & $6 \, 2 \, 1 \, 4 \boxed{3} 5$ & $6 \boxed{2} 3 \, 1 \, 4 \, 5$ & $6 \, 1 \, 4 \boxed{2} 3 \, 5$ & $6 \, 3 \boxed{1} 2 \, 4\, 5$  & $6 \, 2 \, 4 \boxed{1} 3 \, 5$ \\
\hline
(11)&
$u \in \mathcal{UW}_{4,k}$ & $1 \, 1 \, 1  \boxed{3}$ & $1 \, 1 \, 2 \boxed{2}$ & $1 \, 1 \, 1 \boxed{2}$ & $1 \, 1 \, 2 \boxed{1}$  & $1 \, 1 \, 1 \boxed{1} $ \\
\hline
(12)&
$v \in \mathcal{UW}'_{4,k}$ & $1 \, 1 \underbrace{1 \, 1}_{\boxed{2}}$ & $1 \, 1 \underbrace{2 \, 1}_{\boxed{3}}$ & $1 \, 1 \underbrace{1 \, 2}_{\boxed{3}}$ & $1 \, 1 \underbrace{2 \, 2}_{\boxed{4}}$  & $1 \, 1 \underbrace{1 \, 3}_{\boxed{4}}$  \\
\hline
\end{tabular}
\caption{\label{fig:E4}Twelve interpretations for $E_{4,k}, 1 \leq k \leq 4$}
\end{figure}

\subsection{An open problem}
Consider the  so-called \textit{reduced tangent numbers} $t_n=E_{2n+1}/2^n$.
Poupard~\cite{Pou89}  proved  that $t_n$ is the number of 0-2 increasing trees (i.e., the trees in $\Tn$ such that every vertex has 0 or 2 children). However,
 it seems that  there is no interpretation \`a  la Andr\'e for $t_{n}$  in down-up  permutations.
Furthermore, let $t_{n,k}$ denote the number of 0-2 increasing trees such that the leaf of the minimal path is $k$, then  the sequence $(t_{n,k})$
is obviously a refinement of $t_{n}$ as Entringer numbers are for Euler numbers.

Let  $s_n$ (resp. $s_{n,k}$) be the number of \emph{split-pair arrangements} of $[n]$, that are arrangements  $\sigma$ of the multi-set $\{0,0,1,1,2,2,\ldots,n,n\}$
such that $\sigma(1)=n$ (resp. $\sigma(1)=\sigma(k+1)=n$)  and,  between the two occurrences
of $i$ in $\sigma$ ($0 \leq i \leq n-1$), the number $i+1$ appears exactly once.

 Recently, Graham and Zang~\cite{GZ08} proved that
 for $1 \leq k \leq n$, $s_{n,k} = t_{n,k}$. In particular,  $s_n = t_n$.
 There is no bijective proof  between Poupard's model and Graham and Zang's model.

\section*{Acknowledgement} 
We thank the two referees for their careful readings and helpful comments on a previous version of this paper.
This work was partially supported by the French National Research Agency under the grant ANR-08-BLAN-0243-03.

\nocite{*}
\bibliographystyle{plain}

\begin{thebibliography}{11}


\bibitem[And79]{And79}
D.~Andr\'{e}, \emph{D\'{e}veloppement de sec x et tan x}, C. R. Math. Acad. Sci. Paris 88 (1879), 965--979.

\bibitem[Cal05]{Cal05}
D.~Callan, \emph{A note on downup permutations and increasing 0-1-2 trees}, preprint (2009).


\bibitem[Che08]{Che08}
D.~Chebikin, \emph{Variations on descents and inversions in permutations},
  Electron. J. Combin. \textbf{15} (2008), no.~1, Research Paper 132, 34.


\bibitem[Don75]{Don75}
R.~Donaghey, \emph{Alternating  permutations and binary increasing trees}, J.
  Combinatorial Theory Ser. A, 18 (1975), 141--148.

\bibitem[Ent66]{Ent66}
R.C.~Entringer, \emph{A combinatorial interpretation of the Euler and Bernoulli numbers}, Nieuw Arch. Wisk. 14 (1966), 241--246.

\bibitem[FS73]{FS73}
D.~Foata and M.-P. Sch{\"u}tzenberger, \emph{Nombres d'{E}uler et permutations
  alternantes}, A survey of combinatorial theory [J.N. Srivastava et al., eds.],
   Amsterdam, North-Holland (1973), 173--187.

\bibitem[GZ08]{GZ08}
R.~Graham and N.~Zang, \emph{Enumerating split-pair arrangements}, J. Combin.
  Theory Ser. A 115 (2008), no.2, 293--303.

  \bibitem[JV10]{JV09}
M.~Josuat-Verg\`{e}s, \emph{A $q$-enumeration of alternating  permutations},
  European J. Combin. (2010), no.~doi:10.1016/j.ecj.2010.01.008.

\bibitem[Kem33]{Kem33}
A.~J.~Kempner, \emph{On the shape of polynomial curves}, T\^{o}hoku Math. Journal 37 (1933) 347--362.

\bibitem[KPP94]{KPP94}
A.~G.~Kuznetsov, I.~M.~Pak, and A.~E.~Postnikov, \emph{Increasing trees and
  alternating  permutations}, Uspekhi Mat. Nauk 49 (1994), 79--110.

\bibitem[Mar06]{Mar06}
J.~L.~Martin, \emph{The slopes determined by $n$ points in the plane}, Duke Math. J. 131 (2006), no.1, p.119--165

\bibitem[MW09]{MW09}
J.~L.~Martin and J.~D.~Wagner, \emph{Updown numbers and the initial monomials of the slope variety}, Electron. J. Combin. 16 (2009), no.1, Research Paper 82, 8pp.

\bibitem[MSY96]{MSY96}
J.~Millar, N.~J.~A.~Sloane and N.~E.~Young, \emph{A new operation on sequences: the Boustrouphedon transform}, J.Combinatorial Theory, Series A 76(1):44--54 (1996).


\bibitem[Pou82]{Pou82}
C.~Poupard, \emph{De nouvelles significations \'enum\'eratives des nombres
  d'{E}ntringer}, Discrete Math. 38 (1982), 265--271.

\bibitem[Pou89]{Pou89}
C.~Poupard, \emph{Deux propri\'{e}t\'{e}s des arbres binaires ordonn\'{e}s stricts}, European J. Combin. 10 (1989), 369--374.

\bibitem[Pou97]{Pou97}
C.~Poupard, \emph{Two other interpretations of the Entringer numbers}, European Journal of Combinatorics 18 (1997), 939--943.

\bibitem[Sei77]{Sei77}
L.~Seidel, \emph{\"{UW}ber eine einfache Entstehungsweise der Bernoullischen Zahlen und einiger verwandten
Reihen}, Sitzungsber. M�nch. Akad. 4 (1877), 157--187.

\bibitem[SZ10]{SZ10}
H. Shin and J. Zeng, The q-tangent and q-secant numbers via continued fractions, arXiv:0911.4658, to appear in European Journal of Combinatorics, 2010.

\bibitem[Sta09]{Sta09}
R.~Stanley, \emph{A Survey of  Alternating  Permutations}, \texttt{arXiv:0912.4240} (2009).


\end{thebibliography}

\end{document}